\documentclass[a4paper]{article}

\usepackage[latin1]{inputenc}
\usepackage[french,british, german]{babel}
\usepackage[T1]{fontenc}
\usepackage{setspace}
\usepackage{indentfirst} 
\usepackage{fancyhdr} 
\usepackage{nicefrac}
\usepackage{textcomp}
\usepackage{lastpage}
\usepackage{amssymb,amsfonts,amsmath,amsthm}
\usepackage{mathtools}
\usepackage{latexsym}
\usepackage{hyperref}
\usepackage{url}
\usepackage{verbatim}
\usepackage[all]{xy}
\usepackage{mathrsfs}
\usepackage{stmaryrd}
\usepackage{oldgerm}
\usepackage{bm}
\usepackage{graphics}
\usepackage{wrapfig}
\usepackage{fancybox}
\usepackage{xcolor}
\usepackage{graphicx}
\usepackage{dutchcal}
\input xy
\xyoption{all}
\usepackage{tikz}
\usetikzlibrary{arrows,matrix}

\makeatletter
\def\blfootnote{\xdef\@thefnmark{}\@footnotetext}
\makeatother

\theoremstyle{plain}
\newtheorem{theorem}{Theorem}[section]
\newtheorem*{theorem*}{Theorem}
\newtheorem{lemma}[theorem]{Lemma}
\newtheorem*{lemma*}{Lemma}
\newtheorem{corollary}[theorem]{Corollary}
\newtheorem{proposition}[theorem]{Proposition}
\newtheorem*{proposition*}{Proposition}

\newtheorem*{conjecture*}{Conjecture}

\theoremstyle{remark}
\newtheorem{remark}[theorem]{Remark}
\newtheorem{assumption}[theorem]{Assumption}

\theoremstyle{definition}
\newtheorem{definition}[theorem]{Definition}

\newtheorem*{acknowledgements}{Acknowledgements}

\numberwithin{equation}{section}
\renewcommand\labelenumi{(\roman{enumi})}
\renewcommand\theenumi\labelenumi

\newcommand{\sbt}{\,\begin{picture}(-1,1)(-1,-2)\circle*{2}\end{picture}\ }

\newcommand{\Q}{\mathbb{Q}}
\newcommand{\Z}{\mathbb{Z}}
\newcommand{\N}{\mathbb{N}}
\newcommand{\R}{{\mathrm R}}

\newcommand{\lsk}{\mathsf{LS}_{k^0}}
\newcommand{\lss}{\mathsf{LS}_S}
\newcommand{\lscs}{\mathsf{LS}_{\mathcal{S}}}
\newcommand{\lst}{\mathsf{LS}_{T}}

\newcommand{\lsbs}{\overline{\textsf{LS}}_S}
\newcommand{\lsbcs}{\overline{\textsf{LS}}_{\cs}}
\newcommand{\lsbt}{\overline{\textsf{LS}}_T}
\newcommand{\lsbct}{\overline{\textsf{LS}}_{\ct}}

\newcommand{\lelsb}{\mathcal{T}\text{-}\overline{\textsf{LELS}}_T}

\newcommand{\ssk}{\overline{\textsf{LS}}^{\textrm{ss}}_{k^0}}
\newcommand{\sst}{\overline{\textsf{LS}}^{\textrm{ss}}_{\mathcal{T}}}

\newcommand{\ovk}{\overline{K} }
\newcommand{\ovv}{\overline{V} }
\newcommand{\ovx}{\overline{X}}
\newcommand{\ovy}{\overline{Y}}
\newcommand{\ovz}{\overline{Z}}
\newcommand{\ovcz}{\overline{\mathcal{Z}}}
\newcommand{\ovcy}{\overline{\mathcal{Y}}}
\newcommand{\ovtimes}{\overline{\times}}

\newcommand{\ovm}{\overline{M}}
\newcommand{\ovp}{\overline{P}}

\newcommand{\ovt}{{\overline{T}}}
\newcommand{\ovct}{{\overline{\mathcal{T}}}}

\newcommand{\an}{\mathrm{an} }
\newcommand{\gp}{\mathrm{gp} }

\newcommand{\rig}{\mathrm{rig} } 
\newcommand{\dr}{\mathrm{dR} }
\newcommand{\conv}{\mathrm{conv} }
\newcommand{\eet}{\mathrm{\acute{e}t}} 
\newcommand{\h}{\mathrm{h}}
\newcommand{\Zar}{\mathrm{Zar} }
\newcommand{\syn}{ \mathrm{syn} }
\newcommand{\hk}{\mathrm{HK} }

\newcommand{\D}{\mathrm{D}}
\newcommand{\NN}{\mathrm{NN}}
\newcommand{\cris}{\mathrm{cr} }

\newcommand{\nr}{\mathrm{nr}}

\newcommand{\Gd}{\operatorname{Gd}}

\newcommand{\Gal}{\operatorname{Gal} }
\newcommand{\can}{ \operatorname{can} }
\newcommand{\id}{ \operatorname{id} }

\newcommand{\holim}{\operatorname{holim} }
\newcommand{\Cone}{\operatorname{Cone} }
\newcommand{\kker}{\operatorname{Ker} }
\newcommand{\coker}{\operatorname{Coker} }
\newcommand{\im}{\operatorname{Im} }
\newcommand{\fil}{\operatorname{Fil} }

\newcommand{\Hom}{\operatorname{Hom} }

\newcommand{\Sym}{\operatorname{Sym}}
\newcommand{\Spec}{\operatorname{Spec} }
\newcommand{\Proj}{\operatorname{Proj}}

\newcommand{\Spf}{\operatorname{Spf} }
\newcommand{\Spwf}{\operatorname{Spwf}}
\newcommand{\Gr}{\operatorname{Gr}}

\newcommand{\rqhk}{\textsf{RQ}_{\mathrm{HK}}}

\newcommand{\RTOF}{\mathsf{RT}_{\mathscr{O}_F^0}}

\newcommand{\cs}{{\mathcal{S}}} 
\newcommand{\ct}{{\mathcal{T}}}
\newcommand{\cd}{{\mathcal{D}}}
\newcommand{\co}{{\mathcal{O}}}
\newcommand{\cx}{{\mathcal{X}}}
\newcommand{\cy}{{\mathcal{Y}}}
\newcommand{\cz}{{\mathcal{Z}}}
\newcommand{\cm}{{\mathcal{M}}}
\newcommand{\cn}{{\mathcal{N}}}

\newcommand{\so}{{\mathscr{O}}}

\begin{document}

\selectlanguage{british}

\thispagestyle{empty}

\title{Comparison between rigid syntomic and crystalline syntomic cohomology for strictly semistable log schemes with boundary}
\author{Veronika Ertl and Kazuki Yamada}
\date{}
\maketitle 
 
\begin{abstract}
{\noindent
We introduce rigid syntomic cohomology for strictly semistable log schemes over a complete discrete valuation ring of mixed characteristic $(0,p)$. 
In case a good compactification exists, we compare this  cohomology theory  to Nekov\'a\v{r}--Nizio\l{}'s crystalline syntomic cohomology of the generic fibre.
The main ingredients are a modification of Gro\ss{}e-Kl\"{o}nne's rigid Hyodo--Kato theory and a generalization of it for strictly semistable log schemes with boundary.}
\end{abstract}
\vspace{-9pt}
\selectlanguage{french}
\begin{abstract}
{\noindent
On introduit la cohomologie syntomique rigide pour les  sch\'emas logarithmique de r\'eduction strictement semistable sur un anneau de valuation discr\`ete de caract\'eristique $(0,p)$. 
Dans le cas de bonne compactification, on compare cette th\'eorie de cohomologie \`a la cohomologie syntomique cristalline de Nekov\'a\v{r}--Nizio\l{} sur la fibre g\'en\'erique.
La cl\'e est une modification de la th\'eorie  Hyodo--Kato rigide de Gro\ss{}e-Kl\"onne et une g\'en\'eralisation de celle-ci aux sch\'emas logarithmiques de r\'eduction semistable avec bord.\\

\noindent
\textit{Key Words}: Syntomic cohomology, rigid cohomology, semistable reduction.\\
\textit{Mathematics Subject Classification 2000}:   14F30, 14G20, 14F42}

\end{abstract}
\selectlanguage{british}

\blfootnote{The authors' research was supported in part by KAKENHI 26247004, 18H05233, as well as the JSPS Core-to-core program ``Foundation of a Global Research Cooperative Center in Mathematics focused on Number Theory and Geometry''. The first named author was supported by the Alexander~von~Humboldt-Stiftung and the Japan Society for the Promotion of Science as a JSPS International Research Fellow. The second named author was supported by the grant KAKENHI16J01911.}

\tableofcontents

\addcontentsline{toc}{section}{Introduction}
\section*{Introduction}

In this article we construct rigid syntomic cohomology for strictly semistable log schemes, and compare it with  crystalline syntomic cohomology in the case that there exists a good compactification.

Let $K$ be a $p$-adic field and $\so_K$ its ring of integers.
In general terms, syntomic cohomology $ \R\Gamma_\syn(X,r)$ is defined for semistable (log) schemes $X$ over $\so_K$ and integers $r\in\mathbb{Z}$, as a $p$-adic analogue of the Deligne--Beilinson cohomology.
An important application of syntomic cohomology is, at least conjecturally, the $p$-adic Beilinson conjecture.
It states that the special values of $p$-adic $L$-functions are described by the rational part of the special values of $L$-functions and an arithmetic invariant defined in terms of the $p$-adic regulator.
This conjecture was formulated by Perrin-Riou \cite{PerrinRiou1995} in the case that $X$ is smooth over $\so_K$.
While there is still no precise formulation of such a general conjecture in the semistable case, we may expect a similar picture.

There are two constructions of syntomic cohomology.
One uses  (log) crystalline cohomology, and the other one uses (log) rigid cohomology of the special fibre of $X$.
We call them crystalline syntomic cohomology and rigid syntomic cohomology, respectively.
Crystalline syntomic cohomology was originally introduced by Fontaine and Messing \cite{FontaineMessing1987}, and generalized by Kato \cite{Kato1994}, Nekov\'{a}\v{r} and Nizio{\l} \cite{NekovarNiziol2016} to the logarithmic case.
D\'{e}glise and Nizio{\l} \cite{DegliseNiziol2018} proved that Nekov\'{a}\v{r}--Nizio\l's crystalline syntomic cohomology can be regarded as absolute $p$-adic Hodge cohomology if the twist $r$ is non-negative.
Rigid syntomic cohomology was constructed by Besser \cite{Besser2000} for smooth schemes over $\so_K$, and further developed by Bannai \cite{Bannai2002}, as well as Chiarellotto, Ciccioni, and Mazzari \cite{ChiarellottoCiccioniMazzari2013} as an absolute $p$-adic Hodge cohomology.

An advantage of rigid syntomic cohomology is that it is purely $p$-adically analytic.
Thus it is useful for computations of $p$-adic regulators, and should relate directly with $p$-adic $L$-functions.
Indeed, there are several results concerning the  $p$-adic Beilinson conjecture  which use rigid syntomic cohomology \cite{BannaiKings2011, BertoliniDarmon2014, Besser2009, Niklas2010}.

A disadvantage of rigid syntomic cohomology is that the theory of log rigid cohomology often has technical difficulties, since it depends a priori on the choice of local liftings. We address this point in this paper and even construct  canonical log rigid complexes analogous to Besser's canonical rigid complexes introduced in \cite{Besser2000} which follows a suggestion of  Berthelot.

Moreover, the rigid Hyodo--Kato map depends on the choice of a uniformizer of $\so_K$, unlike Beilinson's crystalline Hyodo--Kato map, which was used by Nekov\'{a}\v{r} and Nizio\l.
Hence their crystalline syntomic cohomology does not depend on such a choice, and moreover extends to a very sophisticated theory for any varieties over $K$.
We remark that the constructions of the crystalline and the rigid Hyodo--Kato map are based on very different techniques.
Hence their comparison  does not automatically follow from the construction, whereas the comparison of Frobenius and monodromy operator are more straight forward.

\subsection*{Overview of the paper}
 
In Section \ref{sec: Log rigid complexes} we construct the rigid Hyodo--Kato complex for strictly semistable log schemes with boundary and log rigid complexes for fine log schemes and log schemes with boundary. 
In particular, we construct canonical complexes (not only in the  derived category) for embeddable objects, clarify the definition of these complexes for simplicial objects and prove important functoriality properties. 
The notion of log schemes with boundary was introduced in \cite{GrosseKloenne2004} to express compactifications in the sense of log geometry. 
Although the use of strictly semistable log schemes with boundary makes the construction more involved due to combinatorial difficulties, it pays of as it allows us to compare log rigid and log crystalline cohomology in the appropriate cases.

As mentioned above, in the construction of log rigid cohomology \cite{GrosseKloenne2007} there is usually a choice of local liftings involved. Thus in \S\ref{subsec:Prelim} and  \S\ref{subsec: Axiomatisation} we discuss some technical points which allow us to glue in a canonical way and obtain functoriality.
We use this  in \S\ref{Subsec: Log-rig-com} to construct several versions of canonical log rigid complexes -- 
first a more general version of  canonical log rigid complexes for fine log schemes and fine log schemes with boundary and 
subsequently a more subtle version in the case of strictly semistable log schemes with boundary   based on a construction by Kim and Hain \cite{KimHain2004}, which we call rigid Hyodo--Kato complex. 
The latter one allows an  explicit definition of Frobenius and monodromy operator. 
We finally compare these complexes and discuss base change properties.

In Section~\ref{sec: rigid structures} we look at the additional structure of the complexes constructed in Section 1 which is indispensable for the construction of syntomic cohomology.
Most importantly, we give in \S\ref{subsec: HK-map} a rigid analytic construction of the Hyodo--Kato morphism for log rigid cohomology. 
This is a generalization of Gro\ss{}e-Kl\"onne's construction in \cite{GrosseKloenne2007} to log schemes with boundary. 
We use a combinatorial modification of his construction to show the functoriality of the rigid Hyodo--Kato map.
In \S\ref{Subsec: Frobenius monodromy} we show that the Frobenius on the rigid Hyodo--Kato complex induced by local lifts and the one on the log rigid complexes induced by base change are compatible. 

Section~\ref{section: syntomic cohomology} focuses on syntomic cohomology. 
We review crystalline syntomic cohomology in \S\ref{subsec: syntomic for K-varieties} and include some basic constructions on crystalline cohomology.
In \S\ref{subsec: syntomic def} we finally give a definition of rigid syntomic cohomology for strictly semistable log schemes including a cup product on  cohomology groups.

Section~\ref{section: comparison} is reserved for the comparison of crystalline syntomic and rigid syntomic cohomology in the compactifiable case. 
An essential point is the comparison between log rigid and log crystalline cohomology which is carried out in \S\ref{Subsec:comparison rig - cris}.
The comparison of Frobenius, monodromy and Hyodo--Kato morphisms in \S\ref{subsec: comparison structures} finally implies the compatibility of the rigid syntomic and the crystalline syntomic cohomology.
It follows immediately from the constructions that the cup products on both are compatible.

\subsection*{Notation and conventions}

Let $\so_K$ be a complete discrete valuation ring of mixed characteristic $(0,p)$ with fraction field $K$, perfect residue field $k$, and maximal ideal $\mathfrak{m}_K$. Let $\pi$ be a uniformizer of $\so_K$.
As usual, denote by $\ovk$ an algebraic closure of $K$ and by $\so_{\ovk}$ the integral closure of $\so_K$ in $\ovk$. 
Let $\so_F = W(k)$ be the ring of Witt vectors of $k$, $F$ the fraction field of $\so_F$, and $F^{\nr}$ its maximal unramified extension.
Let $e$ be the absolute ramification index of $K$ and denote by $G_K = \Gal(\ovk/K)$ the absolute Galois group of $K$.
Let $\sigma$ be the canonical Frobenius on $\so_F$.

Since our construction is based on Gro\ss{}e-Kl\"onne's work, weak formal log schemes and dagger spaces will play a crucial part. 
For an $\so_F$-algebra $A$, we denote by $\widehat{A}$ the $p$-adic completion of $A$.
The weak completion of $A$ is the $\so_F$-subalgebra $A^\dagger$ of $\widehat{A}$ consisting of power series $\sum_{(i_1,\ldots,i_n)\in\N^n}a_{i_1,\ldots,i_n}x_1^{i_1}\cdots x_n^{i_n}$ for $x_1,\ldots,x_n\in A$, where $a_I\in\so_F$, such that there exists a constant $c>0$ which satisfies the Monsky--Washnitzer condition
	\[c(\mathrm{ord}_p(a_{i_1,\ldots,i_n}) + 1)\geqslant i_1 + \cdots + i_n\]
for any $(i_1,\ldots,i_n)\in\N^n$.
For further details, we refer to \cite{MonskyWashnitzer1968} and \cite{vanderPut1986} for weakly complete algebras, \cite{Meredith1972} for weak formal schemes, \cite{GrosseKloenne2000} and \cite{LangerMuralidharan2013} for dagger spaces.
Recall that a weak formal scheme over $\so_F$ is said to be admissible, if it is flat over $\so_F$.

We will use a shorthand for certain homotopy limits. Namely,  if $f\colon C\to C'$ is a map  in the derived dg-category of abelian groups, we set
	$$
	[\xymatrix{C \ar[r]^f&C'}]: = \holim(C\to C^{\prime}\leftarrow 0).
	$$
If $f$ is represented by a morphism of complexes, this can be seen as a mapping cone as
	\[ [\xymatrix{C \ar[r]^f&C'}] \cong\Cone(f)[-1].\]
And we set
	$$
	\left[\begin{aligned}
	\xymatrix{C_1 \ar[d]^\beta  \ar[r]^\alpha & C_2 \ar[d]^\delta \\
	C_3 \ar[r]^\gamma & C_4
	}\end{aligned}\right]
	: = [[C_1 \xrightarrow{\alpha} C_2] \xrightarrow{(\beta,\delta)} [C_3 \xrightarrow{\gamma} C_4]],
	$$ 
for  a commutative diagram (the one inside the large bracket) in the derived dg-category of abelian groups.
	
For an abelian category $\mathcal{M}$, we denote by $\mathscr{C}^ + (\mathcal{M})$ the category of bounded below complexes in $\mathcal{M}$, and $\mathscr{D}^ + (\mathcal{M})$ its derived category.

We assume all schemes to be of finite type. 
For a scheme $X/\so_K$ denote by $X_n$ for $n\in\N$ its reduction modulo $p^n$ and let $X_0$ be its special fibre, whereas $X_K$ is the generic fibre.

We will use log structures extensively.
We consider all log structures on log schemes as  given by a sheaf of monoids with respect to the Zariski topology.
Thus if we say that a log scheme $X$ is fine, it means that Zariski locally $X$ has a chart given by a fine monoid.
Note that giving a fine log scheme in our context is equivalent to giving a fine log scheme in the  usual sense which is of Zariski type \cite[Cor.~1.1.11]{Shiho2002}.
For a (weak formal) log scheme $X$, we often denote its underlying (weak formal) scheme by $X$ again, and its log structure by $\cn_X$.

Let $\widehat{\mathcal{T}}$ (resp. $\mathcal{T}$) be the formal (resp. weak formal) log scheme $\Spf\widehat{\so_F[t]}$ (resp. $\Spwf\so_F[t]^\dagger$)  with the log structure associated to the map 
$1 \mapsto  t$, and let $T$ be the reduction of $\mathcal{T}$ mod $p$ with the induced log structure. 
We denote by $\so_F^0$ (resp. $k^0$) the exact closed weak formal log subscheme of $\mathcal{T}$ (resp. $T$) defined by $t = 0$, and by $\so_K^\pi$ the exact closed weak formal log subscheme of $\ct$ defined by $t = \pi$.
Unless otherwise stated, we use on $\widehat{\mathcal{T}}$,  $\mathcal{T}$, $T$, $\so_F^0$, $\so_K^\pi$, $k^0$ the charts $c_\ct$, $c_T$, $c_{\so_F^0}$, $c_{\so_K^\pi}$, $c_{k^0}$ induced by the above map.
Moreover, we write $k^\varnothing, \so_F^\varnothing, \so_K^\varnothing,\ldots$ if we consider the trivial log-structure.
Note that the weak formal log scheme $\so_K^\pi$ is independent of the choice of $\pi$, but the exact closed immersion $\so_K^\pi\hookrightarrow\ct$ and the chart $c_\pi$ depend on $\pi$.
We extend the canonical Frobenius $\sigma$ on $\so_F$ to $\so_F[t]$ by sending $t$ to $t^p$. This induces a unique Frobenius on $\ct$ which we denote by abuse of notation again by $\sigma$. 

For a weak formal log scheme $\mathcal{X}$ over $\so_F$, we denote by $\widehat{\mathcal{X}}$ the $p$-adic completion of $\mathcal{X}$, which is a formal log scheme over $\so_F$.
For a weak formal scheme $\mathcal{X}$ over $\so_F$, we denote by $\mathcal{X}_\Q$ the generic fiber of $\mathcal{X}$, which is a dagger space over $F$.
For any locally closed subset $\mathcal{U}\subset\mathcal{X}$, we denote by $]\mathcal{U}[_{\mathcal{X}}$ the inverse image of the specialization map $\mathcal{X}_\Q\rightarrow\mathcal{X}$.
This is an admissible open subset of $\mathcal{X}_\Q$.

\begin{acknowledgements}
We thank the referee for a careful reading of the manuscript and helpful comments.

Our work on this article started during a visit of the second author to the University of Regensburg and continued throughout the first author's visit to Keio University. We would like to thank these institutions for their support and hospitality. 

We would like to thank Kenichi Bannai and all members of the KiPAS-AGNT group for many helpful suggestions and for providing a pleasant working atmosphere where we enjoyed productive discussions. 

It is a pleasure to thank Wies\l{}awa Nizio\l{}, Atsushi Shiho, Go Yamashita and Seidai Yasuda for stimulating discussions and helpful comments related to the topic of this article.
\end{acknowledgements}

\section{Logarithmic rigid complexes}\label{sec: Log rigid complexes}

Gro\ss{}e-Kl\"onne showed that (non-logarithmic) rigid cohomology can be computed using certain spaces with overconvergent structure sheaves.
He showed that this construction can also be carried out for log schemes. 
This not only simplifies the construction of rigid cohomology, but is essential in \cite{GrosseKloenne2005} to construct the rigid Hyodo--Kato morphism.

In this section we will recall several different versions of log rigid complexes introduced by Gro\ss{}e-Kl\"onne in \cite{GrosseKloenne2005}. 
We use his insight to construct  {\it canonical}  log rigid complexes similar to the rigid complexes Besser considers in the non-logarithmic situation \cite{Besser2000}.

In the first subsection, we start with some technical preliminaries which will be needed in the subsequent sections to glue local constructions, and to prove functoriality even for the simplicial case.
In \S\ref{subsec: Axiomatisation} we describe Besser's approach in a category theoretical manner.
This allows us in \S\ref{Subsec: Log-rig-com} to construct several versions of canonical log rigid complexes. 
We finish this section by  comparing these complexes and discuss base change properties.

\subsection{Preliminaries}\label{subsec:Prelim}

We start with  some topological definitions and facts from \cite{Beilinson2012}.
This will be used to glue rigid complexes which are constructed locally.

\begin{definition}[{\cite[\S 2.1]{Beilinson2012}}]\label{def: base}
	Let $\mathcal{V}$ be an essentially small site. A {\it base} for $\mathcal{V}$ is a pair $(\mathcal{B},\theta)$ of an essentially small category $\mathcal{B}$ and a faithful functor $\theta\colon\mathcal{B} \rightarrow \mathcal{V}$ which satisfy the following property.
	For any $V\in\mathcal{V}$ and a finite family of pairs $(B_\alpha,f_\alpha)$ of $B_\alpha\in\mathcal{B}$ and $f_\alpha\colon V \rightarrow \theta(B_\alpha)$, there exists a set of objects $B'_\beta\in\mathcal{B}$ and a covering family $\{\theta(B'_\beta) \rightarrow  V\}$ such that every composition $\theta(B'_\beta) \rightarrow  V \rightarrow \theta(B_\alpha)$ lies in $\Hom(B'_\beta,B_\alpha)\subset\Hom(\theta(B'_\beta),\theta(B_\alpha))$.
\end{definition}

We define a covering sieve in $\mathcal{B}$ as a sieve whose image by $\theta$ is a covering family in $\mathcal{V}$.

\begin{proposition}[{\cite[\S 2.1]{Beilinson2012}}]\label{prop: topos eq}
	Covering sieves in $\mathcal{B}$ form a Grothendieck topology on $\mathcal{B}$.
	Moreover the functor $\theta$ is continuous and induces an equivalence of the associated topoi.
\end{proposition}

The following notion will be used to define rigid complexes for simplicial objects.

\begin{definition}
Let $\mathcal{C}_{\sbt}$ be a simplicial site.
For any morphism $\alpha\colon[m] \rightarrow [n]$ in the simplex category, let $u_\alpha\colon\mathcal{C}_m \rightarrow \mathcal{C}_n$ be the functor which corresponds to the morphism of sites induced by $\alpha$.
The {\it total site} $\mathcal{C}_{\sbt}^{\mathrm{tot}}$ of $\mathcal{C}_{\sbt}$ is the site whose objects are pairs $(n,U)$ with $n\in\N$ and $U \in \mathcal{C}_n$, whose morphisms are pairs $(\alpha,f)\colon(n,U) \rightarrow (m,V)$ consisting of a morphism $\alpha\colon[m] \rightarrow [n]$ in the simplex category and a morphism $f\colon U \rightarrow  u_\alpha(V)$ in $\mathcal{C}_n$, and whose coverings are families $\{(\mathrm{id},f)\colon (n,U_i) \rightarrow  (n,U)\}$ such that $\{U_i \rightarrow  U\}$ is a covering family in $\mathcal{C}_n$.
\end{definition}

Next we recall and introduce some notions for log geometry.

\begin{definition}\label{def: lss}
	For a fine log scheme $S$, let $\lss$ be the category of fine log schemes over $S$.
	For a fine weak formal log scheme $\cs$, let $\lscs$ be the category of fine weak formal log schemes over $\cs$.
	We regard $\lss$ and $\lscs$ as sites for the Zariski topology.
	Namely a covering family of an object $Y$ in $\lss$ or $\lscs$ is a family of strict open immersions $U_h\hookrightarrow Y$ with $\bigcup_hU_h = Y$.
\end{definition}

When we say immersion in the current context, we mean the following definition.

\begin{definition}
	A morphism $f\colon X \rightarrow  Y$ of (weak formal) log schemes is an \textit{immersion} if $f$ can be factored as $f = i\circ j$ by a closed immersion $j\colon X\hookrightarrow Z$ and a strict open immersion $i\colon Z \rightarrow  Y$.
	This is equivalent to the condition that $f$ is an immersion on the underlying schemes and $f^\ast\cn_Y \rightarrow \cn_X$ is surjective.
	We say that an immersion $f$ as above is \textit{strict} if $f^\ast\cn_Y \rightarrow \cn_X$ is an isomorphism.
\end{definition}

\begin{definition}
	Let $f\colon Y\hookrightarrow Y'$ be an immersion of fine (weak formal) log schemes.
	An \textit{exactification} of $f$ is a factorization $f = h\circ j$ by an exact closed immersion $j\colon Y\hookrightarrow Y''$ and a log \'{e}tale morphism $h\colon Y'' \rightarrow  Y'$.
\end{definition}

Note that according to the proof of \cite[Prop.~4.10 (1)]{Kato1989} an exactification of $i$ exists if $i$ admits a chart. 
In general an exactification of an immersion is not unique.
However one can show that the tubular neighborhoods on exactifications are canonically isomorphic to each other.

Let $\mathcal{S}$ be a fine admissible weak formal log scheme over $\so_F^\varnothing$ and $S$ its reduction modulo $p$.

\begin{lemma}\label{lem: tube is canonical}
	Let $Y\hookrightarrow\mathcal{Y}$ be an immersion over $\mathcal{S}$ of a fine log scheme $Y$ over $S$ into a fine admissible weak formal log scheme $\mathcal{Y}$ over $\mathcal{S}$.
	Assume that there are exactifications $Y\hookrightarrow\mathcal{Y}_i \rightarrow \mathcal{Y}$ for $i = 1,2$.
	Then there is a canonical isomorphism $\iota\colon]Y[_{\cy_1} \cong ]Y[_{\cy_2}$.
	Moreover, if we let $\omega^{\sbt}_{\cy_i/\mathcal{S},\Q}$ for $i = 1,2$ be the complexes of sheaves on $\cy_{i,\Q}$ given by tensoring the log de~Rham complexes of $\cy_i$ over $\mathcal{S}$ with $\Q$, then there is a canonical isomorphism $\iota^\ast\omega^{\sbt}_{\cy_2/\mathcal{S},\Q}|_{]Y[_{\cy_2}} \cong\omega^{\sbt}_{\cy_1/\mathcal{S},\Q}|_{]Y[_{\cy_1}}$.
\end{lemma}

\begin{proof}
	The first statement follows exactly in the same way as \cite[Lem.~1.2]{GrosseKloenne2005} by taking an exactification $Y\hookrightarrow\cy' \rightarrow \cy_1\times_\cy\cy_2$ of the diagonal embedding.
	Since the projection $\cy' \rightarrow \cy_i$ is strict on neighborhoods of $Y$, the second statement follows, too.
\end{proof}

The above lemma implies that the tubes glue to give a canonical dagger space.
Note that we use  the existence of Zariski (not only \'{e}tale) local charts in the gluing argument.

\begin{definition}\label{def: log tube}
	Let $f\colon Y\hookrightarrow\cy$ be an immersion over $\mathcal{S}$ of a fine log scheme $Y$ over $S$ into a fine admissible weak formal log scheme $\cy$ over $\mathcal{S}$.
	We define the log tube $]Y[_\cy^{\log}$ of $i$ and a complex of sheaves $\omega^{\sbt}_{\cy/\mathcal{S},\Q}$ on $]Y[^{\log}_{\cy}$ as follows:
	\begin{enumerate}
	\item If $f$ admits a chart, there exists an exactification $Y\hookrightarrow\cy' \rightarrow \cy$.
		We set $]Y[_\cy^{\log}: = ]Y[_{\cy'}$ and let $\omega^{\sbt}_{\cy/\mathcal{S},\Q}$ be the complex of sheaves on $]Y[^{\log}_\cy$ given by tensoring the log de~Rham complex of $\cy'$ over $\mathcal{S}$ with $\Q$, which is independent of the choice of an exactification by Lemma~\ref{lem: tube is canonical}.
	\item In general, take an open covering $\{U_h\}_h$ of $Y$ and $\{\mathcal{U}_h\}_h$ of $\cy$, such that $f$ induces an immersion $U_h\hookrightarrow\mathcal{U}_h$ which admits a chart for each $h$.
		Then we can define $]U_h[_{\mathcal{U}_h}^{\log}$ and $\omega^{\sbt}_{\mathcal{U}_h/\mathcal{S},\Q}$ as above, and they naturally glue along $]U_h\cap U_{h'}[_{\mathcal{U}_h\times\mathcal{U}_{h'}}^{\log}$.
		The dagger space and the complex of sheaves obtained by gluing are independent of the choice of an open covering, and we denote them by $]Y[_\cy^{\log}$ and $\omega^{\sbt}_{\cy/\mathcal{S},\Q}$.
	\end{enumerate}
\end{definition}

We recall some basic definitions for log schemes with boundary.
For more details see \cite{GrosseKloenne2004}.

\begin{definition}[{\cite[Def.~1.5]{GrosseKloenne2004}, \cite[1.6]{GrosseKloenne2005}}]
	\begin{enumerate}
	\item A {\it log scheme with boundary} is a strict open immersion of log schemes $i\colon Z\hookrightarrow\ovz$ such that $\cn_{\ovz} \rightarrow  i_\ast\cn_Z$ is injective, $\cn_{\ovz}^\gp \rightarrow (i_\ast\cn_Z)^\gp$ is an isomorphism, and $Z$ is schematically dense in $\ovz$.
		To simplify notations we often write  $(Z,\ovz)$ for $i\colon Z\hookrightarrow\ovz$.
		A morphism $f\colon(Z',\ovz')\rightarrow(Z,\ovz)$ of log schemes with boundary is a morphism $f\colon\ovz'\rightarrow\ovz$ of log schemes such that the restriction of $f$ to $Z'$ factors through $Z$.
		Note that such a factorization is unique if it exists.
	\item An {\it $S$-log scheme with boundary} is a log scheme with boundary $(Z,\ovz)$ together with a morphism of log schemes $Z \rightarrow  S$.
		A morphism $f\colon(Z',\ovz')\rightarrow(Z,\ovz)$ of $S$-log schemes with boundary is a morphism of log schemes with boundary such that the restriction of $f$ to $Z'$ is a morphism over $S$.
	\item An $S$-log scheme with boundary $(Z,\ovz)$ is {\it fine}  if $Z$ and $\ovz$ are fine.
	\end{enumerate}
	We define the notions of weak formal log schemes with boundary, weak formal $\cs$-log schemes with boundary, and fine weak formal $\cs$-log schemes with boundary in a similar manner.
\end{definition}

\begin{definition}\label{def: chart of boundary}
	Let $(Z,\ovz)$ be a fine $S$-log scheme with boundary.
	A {\it chart} for $(Z,\ovz)$ is a triple $(\alpha,\beta,\gamma)$ consisting of a chart $\alpha\colon P_{\ovz} \rightarrow \cn_{\ovz}$ for $\ovz$, a chart $\beta\colon Q_S \rightarrow \cn_S$ for $S$, and a homomorphism $\gamma\colon Q \rightarrow  P^\gp$, such that $\alpha^\gp\circ\gamma_{\ovz}\colon Q_{\ovz} \rightarrow \cn_{\ovz}^\gp$ coincides with the composition $Q_{\ovz} \xrightarrow{\beta}j_ + f^{-1}\cn_S \rightarrow  j_ + \cn_Z \rightarrow \cn_{\ovz}^\gp$, where $f\colon Z \rightarrow  S$ is the structure morphism, $j_ + $ is the sheaf theoretic push forward, and the third map is given by \cite[Lem.~1.3]{GrosseKloenne2004}.
	
	If a chart $\beta$ for $S$ is fixed, we call the pair $(\alpha,\gamma)$ a chart for $(Z,\ovz)$ extending $\beta$.
\end{definition}
For the following, see also \cite[Def.~2.1]{GrosseKloenne2004}.
\begin{definition}
	Let $f\colon (Z,\ovz) \rightarrow (Z',\ovz')$ be a morphism of $S$-log schemes with boundary.
	\begin{enumerate}
	\item We say that $f$ is a {\it boundary closed immersion} (resp. {\it boundary exact closed immersion)} if $\ovz \rightarrow \ovz'$ is a closed immersion (resp. an exact closed immersion) of log schemes, and if for any open neighborhood $U\subset Z'$ of $Z$, there exists an open neighborhood $\overline{U}\subset\ovz'$ of $\ovz$ such that $U$ is schematically dense in $\overline{U}$.
	\item We say that $f$ is a {\it boundary strict open immersion} if $\ovz \rightarrow \ovz'$ is a strict open immersion of log schemes and if $Z = \ovz\times_{\ovz'}Z'$.
	\item We say that $f$ is a {\it boundary (strict) immersion} if $f$ can be factored as $f = i\circ j$ by a boundary (exact) closed immersion $j\colon (Z,\ovz) \rightarrow (Y,\ovy)$ and a boundary strict open immersion $i\colon(Y,\overline{Y}) \rightarrow (Z',\ovz')$.
	\item We say that $f$ is a {\it first order thickening} if $\ovz \rightarrow \ovz'$ is an exact closed immersion defined by a square zero ideal in $\mathcal{O}_{\ovz'}$ and if $Z = \ovz\times_{\ovz'}Z'$.
	\end{enumerate}
\end{definition}

\begin{remark}
	In the original definition of \cite[Def.~2.1]{GrosseKloenne2004}, the condition $Z = \ovz\times_{\ovz'}Z'$ was not required for the definition of first order thickenings.
	However the authors think \cite[Prop.~2.3]{GrosseKloenne2004} is not valid without this condition.
	More precisely, to ensure that the morphism $\overline{L}\rightarrow\ovx$ constructed in the end of his proof induces a morphism $L\rightarrow X$, we need to require that $L'\hookrightarrow L$ is also en exact closed immersion defined by a square zero ideal.
	This is the reason why we added the condition.
	
	For example, consider smooth $T$-log scheme $\overline{X}$ and its log schematically dense open log subscheme $X$.
	The morphism $(X,\ovx)\rightarrow(\ovx,\ovx)$ is a first order thickening in the original definition, but not in our definition.
	In the original definition, the smoothness of $(X,\overline{X})$ requires that there locally exists a morphism $(\ovx,\ovx)\rightarrow(X,\ovx)$, but this is impossible in general.
\end{remark}

\begin{definition}
	Let $f\colon(Z,\ovz)\hookrightarrow(Z',\ovz')$ be a boundary immersion of $S$-log schemes with boundary.
	An {\it exactification} of $f$ is a factorization $f = h\circ j$ by a boundary exact closed immersion $j \colon(Z,\ovz)\hookrightarrow(Y,\ovy)$ and a morphism $h\colon (Y,\ovy) \rightarrow (Z',\ovz')$ which is log \'{e}tale as a morphism of log schemes $\ovy \rightarrow \ovz'$.
	
	We define the notion of an exactification of a boundary immersion of weak formal $\cs$-log schemes with boundary in a similar manner.
\end{definition}

\begin{definition}\label{def: lsbs}
	Let $\lsbs$ be the category of fine $S$-log schemes with boundary, and $\lsbcs$ be the category of fine weak formal $\cs$-log schemes with boundary.
	We regard $\lsbs$ and $\lsbcs$ as sites for the Zariski topology.
	Namely, a covering family of an object $(Z,\ovz)$ in $\lsbs$ or $\lsbcs$ is a family of boundary strict open immersions $(U_h,\overline{U}_h)\hookrightarrow (Z,\ovz)$ with $\bigcup_h\overline{U}_h = \ovz$.
\end{definition}

\begin{lemma}[{\cite[Lem.~1.2]{GrosseKloenne2004}}]
	The categories $\lsbs$ and $\lsbcs$ have finite products.
	The products in $\lsbs$ are given by $(Z,\ovz)\times_S(Z',\ovz') = (Z\times_SZ',\ovz\ovtimes_S\ovz')$, where $\ovz\ovtimes_S\ovz'$ is the log schematic image of $Z\times_SZ'$ in $\ovz\times_k\ovz'$.
	The products in $\lsbcs$ are given in a similar way.
\end{lemma}

Next we recall smoothness of $S$-log schemes with boundary introduced  in \cite{GrosseKloenne2004}.
It is defined in terms of the infinitesimal lifting property and another condition, which ensures that  log rigid cohomology is independent of the choice of local lifts.

\begin{definition}[{\cite[Def.~2.1]{GrosseKloenne2004}}]
	An $S$-log scheme with boundary $(Z,\ovz)$  is {\it smooth} if the following conditions hold:
	\begin{enumerate}
	\item For any first order thickening $i\colon (Y,\ovy)\hookrightarrow (Y',\ovy')$ and any morphism $f\colon(Y,\ovy) \rightarrow (Z,\ovz)$, locally on $\ovy'$ there exists a morphism $g\colon(Y',\ovy') \rightarrow (Z,\ovz)$ such that $f = g\circ i$.
	\item For any boundary exact closed immersion $i\colon (Y,\ovy)\hookrightarrow(Y',\ovy')$ and any morphism $f\colon (Y,\ovy) \rightarrow (Z,\ovz)$, locally on $\ovy'\ovtimes_S\ovz$ there exists an exactification $\ovy\hookrightarrow\ovx \rightarrow \ovy'\ovtimes_S\ovz$ of the diagonal embedding such that the projection $\ovx \rightarrow \ovy'$ is strict and log smooth.
	\end{enumerate}
	We define the notion of smoothness for weak formal $\cs$-log schemes with boundary in a similar manner.
\end{definition}

Note that if an $S$-log scheme with boundary $(Z,\ovz)$ is smooth, then $Z$ is log smooth over $S$.

To show functoriality of log rigid cohomology for a certain class of $T$-log schemes with boundary, we use a stronger condition of smoothness for weak formal $\ct$-log schemes with boundary defined in terms of local charts.

\begin{definition}\label{def: strongly smooth}
	A weak formal $\ct$-log scheme with boundary $(\cz,\ovcz)$ is {\it strongly smooth} if locally on $\ovcz$ there exist a chart $(\alpha\colon P_{\ovcz} \rightarrow \cn_{\ovcz},\ \beta\colon\N \rightarrow  P^\gp)$ for  $(\cz,\ovcz)$ extending $c_\ct$ and elements $a,b\in P$, satisfying the following conditions. 
	\begin{enumerate}
	\item $\beta^\gp$ is injective, and its cokernel is $p$-torsion free.
	\item $\beta(1) = b-a$.
	\item $\ovcz \rightarrow \Spwf\so_F[P]^\dagger$ is classically smooth.
	\item $\cz \cong \ovcz[\frac{1}{\alpha'(a)}]^\dagger$, where $\alpha'\colon P_{\ovcz} \rightarrow \co_{\ovcz}$ is induced by $\alpha$.
	\end{enumerate}
\end{definition}

\begin{remark}
	In the setting of Definition~\ref{def: strongly smooth}, let $P'$ be the submonoid of $P^\gp$ generated by $P$ and $-a$, and endow  $\Spwf\so_F[P']^\dagger$ and $\Spwf\so_F[P]^\dagger$ with the log structures defined by $P'$ and $P$, respectively.
	The above condition implies that there exists a Cartesian diagram
	\[\xymatrix{
	\cz \ar[r] \ar[d]&\ovcz \ar[d]\\
	\Spwf\so_F[P']^\dagger \ar[r]&\Spwf\so_F[P]^\dagger,
	}\]
	where the vertical arrows are strict and log smooth.
\end{remark}

\begin{proposition}[{\cite[Thm.~2.5]{GrosseKloenne2004}}]
	Strongly smooth weak formal $\ct$-log schemes with boundary are smooth.
\end{proposition}

\begin{lemma}\label{lem: product of strongly smooth}
	For $i = 1,2$, let $(\cz_i,\ovcz_i)$ be strongly smooth weak formal $\ct$-log schemes with boundary.
	Then $(\cz_1\times_\ct\cz_2,\ovcz_1\ovtimes_\ct\ovcz_2)$ is also strongly smooth.
\end{lemma}

\begin{proof}
	We may assume that  there are charts $(\alpha_i\colon P_{\ovcz_i} \rightarrow \cn_{\ovcz_i},\ \beta_i\colon\N \rightarrow  P^\gp_i)$ for $(\cz_i,\ovcz_i)$ and elements $a_i,b_i\in P_i$ as in Definition~\ref{def: strongly smooth}.
	Let $Q$ be the image of
		\[P_1\oplus P_2 \rightarrow  P_1^\gp\oplus_{\Z}P_2^\gp = \coker\left(\epsilon\colon\Z \rightarrow  P_1^\gp\oplus P_2^\gp,\ n \mapsto (\beta_1^\gp(n),-\beta_2^\gp(n))\right),\]
	and let $c,d\in Q$ be the images of $(a_1,a_2),(b_1,b_2)\in P_1\oplus P_2$, respectively.
	Then by the construction of products in $\lsbct$ we have a chart $\gamma\colon Q_{\ovcz_1\ovtimes_\ct\ovcz_2} \rightarrow \cn_{\ovcz_1\ovtimes_\ct\ovcz_2}$ which is compatible with projections to $\ovcz_i$, and $\alpha_i$, and the induced morphism $\ovcz_1\ovtimes_\ct\ovcz_2 \rightarrow \Spwf\so_F[Q]^\dagger$ is classically smooth.
	 Note that if we let $P_i'$ be the submonoid of $P_i^\gp$ generated by $P_i$ and $-a_i$, then $P_1'\oplus_{\N} P_2'$ is the submonoid of $P_1^\gp\oplus_{\Z}P_2^\gp$ generated by $Q$ and $-(a_1,a_2)$.
	 Thus we have
	 	\[\cz_1\times_\ct\cz_2 = (\ovcz_1\ovtimes_\ct\ovcz_2)[\frac{1}{\gamma'(c)}]^\dagger,\]
	where $\gamma'\colon Q_{\ovcz_1\ovtimes_\ct\ovcz_2} \rightarrow \mathcal{O}_{\ovcz_1\ovtimes_\ct\ovcz_2}$ is induced by $\gamma$.
	 We define $\delta\colon\N \rightarrow  Q^\gp$ by $1 \mapsto  d-c$.
	 From the morphism of exact sequences
	 	\[\xymatrix{
		0 \ar[r]&\Z \ar[r]\ar@{ = }[d]&\Z\oplus\Z \ar[r] \ar[d]^{(\beta_1^\gp,\beta_2^\gp)}&\Z \ar[r] \ar[d]^{\delta^\gp}&0\\
		0 \ar[r]&\Z \ar[r]^<<<<<{\epsilon}&P_1^\gp\oplus P_2^\gp \ar[r]& Q^\gp \ar[r]& 0,
		}\]
	where the upper horizontal arrows are given by $m \mapsto (m,-m)$ and $(m,n) \mapsto  m + n$, we obtain
		\begin{align*}
		\kker\delta^\gp \cong\kker\alpha_1^\gp\oplus\kker\alpha_2^\gp &&\text{and}&& \coker\delta^\gp \cong\coker\alpha_1^\gp\oplus\coker\alpha_2^\gp.
		\end{align*}
	This shows that $(\cz_1\times_\ct\cz_2,\ovcz_1\ovtimes_\ct\ovcz_2)$ is strongly smooth.
\end{proof}

The following lemma which ensures the local existence of exactifications of boundary immersions is the reason why we introduced the notion of strong smoothness.
Note that, the base change of a log schematically dense open immersion by a smooth morphism is in general not a log schematically dense open immersion.

\begin{lemma}\label{lem: strongly smooth}
	Let $(\cz,\ovcz)$ be a strongly smooth weak formal $\ct$-log scheme with boundary, and $f\colon \ovcy \rightarrow \ovcz$  a smooth morphism.
	Then the morphism $\cy: = \ovcy\times_{\ovcz}\cz\hookrightarrow\ovcy$ defines a weak formal $\ct$-log scheme with boundary, which is strongly smooth.
\end{lemma}

\begin{proof}
	Locally on $\ovcz$ and $\ovcy$, we can take a chart $(\alpha\colon P_{\ovcz} \rightarrow \cn_{\ovcz},\ \beta\colon\N \rightarrow  P^\gp)$ for $(\cz,\ovcz)$ and elements $a,b\in P$ as in Definition~\ref{def: strongly smooth}, and a chart $(\gamma\colon Q_{\ovcy} \rightarrow \cn_{\ovcy},\ \delta\colon P \rightarrow  Q)$ for $f$ extending $\alpha$, such that $\gamma^\gp\colon P^\gp \rightarrow  Q^\gp$ is injective with $p$-torsion free cokernel, and
		\[\ovcy \rightarrow \ovcz\times_{\Spwf\so_F[P]^\dagger}\Spwf\so_F[Q]^\dagger\]
	is classically smooth (see also \cite[Rem.~3.6]{Kato1989}).
	Let $Q'$ be the submonoid of $Q^\gp$ generated by $Q$ and $-\delta(a)$.
	Now we have a Cartesian diagram
		\[\xymatrix{
		\cy \ar[r] \ar[d]&\ovcy \ar[d]\\
		\Spwf\so_F[Q']^\dagger \ar[r]&\Spwf\so_F[Q]^\dagger
		}\]
	in which the vertical arrows are strict and log smooth, and hence classically smooth.
	Thus $\cy\hookrightarrow\ovcy$ is log schematically dense, and defines a weak formal $\ct$-log scheme with boundary. 
	Now one can easily see that $(\gamma\colon Q_{\ovcy} \rightarrow \cn_{\ovcy},\ \delta^\gp\circ\beta\colon\N \rightarrow  Q^\gp)$ and $\delta(a),\delta(b)\in Q$ satisfy the desired conditions in Definition~\ref{def: strongly smooth}.
\end{proof}

\begin{lemma}\label{lem: boundary exactification}
	Let $(\cz,\ovcz)\hookrightarrow(\cz',\ovcz')$ be a boundary immersion in $\lsbct$, and assume that $(\cz',\ovcz')$ is strongly smooth.
	Then locally on $\ovcz'$ there exists an exactification $(\cz,\ovcz)\hookrightarrow(\cy,\ovcy) \rightarrow (\cz',\ovcz')$ such that $(\cy,\ovcy)$ is strongly smooth.
\end{lemma}

\begin{proof}
	This follows immediately from Lemma~\ref{lem: strongly smooth}.
\end{proof}

As in the case of fine log schemes, we have the following. 

\begin{lemma}\label{lem: bdry tube}
	Let $(Z,\ovz)\hookrightarrow(\cz,\ovcz)$ be a boundary immersion of a fine $S$-log scheme with boundary into a fine admissible weak formal $\cs$-log scheme with boundary.
	Assume that there are two exactifications $(Z,\ovz)\hookrightarrow(\cz_i,\ovcz_i) \rightarrow (\cz,\ovcz)$ for $i = 1,2$.
	Let $\iota\colon]\ovz[_{\ovcz_1}\hookrightarrow]\ovz[_{\ovcz_2}$ be the canonical isomorphism from Lemma~\ref{lem: tube is canonical}.
	Denote by $\omega^{\sbt}_{(\cz_i,\ovcz_i)/\cs,\Q}$ the complexes of sheaves on $\ovcz_{i,\Q}$ given by tensoring the de~Rham complex of $(\cz_i,\ovcz_i)$ over $\cs$ \cite[Def.~1.6]{GrosseKloenne2004} with $\Q$.
	Then we have a canonical isomorphism $\iota^\ast\omega^{\sbt}_{(\cz_2,\ovcz_2)/\cs,\Q}|_{]\ovz[_{\ovcz_2}} \cong\omega^{\sbt}_{(\cz_1,\ovcz_1)/\cs,\Q}|_{]\ovz[_{\ovcz_1}}$.
\end{lemma}

\begin{proof}
	We may assume that there exists an exactification $\ovz\hookrightarrow\ovcz' \rightarrow \ovcz_1\times_{\ovcz}\ovcz_2$ of the diagonal embedding.
	Then $p_i\colon\ovcz' \rightarrow \ovcz_i$ is smooth and strict on neighborhoods of $\ovz$.
	Since shrinking does not change the tube, we may assume that $p_i$ is strict and smooth.
	If we set $\cz': = p_1^{-1}(\cz_1)\cap p_2^{-1}(\cz_2)$, it is log schematically dense in $\ovcz'$, as in the proof of \cite[Prop.~2.6]{GrosseKloenne2004}.
	Hence we have morphisms $(\cz',\ovcz') \rightarrow (\cz_i,\ovcz_i)$
	As in the classical case, we obtain $\iota_i\colon]\ovz[_{\ovcz'} \cong]\ovz[_{\ovcz_i}$ and $\iota_i^\ast\omega^{\sbt}_{(\cz_i,\ovcz_i)/\cs,\Q}|_{]\ovz[_{\ovcz_i}} \cong\omega^{\sbt}_{(\cz',\ovcz')/\cs,\Q}|_{]\ovz[_{\ovcz'}}$.
\end{proof}

\begin{definition}\label{def: bdry tube}
	Let $(Z,\ovz)\hookrightarrow(\cz,\ovcz)$ be a boundary immersion of a fine $T$-log scheme with boundary into a strongly smooth weak formal $\ct$-log scheme with boundary.
	Since locally on $\ovcz$ there exist  exactifications as in Lemma~\ref{lem: boundary exactification}, we can define the complex of sheaves $\omega^{\sbt}_{(\cz,\ovcz)/\cs,\Q}$ on $]\ovz[_{\ovcz}^{\log}$ by the gluing procedure explained in Definition~\ref{def: log tube}.
\end{definition}

\subsection{Axiomatization of canonical rigid complexes}\label{subsec: Axiomatisation}

In the next subsection we will define several variations of (log) rigid complexes for certain (simplicial) spaces.
Since rigid complexes a priori depend on the choice of local liftings, it is not clear how to extend the ``naive''  definition to simplicial spaces and prove their functoriality in general.
As a remedy we construct canonical rigid complexes which are independent of any choices, by following a method due to Berthelot outlined by Besser \cite{Besser2000}.
As we use it several times, we axiomatize this method here.
While this seems very laborious, we think that it streamlines the constructions later on and hope that it will be useful for future references.

We first introduce \textit{multisets} in order to allow multiplicity in the choice of local liftings. 
This  will make the arguments in the construction of base change simpler.

\begin{definition}
	Let $\mathsf{Set}$ be the category of sets, and let $P$ be a set.
	\begin{enumerate}
	\item A {\it multiset} in $P$ is a functor $A\colon P \rightarrow \mathsf{Set}$, where we regard $P$ as a discrete category.
		We call $\mathrm{Supp}(A): = \{x\in P\mid A(x)\neq\emptyset\}\subset P$ the {\it support} of a multiset $A$.
		We set $\underline{A}: = \coprod_{x\in P}A(x)$.
		A multiset $A$ in $P$ is said to be {\it finite} if its support is finite.
		We denote by $\mathsf{FMSet}(P)$ be the category of finite multisets in $P$ with natural transformations. 
	\item For multisets $A$ and $B$ in a set $P$, let $A\amalg B$ be a multiset in $P$ defined by
		\[(A\amalg B)(x): = A(x)\amalg B(x)\ \ (x\in P).\]
	Then the natural inclusions $A(x)\hookrightarrow A(x)\amalg B(x)$ and $B(x)\hookrightarrow A(x)\amalg B(x)$ give canonical morphisms $A \rightarrow  A\amalg B$ and $B \rightarrow  A\amalg B$.
	\end{enumerate}
\end{definition}

For $x\in\mathrm{Supp}(A)$, we call the cardinality of $A(x)$  the multiplicity of $x$.
Note that a morphism $\tau\colon A \rightarrow  B$ in $\mathsf{FMSet}(P)$ corresponds to a map $\underline{\tau}\colon\underline{A} \rightarrow \underline{B}$ with $\underline{\tau}(A(x))\subset B(x)$ for any $x\in P$.

\subsubsection{Rigid cohomological tuples}\label{subsubsec: Rigid cohomological tuple}

Berthelot--Besser's construction will be applied to a tuple
	\[\Xi = (\mathcal{A},\mathcal{B},\mathrm{red},\mathcal{I},\mathcal{M},\mathcal{K},\{P_Y\}_Y,\{g^\circ_\tau\}_{g,\tau},\{D_A\}_{(Y,\cz,i),A})\]
whose entries, numbered 1.,$\ldots$,9.,  we will define step by step in this section.
We will call such a tuple a {\it rigid cohomological tuple}.

\begin{enumerate}
	\item[1.] An essentially small site $\mathcal{A}$ called {\it base site}.
	\item[2.] An essentially small category $\mathcal{B}$ called {\it lifting category}.
	\item[3.] A functor $\mathrm{red}\colon\mathcal{B} \rightarrow \mathcal{A}$ called {\it reduction functor}.
	\item[4.] A class $\mathcal{I}$ of morphisms in $\mathcal{A}$ called {\it immersion class} which is closed under composition.
\end{enumerate}

\noindent
To identify the subsequent entries in the tuple $\Xi$, the following definition will be helpful.

\begin{definition}
	\begin{enumerate}
		\item A {\it $\Xi$-rigid triple} is a triple $(Y,\cz,i)$ consisting of $Y\in\mathcal{A}$, $\cz\in\mathcal{B}$, and $i\colon Y \rightarrow \mathrm{red}(\cz)$ in $\mathcal{I}$.
			A morphism of $\Xi$-rigid triples $(Y,\cz,i) \rightarrow (Y',\cz',i')$ is a pair $(f,F)$ of morphisms $f\colon Y \rightarrow  Y'$ in $\mathcal{A}$ and $F\colon \cz \rightarrow \cz'$, such that $\mathrm{red}(F)\circ i = i'\circ f$.
			We denote by $\mathsf{RT}_\Xi$ the category of $\Xi$-rigid triples.
			
		\item A {\it $\Xi$-rigid datum} for an object $Y\in\mathcal{A}$ is a pair $(\cz,i)$ which makes $(Y,\cz,i)$ into a $\Xi$-rigid triple.
			A morphism of $\Xi$-rigid data for $Y$ is a morphism $(\mathrm{id}_Y,F)$ in $\mathsf{RT}_\Xi$.
			We denote by $\mathsf{RD}_\Xi(Y)$ the category of $\Xi$-rigid data for $Y$.
			
		\item An object $Y\in\mathcal{A}$ is called {\it $\Xi$-embeddable} if the category $\mathsf{RD}_\Xi(Y)$ is non-empty.
			We denote by $\mathcal{A}_\Xi^{\mathrm{em}}$ the full subcategory of $\mathcal{A}$ consisting of $\Xi$-embeddable objects.
			
		\item An object $Y\in\mathcal{A}$ is called {\it locally $\Xi$-embeddable} if there exists a covering family $\{U_\lambda \rightarrow  Y\}_{\lambda\in\Lambda}$ of $Y$ in $\mathcal{A}$ such that each $U_\lambda$ is $\Xi$-embeddable.
			We denote by $\mathcal{A}_\Xi^{\mathrm{loc}}$ the full subcategory of $\mathcal{A}$ consisting of locally $\Xi$-embeddable objects.
			We regard $\mathcal{A}_\Xi^{\mathrm{loc}}$ as a site with respect to the topology induced from $\mathcal{A}$.
			Note that by definition the natural inclusion $\mathcal{A}_\Xi^{\mathrm{em}} \hookrightarrow\mathcal{A}_\Xi^{\mathrm{loc}}$ is a base for $\mathcal{A}_\Xi^{\mathrm{loc}}$ in the sense of Definition \ref{def: base}.
	\end{enumerate}
\end{definition}

\noindent
The next entries in the tuple $\Xi$ are given as follows.

\begin{enumerate}
	\item[5.] A cocomplete abelian category $\mathcal{M}$, called {\it realization category}.
		
	\item[6.] A functor $\mathcal{K}\colon\mathsf{RT}_\Xi \rightarrow  \mathscr{C}^ + (\mathcal{M})$, called {\it realization functor}.
\end{enumerate}

\begin{assumption}\label{ass: weq}
	For any object $Y\in\mathcal{A}$ and any morphism $F\colon (\cz,i) \rightarrow  (\cz',i')$ in $\mathsf{RD}_\Xi(Y)$, the morphism $\mathcal{K}(Y,\cz',i') \rightarrow \mathcal{K}(Y,\cz,i)$ induced by $(\mathrm{id}_Y,F)$ is a quasi-isomorphism.
\end{assumption}

Berthelot's idea to construct canonical complexes is to take into account all possible available rigid data.
However, it is not possible to take limits naively in such a way that they are functorial. 
Instead one is lead to consider certain filtered index categories as follows.

\begin{definition}
	Let $Y\in\mathcal{A}_\Xi^{\mathrm{em}}$.
	\begin{enumerate}
		\item Let $\mathrm{PT}_\Xi(Y)$ be the set of isomorphism classes of quadruples $(f,Y',\cz',i')$, where $f\colon Y \rightarrow  Y'$ is a morphism in $\mathcal{A}_\Xi^{\mathrm{em}}$ and $(\cz',i')$ is a $\Xi$-rigid datum for $Y'$. 
		
		\item Let $\mathsf{SET}_\Xi(Y): = \mathsf{FMSet}(\mathrm{PT}_\Xi(Y))$ be the category of finite multisets in $\mathrm{PT}_\Xi(Y)$.
			For $x\in\mathrm{PT}_\Xi(Y)$ and $a\in A(x)$, let $(f^{(a)},Y^{(a)},\cz^{(a)},i^{(a)}): = x$.

		\item Let $\mathsf{SET}^0_\Xi(Y)$ be the full subcategory of $\mathsf{SET}_\Xi(Y)$ consisting of all multisets $A$ in $\mathrm{PT}_\Xi(Y)$ such that there is an element $a\in\underline{A}$ with $f^{(a)} = \mathrm{id}_Y$.
			This is a filtered category.
	\end{enumerate}
\end{definition}

\noindent
This allows us to identify the seventh entry in the tuple $\Xi$.

\begin{enumerate}
	\item[7.] A family $\{P_Y\}_Y$ of contravariant functors $P_Y\colon\mathsf{SET}^0_\Xi(Y) \rightarrow \mathsf{RD}_\Xi(Y)$ for $Y\in\mathcal{A}_\Xi^{\mathrm{em}}$.
		We often use the notation $(\cz_A,i_A): = P_Y(A)$.
\end{enumerate}

\begin{definition}
	For $Y\in\mathcal{A}_\Xi^{\mathrm{em}}$, we define an object $\widehat{\mathcal{K}}(Y)$ in $\mathscr{C}^ + (\mathcal{M})$ by
		\[\widehat{\mathcal{K}}(Y): = \varinjlim_{A\in\mathsf{SET}^0_\Xi(Y)}\mathcal{K}(Y, \cz_A,i_A).\]
\end{definition}

\noindent
For a morphism $g\colon Y \rightarrow  Y'$ in $\mathcal{A}_\Xi^{\mathrm{em}}$, there is a map
	\[g^\circ\colon\mathrm{PT}_\Xi(Y') \rightarrow \mathrm{PT}_\Xi(Y),\ (f',\cz',i') \mapsto  (f'\circ g,\cz',i').\]
This induces a functor
	\[g^\circ\colon\mathsf{SET}_\Xi(Y') \rightarrow \mathsf{SET}_\Xi(Y)\]
	by $g^\circ(A)(x): = \coprod_{y\in (g^{\circ})^{-1}(x)}A(y)$.
Note that we have an equality $\underline{g^\circ(A)} = \underline{A}$. 
However $g^\circ$ does not send $\mathsf{SET}^0_\Xi(Y')$ to $\mathsf{SET}^0_\Xi(Y)$.

We assume that the family $\{P_Y\}_Y$ is natural in the sense that there exists a family as described below which forms the eighth entry of the tuple $\Xi$.

\begin{enumerate}
	\item[8.] A family $\{g^\circ_{\tau}\}_{g,\tau}$ of morphisms $g^\circ_{\tau}\colon (Y,\cz_B,i_B) \rightarrow (Y',\cz_C,i_C)$ in $\mathsf{RT}_\Xi$ for $g\colon Y \rightarrow  Y'$ in $\mathcal{A}_\Xi^{\mathrm{em}}$, $B\in\mathsf{SET}^0_\Xi(Y)$, $C\in\mathsf{SET}^0_\Xi(Y')$, and $\tau\colon g^\circ(C) \rightarrow  B$ in $\mathsf{SET}_\Xi(Y)$, such that
		\[P_{Y'}(\iota')\circ g^\circ_{\tau'} = g^\circ_\tau\circ P_Y(\iota)\]
		for any $\iota\colon B \rightarrow  B'$ in $\mathsf{SET}^0_\Xi(Y)$, $\iota'\colon C \rightarrow  C'$ in $\mathsf{SET}^0_\Xi(Y')$, $\tau\colon g^\circ(C) \rightarrow  B$, and $\tau'\colon g^\circ(C') \rightarrow  B'$ with $\tau'\circ g^\circ(\iota') = \iota\circ\tau$.
		Here $P_Y(\iota)$ and $P_{Y'}(\iota')$ are regarded as morphisms of $\Xi$-rigid triples. 
\end{enumerate}

\noindent
When no confusion arises, we  omit $\tau$ and write  $g^\circ = g^\circ_\tau$ to simplify notations.
We denote by
	\[g^\circ_\ast = g^\circ_{\tau,\ast}\colon\mathcal{K}(Y',\cz_C,i_C) \rightarrow \mathcal{K}(Y,\cz_B,i_B)\]
the morphism induced by $g^\circ_{\tau}$.

For $Y\in\mathcal{A}_\Xi^{\mathrm{em}}$, $A\in\mathsf{SET}^0_\Xi(Y)$, and $a\in\underline{A}$, we define $\langle a\rangle\in\mathsf{SET}^0_\Xi(Y^{(a)})$ by
	\[\langle a\rangle(x): = \begin{cases}\{a\}&\text{if }x = (\mathrm{id}_{Y^{(a)}},Y^{(a)},\cz^{(a)},i^{(a)}) \\
	 \emptyset&\text{if }x\neq (\mathrm{id}_{Y^{(a)}},Y^{(a)},\cz^{(a)},i^{(a)})  \end{cases}\]
for each $x\in\mathrm{PT}_\Xi(Y^{(a)})$.
Then the natural morphism $\iota_a\colon f^{(a),\circ}(\langle a\rangle) \rightarrow  A$ induces a morphism
	\[f^{(a),\circ}_{\iota_a}\colon(Y,\cz_A,i_A) \rightarrow (Y^{(a)},\cz^{(a)},i^{(a)})\]
in $\mathsf{RT}_\Xi$.

The following proposition can be proved similarly as \cite[Prop.~4.14]{Besser2000}.

\begin{proposition}\label{prop: functoriality axiom}
	For a morphism $g\colon Y \rightarrow  Y'$ in $\mathcal{A}_\Xi^{\mathrm{em}}$, there exists a unique map
	\begin{equation*}
	g^\ast\colon\widehat{\mathcal{K}}(Y') \rightarrow \widehat{\mathcal{K}}(Y)
	\end{equation*}
	such that the diagram
	\begin{equation}\label{eq: axiom diag}
		\xymatrix{
		\displaystyle\varinjlim_C\mathcal{K}(Y',\cz_C,i_C) \ar[r]^{g^\ast}
		&\displaystyle\varinjlim_B\mathcal{K}(Y,\cz_B,i_B)\\
		\displaystyle\varinjlim_{(B,C)}\mathcal{K}(Y',\cz_C,i_C) \ar[r]^<<<<<{g^\circ_\ast} \ar[u]^{(B,C) \mapsto  C}
		&\displaystyle\varinjlim_{(B,C)}\mathcal{K}(Y,\cz_{B\amalg g^\circ(C)}i_{B\amalg g^\circ(C)}), \ar[u]^{(B,C) \mapsto  B\amalg g^\circ(C)}
	}\end{equation}
	commutes. Here the limits are taken over $B\in\mathsf{SET}_\Xi^0(Y)$ and $C\in\mathsf{SET}_\Xi^0(Y')$, and $g^\circ_\ast$ are associated to the canonical morphisms $g^\circ(C) \rightarrow  B\amalg g^\circ(C)$.
	For any two composable morphisms $g\colon Y \rightarrow  Y'$ and $h\colon Y' \rightarrow  Y''$, we have $(h\circ g)^\ast = g^\ast\circ h^\ast$.
\end{proposition}

To relate $\mathcal{K}(Y,\cz,i)$ and $\widehat{\mathcal{K}}(Y)$ in a functorial way, we need to consider another index category.

\begin{definition}
Let $(Y,\cz,i)\in\mathsf{RT}_\Xi$.
	\begin{enumerate}
		\item Let $\mathrm{PT}_\Xi(Y,\cz,i)$ be the set of all isomorphism classes of morphisms of rigid triples whose domains are $(Y,\cz,i)$.
		\item Let $\mathsf{SET}_\Xi(Y,\cz,i): = \mathsf{FMSet}(\mathrm{PT}_\Xi(Y,\cz,i))$, the category of finite multisets in $\mathrm{PT}_\Xi(Y,\cz,i)$.
			For $x\in\mathrm{PT}_\Xi(Y,\cz,i)$ and $a\in A(x)$, let
				\[(f^{(a)},F^{(a)})\colon(Y,\cz,i) \rightarrow (Y^{(a)},\cz^{(a)},i^{(a)})\]
			be the morphism $x$. 
			
		\item  Let $\mathsf{SET}_\Xi^0(Y,\cz,i)$ be the full subcategory of $\mathsf{SET}_\Xi(Y,\cz,i)$ consisting of all multisets $A$ in $\mathrm{PT}_\Xi(Y,\cz,i)$ with $A(\mathrm{id}_{(Y,\cz,i)})\neq\emptyset$.
			This is a filtered category.
	\end{enumerate}
\end{definition}

\noindent
Similarly as above, for a morphism $(g,G)\colon(Y,\cz,i) \rightarrow (Y',\cz',i')$ in $\mathsf{RT}_\Xi$, we obtain a map
	\[(g,G)^\circ\colon\mathsf{PT}_\Xi(Y',\cz',i') \rightarrow \mathsf{PT}_\Xi(Y,\cz,i),\ (f,F) \mapsto  (f,F)\circ(g,G)\]
and a functor
	\[(g,G)^\circ\colon\mathsf{SET}_\Xi(Y',\cz',i') \rightarrow \mathsf{SET}_\Xi(Y,\cz,i),\]
so that $(g,G)^\circ(A)(x): = \coprod_{y\in(g,G)^{\circ,-1}(x)}A(y)$.
	
For $(Y,\cz,i)\in\mathsf{RT}_\Xi$, the forgetful map $\Pi\colon\mathrm{PT}_\Xi(Y,\cz,i) \rightarrow \mathrm{PT}_\Xi(Y)$ induces functors
	\begin{eqnarray*}
	&&\Pi_{(Y,\cz,i)}\colon\mathsf{SET}_\Xi(Y,\cz,i) \rightarrow \mathsf{SET}_\Xi(Y),\\
	&&\Pi_{(Y,\cz,i)}\colon\mathsf{SET}^0_\Xi(Y,\cz,i) \rightarrow \mathsf{SET}^0_\Xi(Y),
	\end{eqnarray*}
which are defined by $\Pi(A)(x): = \coprod_{y\in\Pi^{-1}(x)} A(y)$.

Note that we have
	\[\Pi_{(Y,\cz,i)}\circ(g,G)^\circ = g^\circ\circ\Pi_{(Y',\cz',i')}\]
for any $(g,G)\colon(Y,\cz,i) \rightarrow (Y',\cz',i')$.

\begin{definition}
	\begin{enumerate}
	\item For $(Y,\cz,i)\in\mathsf{RT}_\Xi$, we define a functor 
		\[P_{(Y,\cz,i)}\colon\mathsf{SET}^0_\Xi(Y,\cz,i) \rightarrow \mathsf{RD}_\Xi(Y)\] 
	by $P_{(Y,\cz,i)}: = P_Y\circ\Pi_{(Y,\cz,i)}$.
		For $A\in\mathsf{SET}^0_\Xi(Y,\cz,i)$, we often denote $(\cz_A,i_A): = (\cz_{\Pi_{(Y,\cz,i)}(A)},i_{\Pi(Y,\cz,i)}(A)) = P_{(Y,\cz,i)}(A)$.
	\item For a morphism $(g,G)\colon(Y,\cz,i) \rightarrow (Y',\cz',i')$ in $\mathsf{RT}_\Xi$, $B\in\mathsf{SET}_\Xi^0(Y,\cz,i)$, $C\in\mathsf{SET}^0_\Xi(Y',\cz',i')$, and $\tau\colon(g,G)^\circ(C) \rightarrow  B$, the morphism
		\[\Pi(\tau)\colon g^\circ\circ\Pi_{(Y',\cz',i')}(C) = \Pi_{(Y,\cz,i)}\circ(g,G)^\circ(C) \rightarrow \Pi_{(Y,\cz,i)}(B)\]
	induces $g^\circ_{\Pi(\tau)}\colon (Y,\cz_B,i_B) \rightarrow (Y',\cz_C,i_C)$.
		We denote this by $(g,G)^\circ_\tau: = g^\circ_{\Pi(\tau)}$.
	\end{enumerate}
\end{definition}

\noindent
We can now identify the ninth entry of the tuple $\Xi$.

\begin{enumerate}
	\item[9.] A family $\{D_A\}_{(Y,\cz,i),A}$ of morphisms $D_A\colon(\cz,i) \rightarrow (\cz_A,i_A)$ in $\mathsf{RD}_\Xi(Y)$ for $(Y,\cz,i)\in\mathsf{RT}_\Xi$ and $A\in\mathsf{SET}^0_\Xi(Y,\cz,i)$, such that
		\[D_A = P(\iota)\circ D_{A'}\]
		for any morphism $\iota\colon A \rightarrow  A'$ in $\mathsf{SET}^0_\Xi(Y,\cz,i)$, and
		\[(g,G)^\circ_\tau\circ D_B = D_C\circ (g,G)\]
		for any $(g,G)\colon(Y,\cz,A) \rightarrow (Y',\cz',i')$ in $\mathsf{RT}_\Xi$, $B\in\mathsf{SET}^0_\Xi(Y,\cz,i)$, $C\in\mathsf{SET}^0_\Xi(Y',\cz',i')$, and $\tau\colon(g,G)^\circ(C) \rightarrow  B$.
\end{enumerate}

\begin{definition}
	For $(Y,\cz,i)\in\mathsf{RT}_\Xi$, we define an object $\widetilde{\mathcal{K}}(Y,\cz,i)$ in $\mathscr{C}^ + (\mathcal{M})$ by
		\[\widetilde{\mathcal{K}}(Y,\cz,i): = \varinjlim_{A\in\mathsf{SET}^0_\Xi(Y,\cz,i)}\mathcal{K}(Y,\cz_A,i_A).\]
\end{definition}

We obtain functoriality for $\widetilde{\mathcal{K}}$ in a similar manner as for $\widehat{\mathcal{K}}$ (cf.$\;$Proposition \ref{prop: functoriality axiom}).

\begin{proposition}
	For a morphism $(g,G)\colon (Y,\cz,i) \rightarrow  (Y',\cz',i')$ in $\mathsf{RT}_\Xi$, there exists a unique map
	\begin{equation*}
	(g,G)^\ast\colon\widetilde{\mathcal{K}}(Y',\cz',i') \rightarrow \widetilde{\mathcal{K}}(Y,\cz,i)
	\end{equation*}
	such that the diagram
	\begin{equation}\label{eq: axiom diag tilde}
		\xymatrix{
		\displaystyle\varinjlim_C\mathcal{K}(Y',\cz_C,i_C) \ar[r]^{(g,G)^\ast}
		&\displaystyle\varinjlim_B\mathcal{K}(Y,\cz_B,i_B)\\
		\displaystyle\varinjlim_{(B,C)}\mathcal{K}(Y',\cz_C,i_C) \ar[r]^<<<<<{(g,G)^\circ_\ast} \ar[u]^{(B,C) \mapsto  C}
		&\displaystyle\varinjlim_{(B,C)}\mathcal{K}(Y,\cz_{B\amalg (g,G)^\circ(C)},i_{B\amalg (g,G)^\circ(C)}), \ar[u]^{(B,C) \mapsto  B\amalg (g,G)^\circ(C)}
	}\end{equation}
	where the limits are taken over $B\in\mathsf{SET}_\Xi^0(Y,\cz,i)$ and $C\in\mathsf{SET}_\Xi^0(Y',\cz',i')$, commutes.
	For any two composable morphisms $(g,G)\colon (Y,\cz,i) \rightarrow  (Y',\cz',i')$ and $(h,H)\colon (Y',\cz',i') \rightarrow  (Y'',\cz'',i'')$, we have $(h\circ g,H\circ G)^\ast = (g,G)^\ast\circ (h,H)^\ast$.
\end{proposition}

By  Assumption \ref{ass: weq}, for any $(Y,\cz,i)\in\mathcal{A}_\Xi^{\mathrm{em}}$ and $A\in\mathsf{SET}^0_\Xi(Y,\cz,i)$, the morphism $D^\ast_A\colon\mathcal{K}(Y,\cz_A,i_A) \rightarrow \mathcal{K}(Y,\cz,i)$ induced by $D_A$ is a quasi-isomorphism.
If we take the limit over $\mathsf{SET}^0_\Xi(Y,\cz,i)$, we obtain a quasi-isomorphism
	\begin{equation}
	D_{(Y,\cz,i)}\colon\widetilde{\mathcal{K}}(Y,\cz,i) \rightarrow \mathcal{K}(Y,\cz,i),
	\end{equation}
which is functorial on $\mathsf{RT}_\Xi$.

Recall that for $(Y,\cz,i)\in\mathsf{RT}_\Xi$ and $A\in\mathsf{SET}^0_\Xi(Y,\cz,i)$, we set $(Y,\cz_A,i_A): = (Y,\cz_{\Pi(A)},i_{\Pi(A)})$ and hence
	\[\widetilde{\mathcal{K}}(Y,\cz,i) = \varinjlim_{A\in\mathsf{SET}^0_\Xi(Y,\cz,i)}\mathcal{K}(Y,\cz_{\Pi(A)},i_{\Pi(A)}).\]
Since the natural map
	\[\varinjlim_{A\in\mathsf{SET}^0_\Xi(Y,\cz,i)}\mathcal{K}(Y,\cz_{\Pi(A)},i_{\Pi(A)}) \rightarrow \widehat{\mathcal{K}}(Y)\]
is a quasi-isomorphism, we obtain a quasi-isomorphism
	\[\Delta_{(Y,\cz,i)}\colon\widetilde{\mathcal{K}}(Y,\cz,i) \rightarrow \widehat{\mathcal{K}}(Y).\]
It is straightforward to show  that this is functorial in $(Y,\cz,i)$.

\begin{definition}
	\begin{enumerate}
	\item Let $\mathfrak{K}_\Xi$ be the sheafification of the presheaf $\widehat{\mathcal{K}}$ on $\mathcal{A}_\Xi^{\mathrm{em}}$.
		By Proposition \ref{prop: topos eq}, we may extend it to a sheaf on $\mathcal{A}_\Xi^{\mathrm{loc}}$.
		For $Y\in\mathcal{A}_\Xi^{\mathrm{loc}}$, we define an object $\mathbb{R}\Gamma_\Xi(Y)\in\mathscr{D}^ + (\mathcal{M})$ by
			\[\mathbb{R}\Gamma_\Xi(Y): =  \R\Gamma(\mathcal{A}_\Xi^{\mathrm{loc}}/Y, \mathfrak{K}_\Xi).\]
	\item Let $Y_{\sbt}$ be a simplicial object on $\mathcal{A}_\Xi^{\mathrm{loc}}$.
		Let $(\mathcal{A}_\Xi^{\mathrm{loc}}/Y_{\sbt})^\mathrm{tot}$ be the total site of the simplicial site $\mathcal{A}_\Xi^{\mathrm{loc}}/Y_{\sbt}$.
		Then the restriction of $\mathfrak{K}_\Xi$ defines a sheaf on the total site on $(\mathcal{A}_\Xi^{\mathrm{loc}}/Y_{\sbt})^{\mathrm{tot}}$, which we denote again by $\mathfrak{K}_\Xi$.
		We define an object $\mathbb{R}\Gamma_\Xi(Y_{\sbt})\in\mathscr{D}^ + (\mathcal{M})$ by
			\[\mathbb{R}\Gamma_\Xi(Y_{\sbt}): =  \R\Gamma((\mathcal{A}_\Xi^{\mathrm{loc}}/ Y_{\sbt})^{\mathrm{tot}},\mathfrak{K}_\Xi).\]
	\end{enumerate}		
\end{definition}

\noindent
For $Y\in\mathcal{A}_\Xi^{\mathrm{em}}$, there is a natural morphism $\theta_Y\colon\mathrm{loc}\circ\widehat{\mathcal{K}}(Y) \rightarrow \mathbb{R}\Gamma_\Xi(Y)$ in $\mathscr{D}^ + (\mathcal{M})$, where $\mathrm{loc}\colon\mathscr{C}^ +  (\mathcal{M}) \rightarrow \mathscr{D}^ +  (\mathcal{M})$ denotes the canonical functor.

\begin{assumption}\label{ass: descent}
	For $(Y,\cz,i)\in\mathsf{RT}_\Xi$ and a covering family $\{j_\lambda\colon U_\lambda \rightarrow  Y\}_{\lambda\in\Lambda}$ in $\mathcal{A}$, we have the following two conditions.
	\begin{enumerate}
	\item The triple $(U_\lambda,\cz,i\circ j_\lambda)$ is also in $\mathsf{RT}_\Xi$.
	\item Bearing in mind that the first condition provides us in an  obvious way with a simplicial object $(U_{\sbt},\cz,i\circ j_{\sbt})$ in $\mathsf{RT}_\Xi$, we require that the natural morphism $\mathcal{K}(Y,\cz,i) \rightarrow  s(\mathcal{K}(U_{\sbt},\cz,i\circ j_{\sbt}))$ is a quasi-isomorphism, where $s$ is the functor which associates to a simplicial complex its simple complex in $\mathscr{C}^ + (\mathcal{M})$.
	\end{enumerate}
\end{assumption}

\begin{proposition}
	For $Y\in\mathcal{A}_\Xi^{\mathrm{em}}$, the morphism $\theta_Y\colon\mathrm{loc}\circ\widehat{\mathcal{K}}(Y) \rightarrow \mathbb{R}\Gamma_\Xi(Y)$ is an isomorphism.
\end{proposition}

\begin{proof}
	Take a $\Xi$-rigid datum $(\cz,i)$ for $Y$.
	For any covering family $\{j_\lambda\colon U_\lambda \rightarrow  Y\}_{\lambda\in\Lambda}$ in $\mathcal{A}$, by the assumption, the left vertical arrow in a diagram
		\[\xymatrix{
		\mathcal{K}(Y,\cz,i) \ar[d]& \widetilde{\mathcal{K}}(Y,\cz,i) \ar[d] \ar[l]_\sim \ar[r]^\sim& \widehat{\mathcal{K}}(Y) \ar[d]\\
		s(\mathcal{K}(U_{\sbt},\cz,i\circ j_{\sbt}))& s(\widetilde{\mathcal{K}}(U_{\sbt},\cz,i\circ j_{\sbt})) \ar[l]_\sim \ar[r]^<<<<<\sim& s(\widehat{\mathcal{K}}(U_{\sbt}))
		}\]
	is a quasi-isomorphism, and hence the right vertical arrow is also a quasi-isomorphism.
	This implies the proposition.
\end{proof}

In conclusion, we obtained the following theorem.

\begin{theorem}\label{Thm: cohomology functors}
	The constructions above give functors
		\begin{align*}
		\mathbb{R}\Gamma_\Xi\colon\mathcal{A}_\Xi^{\mathrm{loc}} \rightarrow \mathscr{D}^ + (\mathcal{M}),&&
		\widehat{\mathcal{K}}\colon\mathcal{A}_\Xi^{\mathrm{em}} \rightarrow \mathscr{C}^ + (\mathcal{M}),&&
		\widetilde{\mathcal{K}}\colon\mathsf{RT}_\Xi \rightarrow \mathscr{C}^ + (\mathcal{M}),
		\end{align*}
	with functorial isomorphisms $\widehat{\mathcal{K}}(Y) \rightarrow \mathbb{R}\Gamma_\Xi(Y)$ in $\mathscr{D}^ + (\mathcal{M})$ for $\mathcal{A}_\Xi^{\mathrm{em}}$, and quasi-isomorphisms
		\[\widehat{\mathcal{K}}(Y)\xleftarrow[]{\Delta_{(Y,\cz,i)}}\widetilde{\mathcal{K}}(Y,\cz,i) \xrightarrow[]{D_{(Y,\cz,i)}}\mathcal{K}(Y,\cz,i)\]
	in $\mathscr{C}^ + (\mathcal{M})$, which are functorial on $\mathsf{RT}_\Xi$ if we regard $\widehat{\mathcal{K}}$ as a functor from $\mathsf{RT}_\Xi$ through the forgetful functor $\mathsf{RT}_\Xi \rightarrow \mathcal{A}_\Xi^{\mathrm{em}}$.
\end{theorem}

\begin{remark}\label{rem: gluing lifts}
	In certain cases, it is possible to compute $\mathbb{R}\Gamma_\Xi(Y)$ explicitly even if  $Y\in\mathcal{A}_\Xi^{\mathrm{loc}}$ is not necessarily embeddable.
	Let $\{U_\lambda \rightarrow  Y\}_{\lambda\in\Lambda}$ a covering family of $Y$ in $\mathcal{A}$ by $\Xi$-embeddable objects with $\Xi$-rigid data $(\cz_\lambda,i_\lambda)$ for $U_\lambda$.
	Assume that for any $I = (\lambda_0,\ldots,\lambda_m)\in\Lambda^{m + 1}$ there exist fiber products $U_I: = U_{\lambda_0}\times_Y\cdots\times_YU_{\lambda_m}$.
	Then by Assumption~\ref{ass: descent}, for any $I = (\lambda_0,\ldots,\lambda_m)$ and $j = 0,\ldots,m$, the composition $i_{I,j}\colon U_I \rightarrow  U_{\lambda_j} \xrightarrow{i_{\lambda_j}}\cz_{\lambda_j}$ defines a $\Xi$-rigid datum for $U_I$.
	Let $A_I$ be a multiset in $\mathsf{SET}^0_\Xi(U_I)$ whose support is $\{(\mathrm{id}_{U_I},U_I,\cz_{\lambda_j},i_{I,j})\}_{j = 0,\ldots,n}$ and assume that the multiplicity of each element is one.
	Then in an obvious way we obtain a simplicial object $(U_{\sbt},\cz_{\sbt},i_{\sbt})$ in $\mathsf{RT}_\Xi$ given by
	$$(U_m,\cz_m,i_m) = (\coprod_{I\in\Lambda^{m + 1}}U_I,\coprod_{I\in\Lambda^{m + 1}}\cz_{A_I},\coprod_{I\in\Lambda^{m + 1}} i_{A_I})$$ 
	and quasi-isomorphisms
		\begin{equation}\label{eq: gluing lifts}
		\mathcal{K}(U_{\sbt},\cz_{\sbt},i_{\sbt})\xleftarrow{\sim}
		\widetilde{\mathcal{K}}(U_{\sbt},\cz_{\sbt},i_{\sbt}) \xrightarrow{\sim}
		\widehat{\mathcal{K}}(U_{\sbt})\xleftarrow{\sim}
		\widehat{\mathcal{K}}(Y) \xrightarrow{\sim}
		\mathbb{R}\Gamma_\Xi(Y).
		\end{equation}
\end{remark}

\subsubsection{Morphisms of rigid cohomological tuples}\label{subsubsec: Morphisms}

Morphisms of rigid cohomological tuples will play a role in the construction of base change morphisms and Frobenius endomorphisms.

Let $\Xi$ and $\Xi'$ be rigid cohomological tuples as described in the previous section.
A morphism of rigid cohomological tuples $L\colon\Xi \rightarrow \Xi'$ is a quadruple
	\[L = (L_{\mathrm{base}},L_{\mathrm{lift}},L_{\mathrm{mod}}, L_{\mathrm{coh}})\]
whose entries are defined as follows.
	\begin{enumerate}
	 \item The first two entries are a continuous functor $L_{\mathrm{base}}\colon\mathcal{A} \rightarrow \mathcal{A}'$ and a functor $L_{\mathrm{lift}}\colon\mathcal{B} \rightarrow \mathcal{B}'$, such that $L_{\mathrm{base}}(\mathcal{I})\subset\mathcal{I}'$ and $\mathrm{red}'\circ L_{\mathrm{lift}} = L_{\mathrm{base}}\circ\mathrm{red}$.
	As a consequence they induce functors
	\begin{align*}
	&L_{\mathsf{RT}}\colon\mathsf{RT}_\Xi \rightarrow \mathsf{RT}_{\Xi'},&&\\
	&L_{\mathsf{RD},Y}\colon\mathsf{RD}_\Xi(Y) \rightarrow \mathsf{RD}_{\Xi'}(L_{\mathrm{base}}(Y))&&\text{for }Y\in\mathcal{A}_\Xi^{\mathrm{em}},\\
	&L_{\mathsf{SET},Y}\colon\mathsf{SET}_\Xi(Y) \rightarrow \mathsf{SET}_{\Xi'}(L_{\mathrm{base}}(Y))&&\text{for }Y\in\mathcal{A}_\Xi^{\mathrm{em}},\\
	&L_{\mathsf{SET}^0,Y}\colon\mathsf{SET}^0_\Xi(Y) \rightarrow \mathsf{SET}^0_{\Xi'}(L_{\mathrm{base}}(Y))&&\text{for }Y\in\mathcal{A}_\Xi^{\mathrm{em}},\\
	&L_{\mathsf{SET},(Y,\cz,i)}\colon\mathsf{SET}_\Xi(Y,\cz,i) \rightarrow \mathsf{SET}_{\Xi'}(L_{\mathsf{RT}}(Y,\cz,i))&&\text{for }(Y,\cz,i)\in\mathsf{RT}_\Xi,\\
	&L_{\mathsf{SET}^0,(Y,\cz,i)}\colon\mathsf{SET}^0_\Xi(Y,\cz,i) \rightarrow \mathsf{SET}^0_{\Xi'}(L_{\mathsf{RT}}(Y,\cz,i))&&\text{for }(Y,\cz,i)\in\mathsf{RT}_\Xi.
	\end{align*}
	\item The remaining two entries consist of an exact functor $L_{\mathrm{mod}}\colon\mathcal{M}'\rightarrow \mathcal{M}$ and a natural transformation $L_{\mathrm{coh}}\colon \mathcal{K} \rightarrow L_{\mathrm{mod}}\circ\mathcal{K}'\circ L_{\mathrm{RT}}$. 
	
	\item We require the compatibility in a natural way, that is:
	 	\begin{align*}
		&L_{\mathsf{RD},Y}\circ P_Y = P_{L_{\mathrm{base}}(Y)}\circ L_{\mathsf{SET}^0,Y}&\text{for }&Y\in\mathcal{A}_\Xi^{\mathrm{em}},\\
		&L_{\mathrm{base}}(g)_{L_{\mathsf{SET},Y}(\tau)}^\circ = L_{\mathsf{RT}}(g^\circ_\tau)&\text{for }&g\colon Y \rightarrow  Y'\text{ in }\mathcal{A}_\Xi^{\mathrm{em}},\ B\in\mathsf{SET}^0_\Xi(Y),\ \\
		& &&C\in\mathsf{SET}^0_\Xi(Y'), \\
		& &&\tau\colon g^\circ(C) \rightarrow  B\text{ in }\mathsf{SET}_\Xi(Y),\\
		&D_{\mathsf{SET}^0,(Y,\cz,i)} = L_{\mathsf{RD},Y}(D_A)&\text{for }&(Y,\cz,i)\in\mathsf{RT}_\Xi.
	\end{align*}
	\end{enumerate}
	 Here, by abuse of notation, we also denote by $L_{\mathrm{mod}}\colon\mathscr{C}^ + (\mathcal{M}') \rightarrow \mathscr{C}^ + (\mathcal{M})$ the functor induced by $L_{\mathrm{mod}}$.

\begin{proposition}\label{prop: eq: axiom bc}
	There exist canonical natural transformations
		\begin{align*}
		&\widehat{L}_{\mathrm{coh}}\colon \widehat{\mathcal{K}} \rightarrow L_{\mathrm{mod}}\circ\widehat{\mathcal{K}'}\circ L_{\mathrm{base}},\\
		&\widetilde{L}_{\mathrm{coh}}\colon \widetilde{\mathcal{K}} \rightarrow  L_{\mathrm{mod}}\circ\widetilde{\mathcal{K}'}\circ L_{\mathsf{RT}},
		\end{align*}
	such that for any $(Y,\cz,i)\in\mathsf{RT}_\Xi$. the following diagram commutes:
		\[\mathclap{
		\xymatrix{
		\widehat{\mathcal{K}}(Y) \ar[d]^{\widehat{L}_{\mathrm{coh}}}
		&\widetilde{\mathcal{K}}(Y,\cz,i) \ar[d]^{\widetilde{L}_{\mathrm{coh}}} \ar[l]_{\Delta}^\sim \ar[r]^{D}_\sim
		&\mathcal{K}(Y,\cz,i) \ar[d]^{L_{\mathrm{coh}}}\\
		L_{\mathrm{mod}}\circ\widehat{\mathcal{K}'}\circ L_{\mathrm{base}}(Y)
		&L_{\mathrm{mod}}\circ\widetilde{\mathcal{K}'}\circ L_{\mathsf{RT}}(Y,\cz,i) \ar[l]^{L_{\mathrm{mod}}(\Delta)}_\sim \ar[r]_{L_{\mathrm{mod}}(D)}^\sim
		&L_{\mathrm{mod}}\circ\mathcal{K'}\circ L_{\mathsf{RT}}(Y,\cz,i)
		}}\]
		where  subscripts were omitted.
\end{proposition}

\begin{proof}
	The natural transformations $\widehat{L}_{\mathrm{coh}}$ and $\widetilde{L}_{\mathrm{coh}}$ are defined as the compositions of the vertical arrows in the following diagram for each $(Y,\cz,i)\in\mathsf{RT}_\Xi$.
	{\small\begin{equation*}\label{eq: axiom bc}\mathclap{
		\xymatrix{
		\varinjlim_A\mathcal{K}\circ P_Y(A) \ar[d]^{L_{\mathrm{coh}}} &
		\varinjlim_B\mathcal{K}\circ P_{(Y,\cz,i)}(B)  \ar[l]_\Delta^{\sim} \ar[rd]^D_{\sim} \ar[d]^{L_{\mathrm{coh}}} &
		\\
		\varinjlim_AL_{\mathrm{mod}}\circ\mathcal{K}'\circ L_{\mathsf{RT}}\circ P_Y(A) \ar@{ = }[d] &
		\varinjlim_BL_{\mathrm{mod}}\circ\mathcal{K}'\circ L_{\mathsf{RT}}\circ P_{(Y,\cz,i)}(B)  \ar[l]^{\sim} \ar[rd]_{\sim} \ar@{ = }[d]&
		\mathcal{K}(Y,\cz,i) \ar[d]^{L_{\mathrm{coh}}} \\
		\varinjlim_AL_{\mathrm{mod}}\circ\mathcal{K}'\circ P'_{L_{\mathrm{base}}(Y)}\circ L_{\mathsf{SET}^0,Y}(A) \ar[d]_{\sim} &
		\varinjlim_B L_{\mathrm{mod}}\circ\mathcal{K}'\circ P'_{L_{\mathsf{RT}}(Y,\cz,i)}\circ L_{\mathsf{SET}^0,(Y,\cz,i)}(B) \ar[d]_{\sim} \ar[l]^{\sim} \ar[r]_>>>>>{\sim} &
		L_{\mathrm{mod}}\circ\mathcal{K}'\circ L_{\mathsf{RT}}(Y,\cz,i)\\
		\varinjlim_{A'}L_{\mathrm{mod}}\circ\mathcal{K}'\circ P'_{L_{\mathrm{base}}(Y)}(A')&
		\varinjlim_{B'}L_{\mathrm{mod}}\circ\mathcal{K}'\circ P'_{L_{\mathsf{RT}}(Y,\cz,i)}(B') \ar[l]^{\sim}_\Delta \ar[ru]^D_{\sim}&
	}}\end{equation*}}
where the limits run over $A\in\mathsf{SET}^0_\Xi(Y)$, $B\in\mathsf{SET}^0_\Xi(Y,\cz,i)$, $A'\in\mathsf{SET}^0_{\Xi'}(L_{\mathrm{base}}(Y))$, and $B'\in\mathsf{SET}^0_{\Xi'}(L_{\mathsf{RT}}(Y,\cz,i))$.
	The left arrows in the above diagram are given in the same way as $\Delta$, and the right arrows  are induced by $D_B$, $D_{L_{\mathsf{SET}^0,(Y,\cz,i)}(B)}$, and $D_{B'}$.
\end{proof}

Finally we look at the functoriality of $\mathbb{R}\Gamma_\Xi$.
The transformation $\widehat{L}_{\mathrm{coh}}$ gives a morphism $\widehat{\mathcal{K}} \rightarrow  L_{\mathrm{mod}}\circ L_{\mathrm{base}}^{-1}\widehat{\mathcal{K}}'$ of $\mathscr{C}^ + (\mathcal{M})$-valued presheaves on $\mathcal{A}_\Xi^{\mathrm{em}}$, and hence a morphism $\mathbb{K} \rightarrow  L_{\mathrm{mod}}\circ L_{\mathrm{base}}^{\ast}\mathbb{K}'$ of $\mathscr{C}^ + (\mathcal{M})$-valued sheaves on $\mathcal{A}_\Xi^{\mathrm{loc}}$.
	\[\mathbb{R}\Gamma_\Xi(Y) \rightarrow R\Gamma(\mathcal{A}_\Xi^{\mathrm{loc}}/Y,L_{\mathrm{mod}} \circ L_{\mathrm{base}}^\ast\mathbb{K}') \rightarrow L_{\mathrm{mod}}(\mathbb{R}\Gamma_{\Xi'}(L_{\mathrm{base}}(Y))).\]

%
\subsection{Canonical rigid complexes}\label{Subsec: Log-rig-com}
%

In this section we explain how to  apply the axioms in \S\ref{subsubsec: Rigid cohomological tuple} to several concrete situations.

\subsubsection{Rigid cohomology of fine log schemes}\label{subsubsec: usual rigid cohomology}
First we assume that the pair $(R,S\hookrightarrow\mathcal{S})$ is one of $(F,k^0\hookrightarrow\so_F^0)$, $(K,k^0\hookrightarrow\so_K^\pi)$, or $(F[t]^\dagger,T\hookrightarrow\ct)$.
We define a rigid cohomological tuple $\Xi(\mathcal{S})$ as follows:

\begin{enumerate}
	\item The base site $\mathsf{LS}_S$ consists of fine log schemes over $S$, and is endowed with the Zariski topology.
		The lifting category $\mathsf{LS}^{\mathrm{sm}}_{\mathcal{S}}$ consists of smooth weak formal log schemes over $\mathcal{S}$.
		The reduction functor is defined by $\cz \mapsto \cz\times_{\mathcal{S}}S$.
		The immersion class consists of all immersions in $\mathsf{LS}_S$.
		
	\item The realization category $\mathsf{Mod}_R$ is the category of $R$-modules, and the realization functor is defined by
		\begin{align*}
		&\mathsf{RT}_{\Xi(\mathcal{S})} \rightarrow  \mathscr{C}^+ (\mathsf{Mod}_R)\; ,\\ &(Y,\cz,i) \mapsto \mathcal{K}_\rig(Y,\cz,i): = \Gamma(]Y[_{\cz}^{\log},\Gd_{\an}\omega^{\sbt}_{\cz/\mathcal{S},\mathbb{Q}}).
		\end{align*}
		
	\item For a fine log scheme $Y/S$, the functor $P_Y\colon \mathsf{SET}^0_{\Xi(\mathcal{S})}(Y) \rightarrow \mathsf{RD}_{\Xi(\mathcal{S})}$ is defined by 
	$$P_Y(A) = (\cz_A,i_A): = (\prod_{\mathcal{S}}(\cz^{(a)})_{a\in \underline{A}},i_A),$$ 
	where $i_A\colon Y\hookrightarrow\prod_{\mathcal{S}}(\cz^{(a)})_{a\in \underline{A}}$ is the diagonal immersion.
	
	\item For a morphism of fine log schemes $g\colon Y \rightarrow  Y'$, $B\in\mathsf{SET}^0_{\Xi(\mathcal{S})}(Y)$, $C\in\mathsf{SET}^0_{\Xi(\mathcal{S})}(Y')$, $\tau\colon g^\circ(C) \rightarrow  B$, the morphism $g^\circ_\tau\colon (Y,\cz_B,i_B) \rightarrow (Y',\cz_C,i_C)$ is defined by $g\colon Y \rightarrow  Y'$ and the projection
			\[\prod_{\mathcal{S}}(\cz^{(b)})_{b\in\underline{B}} \rightarrow \prod_{\mathcal{S}}(\cz^{(c)})_{c\in \underline{g^\circ(C)}} = \prod_{\mathcal{S}}(\cz^{(c)})_{c\in\underline{C}}\]
		induced by $\underline{\tau}\colon\underline{B} \rightarrow \underline{g^\circ(C)} = \underline{C}$.
		
	\item For $(Y,\cz,i)\in\mathsf{RT}_{\Xi(\mathcal{S})}$ and $A\in\mathsf{SET}^0_{\Xi(\mathcal{S})}(Y,\cz,i)$, the morphism $D_A\colon(\cz,i) \rightarrow (\cz_A,i_A)$ is given by the diagonal morphism $\cz \rightarrow \prod_{\mathcal{S}}(\cz^{(a)})_{a\in\underline{A}} = \cz_A$ induced by $\{F^{(a)}\}_{a\in\underline{A}}$.
\end{enumerate}

We need to verify that Assumtions \ref{ass: weq} and \ref{ass: descent} hold in this scenario. But the former is just a reformulation of \cite[Lemma 1.4]{GrosseKloenne2005}, and the latter follows from the \v{C}ech descent of dagger spaces.
Moreover, the next lemma ensures that every object in $\mathsf{LS}_S$ is locally embeddable, that is $\mathsf{LS}_{S,\Xi(\mathcal{S})}^{\mathrm{loc}} = \mathsf{LS}_S$.

\begin{lemma}\label{lem: log rig embeddable}
	Let $Y$ be an object in $\lss$ which is affine and admits a (global) chart extending $c_S$ for the structure morphism $Y \rightarrow  S$.
	Then $Y$ is $\Xi(\mathcal{S})$-embeddable.
\end{lemma}

\begin{proof}
	Let $\so$, $\so_0$, and $A$ be the coordinate rings of $\mathcal{S}$, $S$, and $Y$, respectively, i.e. $\so = \so_F$, $\so_K$, or $\so_F[t]^\dagger$, and $\so_0 = k$ or $k[t]$.
	Denote by $\overline{\so}$ the reduction of $\so$ modulo $p$. 
	Assume that we have a chart $(\alpha\colon P_Y \rightarrow \cn_Y,\ \beta\colon \mathbb{N} \rightarrow  P)$ for the structure morphism $Y \rightarrow  S$ that extends $c_S$.
	Let $\alpha'\colon P \rightarrow  A$ be the map induced by $\alpha$ and $\epsilon\colon \mathbb{N} \rightarrow \so$  the map induced by $c_\cs\colon \mathbb{N}_\cs \rightarrow \cn_\cs$.
	Take surjections $\gamma\colon \so_0[\mathbb{N}^m] \rightarrow  A$ and $\delta\colon\mathbb{N}^n \rightarrow  P$.
	We set $\cz: = \Spwf \so[\mathbb{N}^m\oplus\mathbb{N}^n]^\dagger$ and endow $\cz$ with the log structure associated to the map
		\[\rho\colon\mathbb{N}^n\oplus\mathbb{N} \rightarrow  \so[\mathbb{N}^m\oplus\mathbb{N}^n]^\dagger,\ (\mathbf{n},\ell) \mapsto \epsilon(\ell)\cdot\mathbf{n}.\]
	Let $\eta\colon \so \rightarrow  \overline{\so}$ be the natural surjection.
	Then we have a commutative diagram
		\[\xymatrix{
		\so \ar[r]&\so[\mathbb{N}^m\oplus\mathbb{N}^n]^\dagger \ar[r]^>>>>>\xi& A\\
		\mathbb{N} \ar[r] \ar[u]^\epsilon&\mathbb{N}^n\oplus\mathbb{N} \ar[r]_{(\delta,\beta)} \ar[u]^\rho& P \ar[u]_{\alpha'},
		}\]
	where the map $\xi$ is defined by
		\[a\cdot\mathbf{m}\cdot\mathbf{n} \mapsto \eta(a)\cdot\gamma(\mathbf{m}) \cdot\alpha'\circ\delta(\mathbf{n})\]
	for $a\in \so$, $\mathbf{m}\in\mathbb{N}^m$, and $\mathbf{n}\in\mathbb{N}^n$.
	The horizontal arrows in the right square are surjective, and hence define a closed immersion $i\colon Y\hookrightarrow\cz$.
	The left square defines the structure morphism $\cz \rightarrow \cs$.
	The smoothness criterion \cite[Thm.~3.5]{Kato1989} implies immediately that $\cz$ is smooth over $\mathcal{S}$.
\end{proof}

Thus one can define the  {\it log rigid cohomology} of $Y$ over $\mathcal{S}$  via the realization functor given above as
$$\mathbb{R}\Gamma_{\rig}(Y/\mathcal{S}): = \mathbb{R}\Gamma_{\Xi(\mathcal{S})}(Y).$$

\subsubsection{Rigid cohomology of fine log schemes with boundary}

In this section we introduce two variations of rigid cohomology for fine log schemes with boundary. We start by defining a  rigid cohomological tuple $\Xi_b(\ct)$ for fine $T$-log schemes with boundary.

\begin{enumerate}
	\item The base site $\overline{\mathsf{LS}}_T$ consists of  fine $T$-log schemes with boundary, and is endowed with the Zarsiki topology.
		The lifting category $\overline{\mathsf{LS}}^{\mathrm{sm}}_{\ct}$ consists of strongly smooth weak formal $\ct$-log schemes with boundary.
		The reduction functor is defined by $(\cz,\ovcz) \mapsto (\cz\times_{\so_F^\varnothing}k^\varnothing,\ovcz\times_{\so_F^\varnothing}k^\varnothing)$.
		The immersion class consists of all boundary immersions.
		
	\item The realization category $\mathsf{Mod}_F$ consists of  $F$-vector spaces, and the realization functor is defined by
	\begin{align*}
	&\mathsf{RT}_{\Xi_b(\ct)} \rightarrow  \mathscr{C}^+ (\mathsf{Mod}_F)\\
	&((Y,\ovy),(\cz,\ovcz),i) \mapsto  \mathcal{K}_\rig((Y,\ovy),(\cz,\ovcz),i):=\Gamma(]\ovy[_{\ovcz}^{\log},\Gd_{\an}\omega^{\sbt}_{(\cz,\ovcz)/\ct,\mathbb{Q}}).
	\end{align*}

	\item For a fine $T$-log scheme with boundary $(Y,\ovy)$, the functor
		\[P_{(Y,\ovy)}\colon\mathsf{SET}^0_{\Xi_b(\ct)}(Y,\ovy) \rightarrow \mathsf{RD}_{\Xi_b(\ct)}(Y,\ovy)\]
		 is defined by
			\[P_{(Y,\ovy)}(A) = ((\cz_A,\ovcz_A),i_A): = (\prod_{\ct}((\cz^{(a)},\ovcz^{(a)}))_{a\in\underline{A}},i_A),\]
		where $i_A\colon(Y,\ovy)\hookrightarrow\prod_{\ct}((\cz^{(a)},\ovcz^{(a)}))_{a\in\underline{A}}$ is the diagonal immersion.
		Note that $(\cz_A,\ovcz_A)$ is strongly smooth by Lemma \ref{lem: product of strongly smooth}.
		
	\item For a morphism of fine  $T$-log schemes with boundary $g\colon(Y,\ovy) \rightarrow (Y',\ovy')$, $B\in\mathsf{SET}^0_{\Xi_b(\ct)}(Y,\ovy)$, $C\in\mathsf{SET}^0_{\Xi_b(\ct)}(Y',\ovy')$, $\tau\colon g^\circ(C) \rightarrow  B$, the morphism 
	$$g^\circ_\tau\colon((Y,\ovy),(\cz_B,\ovcz_B),i_B) \rightarrow ((Y',\ovy'),(\cz_C,\ovcz_C),i_C)$$
	 is defined by $g\colon(Y,\ovy) \rightarrow (Y',\ovy')$ and the projection
			\[\prod_{\ct}((\cz^{(b)},\ovcz^{(b)}))_{b\in\underline{B}} \rightarrow \prod_{\ct}((\cz^{(c)},\ovcz^{(c)}))_{c\in\underline{g^\circ(C)}} = \prod_{\ct}((\cz^{(c)},\ovcz^{(c)}))_{c\in\underline{C}}\]
		induced by $\underline{\tau}\colon\underline{B} \rightarrow \underline{g^\circ(C)} = \underline{C}$.
		
	\item For $((Y,\ovy),(\cz,\ovcz),i)\in\mathsf{RT}_{\Xi_b(\ct)}$ and $A\in\mathsf{SET}^0_{\Xi_b(\ct)}((Y,\ovy),(\cz,\ovcz),i)$, the morphism 	
	$$D_A\colon((\cz,\ovcz),i) \rightarrow ((\cz_A,\ovcz_A),i_A)$$
 is given by the diagonal morphism $(\cz,\ovcz) \rightarrow \prod_{\ct}((\cz^{(a)},\ovcz^{(a)}))_{a\in\underline{A}} = (\cz_A,\ovcz_A)$ induced by $\{F^{(a)}\}_{a\in\underline{A}}$.
\end{enumerate}

Then Assumption \ref{ass: weq} holds again by the same argument as \cite[Lemma 1.4]{GrosseKloenne2005} and Assumption \ref{ass: descent} follows from the \v{C}ech descent of dagger spaces.
Thus it makes sense to define
	\[\mathbb{R}\Gamma_\rig((Y,\ovy)/\ct): = \mathbb{R}\Gamma_{\Xi_b(\ct)}(Y,\ovy)\]
for any $(Y,\ovy)\in\overline{\mathsf{LS}}_{T,\Xi_b(\ct)}^{\mathrm{loc}}$, which we call the {\it log rigid cohomology} of $(Y,\ovy)$ over $\ct$.

The authors do not know whether any given $T$-log scheme with boundary is locally $\Xi_b(\ct)$-embeddable, but it will later be used that they are locally of the following form.

\begin{proposition}\label{prop: embeddable boundary}
	Let $(Z,\ovz)\in\overline{\mathsf{LS}}_T$.
	Assume that $\ovz = \Spec A$ is affine, and that there exist a chart $(\alpha\colon P_{\ovz} \rightarrow \cn_{\ovz},\beta\colon\mathbb{N} \rightarrow  P^\gp)$ extending $c_T$ and elements $a,b\in P$ such that $\beta(1) = b-a$ and $Z = \Spec A[\frac{1}{\alpha'(a)}]$, where $\alpha'\colon P \rightarrow  A$ is the map induced by $\alpha$.
	Then $(Z,\ovz)$ is $\Xi_b(\ct)$-embeddable.
	\end{proposition}
	
\begin{proof}
	Indeed, take surjections $\gamma\colon k[\mathbb{N}^m] \rightarrow  A$ and $\delta\colon \mathbb{N}^n \rightarrow  P$.
	We set $Q: = \mathbb{N}^n\oplus\mathbb{N}$ and $B: = \so_F[\mathbb{N}^m\oplus\mathbb{N}^n\oplus\mathbb{N}]^\dagger$, and let $\epsilon\colon Q \rightarrow  B$ be the canonical injection.
	We define $\rho\colon B \rightarrow  A$ by 
	\[\rho(x\cdot \mathbf{m}\cdot\mathbf{n}\cdot \ell): = \overline{x}\cdot \gamma(\mathbf{m})\cdot\alpha\circ\delta(\mathbf{n})\]
	for $x\in\so_F$, $\mathbf{m}\in\mathbb{N}^m$, $\mathbf{n}\in\mathbb{N}^n$, and $\ell\in\mathbb{N}$.
	Take $a',b'\in\mathbb{N}^n$ with $\epsilon(a') = a$ and $\epsilon(b') = b$, and let $\widetilde{a}: = (a',0)\in Q$, $\widetilde{b}: = (b',0)\in Q$.
	Let $Q'$ be the submonoid of $Q^\gp$ generated by $Q$ and $-\widetilde{a}$, and let $\widetilde{t}: = (b'-a',1)\in Q'$.
	We endow $\ovcz: = \Spwf B$ and $\cz: = \Spwf B[\frac{1}{\epsilon(\widetilde{a})}]^\dagger$ with the log structures associated to $Q \xrightarrow{\epsilon}B$ and $Q' \rightarrow  B[\frac{1}{\epsilon(\widetilde{a})}]^\dagger$.
	We define a map $\eta\colon\mathbb{N} \rightarrow  Q'$ by $1 \mapsto \widetilde{t}$.
	Then the commutative diagram
	\[\xymatrix{
	Q \ar[r]^{\epsilon} \ar[d]_{(\delta,0)}&B \ar[d]^\rho\\
	P \ar[r]^{\alpha}&A
	}\]
	induces closed immersions $\ovz\hookrightarrow\ovcz$ and $(Z,\ovz)\hookrightarrow(\cz,\ovcz)$.
	Obviously we have $\ker\eta^\gp = 0$ and $\coker\eta^\gp = \mathbb{Z}^n$ for $\eta^\gp\colon \mathbb{Z} \rightarrow  Q'^\gp = \mathbb{Z}^n\oplus\mathbb{Z}$.
	Thus $(\cz,\ovcz)$ is a strongly smooth $\ct$-log scheme with boundary.
\end{proof}

We introduce now a variant of log rigid cohomology for log schemes with boundary, which will be used to construct a canonical Frobenius endomorphism.
Let $\ovt = \Proj k[x_1,x_2]$ and $\ovct = \Proj \so_F[x_1,x_2]^\dagger$ be the fine log scheme and fine weak formal log scheme with the log structures associated to the points $x_1 = 0$ and $x_2 = 0$.
Take embeddings $T\hookrightarrow\ovt$ and $\ovt\hookrightarrow\ovct$ defined by the equality $t = \frac{x_2}{x_1}$.
Then $(T,\ovt)$ and $(\ct,\ovct)$ are strongly smooth objects in $\lsbt$ and $\lsbct$, respectively.
In this setup, we define a rigid cohomological tuple $\Xi_b(\ct,\ovct)$ as follows:

\begin{enumerate}
	\item The base site $\overline{\mathsf{LS}}_{(T,\ovt)}$ is the full subcategory of $\overline{\mathsf{LS}}_T$ which consists of objects $(Y,\ovy)$ whose structure morphism $Z \rightarrow  T$ extends to $\ovz \rightarrow \ovt$, and the topology on $\overline{\mathsf{LS}}_{(T,\ovt)}$ is the Zariski topology.
		The lifting category $\overline{\mathsf{LS}}_{(\ct,\ovct)}^{\mathrm{sm}}$ is the full subcategory of $\overline{\mathsf{LS}}_{\ct}^{\mathrm{sm}}$ consisting of objects $(\cz,\ovcz)$ whose structure morphism $\cz \rightarrow \ct$ extends to $\ovcz \rightarrow \ovct$.
		Note that such extensions are unique if they exist.
	\item The reduction functor and the immersion class, the realization category, and the realization functor, the functors $P_{(Y,\ovy)}$, the morphisms $g^\circ_\tau$ and $D_A$ are defined in the same way as for $\Xi_b(\ct)$.
\end{enumerate}

We define
	\[\mathbb{R}\Gamma_{\rig}((Y,\ovy)/(\ct,\ovct)): = \mathbb{R}\Gamma_{\Xi_b(\ct,\ovct)}(Y,\ovy)\]
for any $(Y,\ovy)\in\overline{\mathsf{LS}}^{\mathrm{loc}}_{(T,\ovt),\Xi_b(\ct,\ovct)}$, which we call the {\it log rigid cohomology} of $(Y,\ovy)$ over $(\ct,\ovct)$.

\subsubsection{Rigid Hyodo--Kato cohomology}\label{subsubsec: Rigid Hyodo-Kato cohomology}

Finally we define a rigid cohomological tuple $\Xi_\hk^\rig$.

\begin{definition}\label{def: ss log scheme}
	\begin{enumerate}
	\item 		A $k^0$-log scheme with boundary $(Y,\overline{Y})$ is called {\it strictly semistable} if Zariski locally on $\ovy$ there exists a chart $(\alpha\colon P_{\ovy} \rightarrow \cn_{\ovy},\ \beta\colon\N \rightarrow  P^\gp)$ for $(Y,\ovy)$ extending $c_{k^0}$ of the following form:
	 	\begin{itemize}
		\item $P = \N^m\oplus\N^n$ for some integers $m\geqslant 1$ and $n\geqslant 0$, and $\beta$ is given by the composition of the diagonal map $\N \rightarrow \N^m$ and the canonical injection $\N^m \rightarrow \Z^m\oplus\Z^n$.
		
		\item The morphism of schemes
				\begin{align*}
				\ovy \rightarrow& \Spec k\times_{\Spec k[\N]}\Spec k[\N^m\oplus\N^n] \\
				&\qquad \qquad= \Spec k[t_1,\ldots,t_m,s_1,\ldots,s_n]/(t_1\cdots t_m)
				\end{align*}
			induced by $\beta'$ is smooth, and makes the diagram 
				\[\xymatrix{
				\overline{Y} \ar[r]&\Spec k[t_1,\ldots,t_m,s_1,\ldots,s_n]/(t_1\cdots t_m)\\
				Y \ar[u] \ar[r]&\Spec k[t_1,\ldots,t_m,s_1^{\pm 1},\ldots,s_n^{\pm 1}]/(t_1\cdots t_m). \ar[u]
				}\]
			Cartesian.
		\end{itemize}
		We denote the category of strictly semistable $k^0$-log schemes with boundary by $\ssk$.
		In addition we regard $\ssk$ as a site with respect to the Zariski topology.
		Namely a covering family of an object $(Y,\ovy)$ in $\ssk$ is a family of boundary strict open immersions $(U_h,\overline{U}_h)\hookrightarrow(Y,\ovy)$ with $\bigcup_h\overline{U}_h = \ovy$.
	\item	A weak formal $\ct$-log scheme with boundary $(\cz,\ovcz)$ is called {\it strictly semistable} if  Zariski locally on $\ovcz$ there exists a chart $(\alpha\colon P_{\ovcz} \rightarrow  \cn_{\ovcz},\ \beta\colon \N \rightarrow  P^\gp)$ for $(\cz,\ovcz)$ extending $c_\ct$ of the following form:
		\begin{itemize}
		\item $P = \N^m\oplus\N^n$ for some integers $m\geqslant 1$ and $n\geqslant 0$, and $\beta$ is given by the composition of the diagonal map $\N \rightarrow \N^m$ and the canonical injection $\N^m \rightarrow \Z^m\oplus\Z^n$.
		\item The morphism of weak formal schemes
				\[\ovcz \rightarrow \Spwf\so_F[\N^m\oplus\N^n]^\dagger = \Spwf\so_F[t_1,\ldots,t_m,s_1,\ldots,s_n]^\dagger\]
			induced by $\beta'$ is smooth, and makes the diagram 
				\[\xymatrix{
				\ovcz \ar[r]&\Spwf\so_F[t_1,\ldots,t_m,s_1,\ldots,s_n]^\dagger\\
				\cz \ar[u] \ar[r]&\Spwf\so_F[t_1,\ldots,t_m,s_1^{\pm 1},\ldots,s_n^{\pm 1}]^\dagger. \ar[u]
				}\]
			Cartesian.
		\end{itemize}
		We denote the category of strictly semistable weak formal $\ct$-log schemes with boundary by $\sst$.
	\end{enumerate}
\end{definition}

\begin{remark}
	For a strictly semistable $k^0$-log scheme with boundary $(Y,\ovy)$, the structure morphism $Y \rightarrow  k^0$ extends to a morphism $\ovy \rightarrow  k^0$ with a chart $\beta'\colon \N \rightarrow \N^m\oplus\N^n = P$, where $P$ is a local chart for $\ovy$ as in Definition \ref{def: ss log scheme}.
	Note that this does not hold for $k^0$-log schemes with boundary in general.
	Similarly, for a strictly semistable weak formal $\ct$-log scheme with boundary $(\cz,\ovcz)$, the structure morphism $\cz \rightarrow \ct$ extends to a morphism $\ovcz \rightarrow \ct$ with a chart $\beta'\colon \N \rightarrow \N^m\oplus\N^n = P$.
\end{remark}

A strictly semistable log scheme in the sense of \cite[\S 2.1]{GrosseKloenne2005} can be regarded as a strictly semistable $k^0$-log scheme with boundary with $Y = \ovy$.
Note that for an object $(Y,\ovy)$ in $\ssk$, the log schemes $Y$ and $\ovy$ are log smooth over $k^0$. For an object $(\cz,\ovcz)$ in $\sst$, the weak formal log schemes $\cz$ and $\ovcz$ are log smooth over $\ct$.

\begin{definition}\label{def: upsilon}
	\begin{enumerate}
	\item Let $(Y,\overline{Y})$ be an object in $\ssk$ and $D: = \ovy\setminus Y$ the reduced closed subscheme, which is a simple normal crossing divisor.
		A {\it horizontal component} of $D$ is a Cartier divisor on $\ovy$ whose support is contained in $D$, and which can not be written as the sum of two non-trivial Cartier divisors.
		We denote by $\Upsilon_{\ovy}$ the set of irreducible components of $\ovy$, and by $\Upsilon_D$ the set of horizontal components of $D$.
		Note that each horizontal component is reduced, and that we have $D = \sum_{\beta\in\Upsilon_D}D_\beta$, where $D_\beta$ is the Cartier divisor corresponding to $\beta$.
	\item Let $(\cz,\ovcz)$ be an object in $\sst$, and let $\cd: = \ovcz\setminus\cz$ be the reduced closed  admissible weak formal subscheme, which is a relative simple normal crossing divisor over $\ct$. Set $\ovcy: = \ovcz\times_{\ct}\so_F^0$.
		We denote by $\Upsilon_{\ovcy}$ (resp. $\Upsilon_{\cd}$) the set of irreducible components of $\ovcy$ (resp. $\cd$). 
		\end{enumerate}
\end{definition}

\begin{definition}
	For a strictly semistable weak formal $\ct$-log scheme with boundary $(\cz,\ovcz)$, a {\it Frobenius lifting} on $(\cz,\ovcz)$ is an endomorphism $\phi$ on $(\cz,\ovcz)$ which lifts the $p$-th power Frobenius on the reduction $\overline{Z}: = \ovcz\times_\ct T$, is compatible with $\sigma$ on $\ct$, and sends the equations of all irreducible components of $\ovcy: = \ovcz \times_{\ct}\so_F^0$ and $\cd: = \ovcz\setminus\cz$ to their $p$-th power up to units.
\end{definition}

We come now to the definition of the rigid cohomological tuple $\Xi_\hk^\rig$.

\begin{enumerate}
	\item[(i)] The base site $\overline{\mathsf{LS}}_{k^0}^{\mathrm{ss}}$ consists of  strictly semistable $k^0$-log schemes with boundary, and is endowed with the Zariski topology.
		The lifting category $\overline{\mathsf{LS}}_{\ct}^{\mathrm{ss}}(\varphi)$ consists of triples $(\cz,\ovcz,\phi)$, where $(\cz,\ovcz)$ is a strictly semistable weak formal $\ct$-log scheme with boundary $(\cz,\ovcz)$ and $\phi$ is a Frobenius lifting on $(\cz,\ovcz)$.
		The reduction functor is defined by $(\cz,\ovcz,\phi) \mapsto (\cz\times_{\ct}k^0,\ovcz\times_{\ct}k^0)$.
		The immersion class consists of all boundary strict immersions.
\end{enumerate}

For an object $(\cz,\ovcz)$ in $\sst$, let
	$$
	\widetilde{\omega}_{\ovcz}^{\sbt}: = \omega^{\sbt}_{\ovcz/\so_F^\varnothing}
	\quad\text{ and }\quad
	\omega^{\sbt}_{\ovcy}: = \omega^{\sbt}_{\ovcy/\so_F^0}
	$$
be the log de~Rham complexes of $\ovcz$ over $\so_F^\varnothing$ and $\ovcy: = \ovcz\times_{\ct}\so_F^0$ over $\so_F^0$, respectively.
Recall that in the first case we consider $\so_F$ with the trivial log structure, while in the second case we consider the log structure associated to the map $\mathbb{N} \rightarrow \so_F,\ 1 \mapsto  0$.
Set
	\[\widetilde{\omega}^{\sbt}_{\ovcy}: = \widetilde{\omega}^{\sbt}_{\ovcz}\otimes\mathcal{O}_{\ovcy}\]
Then there is a short exact sequence
	\begin{equation}\label{eq: rigid ses}
	0 \rightarrow \omega^{\sbt}_{\ovcy}[-1] \xrightarrow[]{\wedge d\log t}\widetilde{\omega}^{\sbt}_{\ovcy} \rightarrow \omega^{\sbt}_{\ovcy} \rightarrow  0.
	\end{equation}
Let $\ovcy_{\Q}$ be the generic fibre of $\ovcy$, and $\omega^{\sbt}_{\ovcy,\Q}$ (resp. $\widetilde{\omega}^{\sbt}_{\ovcy,\Q}$) the complex of sheaves on $\ovcy_{\Q}$ obtained from $\omega^{\sbt}_{\ovcy}$ (resp. $\widetilde{\omega}^{\sbt}_{\ovcy}$) by tensoring with $\Q$.

To be able to consider a cup product we use  an analogue of a construction due to  Kim and Hain \cite[pp. 1259--1260]{KimHain2004}.

\begin{definition}\label{def: Kim - Hain}
	For an object $(\cz,\ovcz,\phi)$  in $\overline{\mathsf{LS}}_{\ct}^{\mathrm{ss}}(\varphi)$, let $\ovcy$, $\omega_{\ovcy,\Q}^{\sbt}$, and $\widetilde{\omega}_{\ovcy,\Q}^{\sbt}$ be as above.
	Let $\widetilde{\omega}_{\ovcy,\Q}^{\sbt}[u]$ be the commutative dg-algebra on $\ovcy_\Q$ generated by $\widetilde{\omega}_{\ovcy,\Q}^{\sbt}$ and degree zero elements $u^{[k]}$ for $k\geqslant 0$ with the relations
		$$
		du^{[k]} = d\log t\cdot u^{[k-1]}\quad\text{ and }\quad u^{[0]} = 1.
		$$
	The multiplication is given by
		\[u^{[k]}\wedge u^{[\ell]} = \frac{(k + \ell)!}{k!\ell!}u^{[k + \ell]}.\]
	Let $\widetilde{\omega}^{\sbt}_{\ovcy,\Q}\llbracket u\rrbracket:=\varprojlim_i\widetilde{\omega}^{\sbt}_{\ovcy,\Q}[u]/(u^{[i]})$ be the completion of $\widetilde{\omega}^{\sbt}_{\ovcy,\Q}[u]$ with respect to the ideal generated by $u^{[1]}$, which is also a commutative dg-algebra in obvious way.
	We define the Frobenius operator $\varphi$ on $\widetilde{\omega}_{\ovcy,\Q}^{\sbt}\llbracket u\rrbracket$ by the Frobenius action on $\widetilde{\omega}_{\ovcy,\Q}^{\sbt}$ induced by $\phi$ and by $\varphi(u^{[k]}): = p^ku^{[k]}$.
	We define the monodromy operator $N$ on $\widetilde{\omega}_{\ovcy,\Q}^{\sbt}\llbracket u\rrbracket$ by $N(u^{[k]}): = u^{[k-1]}$.
	Then clearly $N\varphi = p\varphi N$ holds.
	Moreover we have $\varphi(\eta_1\wedge\eta_2) = \varphi(\eta_1)\wedge\varphi(\eta_2)$ and $N(\eta_1\wedge\eta_2) = N(\eta_1)\wedge\eta_2 + \eta_1\wedge N(\eta_2)$ for any sections $\eta_1$ and $\eta_2$ of $\widetilde{\omega}_{\ovcy,\Q}^{\sbt}\llbracket u\rrbracket$.
\end{definition}

\begin{lemma}\label{lem: qis with KimHain}
	The morphism $\widetilde{\omega}_{\ovcy,\Q}^{\sbt}\llbracket u\rrbracket \rightarrow \omega^{\sbt}_{\ovcy,\Q}$ which sends $u^{[k]} \mapsto  0$ for $k\geqslant 1$ is a quasi-isomorphism of dg-algebras.
\end{lemma}

\begin{proof}
	The compatibility of this morphism with the differentials and the multiplications are straightforward.
	Note moreover, that as complex, $\widetilde{\omega}_{\ovcy,\Q}^{\sbt}\llbracket u\rrbracket$ is given as the product-total complex of the double complex $(\widetilde{\omega}_{\ovcy,\Q}^{i + j}u^{[-j]})_{i,j}$, with $u^{[-j]} = 0$ for $j> 0$, and  horizontal differential $\partial_1$ and  vertical differential $\partial_2$ defined by 
		\begin{align*}
		&\partial_1^{i,j}\colon\widetilde{\omega}_{\ovcy,\Q}^{i + j}u^{[-j]} \rightarrow \widetilde{\omega}_{\ovcy,\Q}^{i + j + 1}u^{[-j]},\ \eta\cdot u^{[-j]} \mapsto (d\eta)\cdot u^{[-j]}&&(\eta\in\widetilde{\omega}_{\ovcy,\Q}^{i + j})\\
		&\partial_2^{i,j}\colon\widetilde{\omega}_{\ovcy,\Q}^{i + j}u^{[-j]} \rightarrow \widetilde{\omega}_{\ovcy,\Q}^{i + j + 1}u^{[-j-1]},\ \eta\cdot u^{[-j]} \mapsto (-1)^{i+1}d\log t\wedge\eta \cdot u^{[-j-1]}&&(\eta\in\widetilde{\omega}_{\ovcy,\Q}^{i + j})
		\end{align*}
	By the short exact sequence \eqref{eq: rigid ses}, we see that the sequence
		\[\cdots \rightarrow \widetilde{\omega}^{k-2}_{\ovcy,\Q}u^{[2]} \xrightarrow{\partial_2^{-2,k}}\widetilde{\omega}^{k-1}_{\ovcy,\Q}u^{[1]} \xrightarrow{\partial_2^{-1,k}}\widetilde{\omega}^k_{\ovcy,\Q}u^{[0]} \rightarrow \omega^k_{\ovcy,\Q} \rightarrow  0\]
	is exact for any $k\geqslant 0$.
	Hence the $k$-th column of $(\widetilde{\omega}_{\ovcy,\Q}^{i + j}u^{[-j]})_{i,j}$ is quasi-isomorphic to $\omega^k_{\ovcy,\Q}$.
	By the acyclic assembly lemma \cite[Lemma 2.7.3]{Weibel1994}, this shows the claim of  the lemma.
\end{proof}

\begin{remark}\label{rem: Kim Hain}
	\begin{enumerate}
	\item Since the double complex $(\widetilde{\omega}_{\ovcy,\Q}^{i + j}u^{[-j]})_{i,j}$ is in the fourth quadrant, it does not follow from formal arguments that $\widetilde{\omega}^{\sbt}_{\ovcy,\Q}[u]\rightarrow \omega^{\sbt}_{\ovcy,\Q}$ is a quasi-isomorphism.
		
		 The authors of \cite{KimHain2004}  define the Hyodo--Kato complex via the complex $W\widetilde{\omega}[u]$, which is the de Rham--Witt counterpart of our complex $\widetilde{\omega}^{\sbt}_{\ovcy,\Q}[u]$.
		However the proof of \cite[Lemma 7]{KimHain2004} only shows that the map $W\widetilde{\omega}\llbracket u\rrbracket\rightarrow W\omega$ is a quasi-isomorphism, where $W\widetilde{\omega}\llbracket u\rrbracket$ is the completion of $W\widetilde{\omega}[u]$ with respect to the ideal generated by $u^{[1]}$.
		We solve this issue in our subsequent paper.
		
	\item Note that $\widetilde{\omega}^{\sbt}_{\ovcy,\Q}\llbracket u\rrbracket$ does not have a weight filtration.
	By an analogous construction as in \cite[p. 1270]{KimHain2004}, we may also define a commutative dg-algebra with $\varphi$, $N$, and the weight filtration, which is quasi-isomorphic to $\widetilde{\omega}^{\sbt}_{\ovcy,\Q}\llbracket u\rrbracket$.
	Another possibility is to construct a  complex of $(\varphi,N)$-modules via a Steenbrink double complex as in \cite[\S 5]{GrosseKloenne2005}.
	It is quasi-isomorphic to $\omega^{\sbt}_{\ovcy,\Q}$ as well, and in addition has a weight filtration.
	However the wedge product is not well-defined there.
	\end{enumerate}
\end{remark}

We continue in our definition of the tuple $\Xi_\hk^\rig$.

\begin{enumerate}
	\item[(ii)] The realization category $\mathsf{Mod}_F(\varphi,N)$ consists of $(\varphi,N)$-modules over $F$.
		The realization functor is defined by
			\begin{align*}&\mathsf{RT}_{\Xi_\hk^\rig} \rightarrow \mathcal{C}^+(\mathsf{Mod}_F(\varphi,N))\;,\\
			&((Y,\ovy),(\cz,\ovcz,\phi),i) \mapsto \mathcal{K}_\hk((Y,\ovy),(\cz,\ovcz,\phi),i):=\Gamma(]\ovy[_{\ovcy}^{\log},\Gd_{\an}\widetilde{\omega}^{\sbt}_{\ovcy,\Q}\llbracket u\rrbracket).
			\end{align*}
\end{enumerate}

We discuss now the requirements and set-up for the definition of the functor $P_{(Y,\ovy)}$
for $(Y,\ovy)\in\overline{\mathsf{LS}}_{k^0,\Xi_\hk^\rig}^{\mathrm{em}}$.
Let $A\in\mathsf{SET}_{\Xi_\hk^\rig}^0(Y,\ovy)$, and fix an element $a_0\in \underline{A}$ with $f^{(a_0)} = \mathrm{id}_{(Y,\ovy)}$.
For each $a\in\underline{A}$ and $\alpha\in\Upsilon_{\ovcy^{(a)}}$, let
	\[\Upsilon_{\ovy}^{a,\alpha}: = \{\gamma\in\Upsilon_{\ovy}\mid (i^{(a)}\circ f^{(a)})(\ovy_\gamma)\subset\ovcy^{(a)}_\alpha\},\]
where $\ovy_\gamma$ and $\ovcy^{(a)}_\alpha$ are the irreducible components corresponding $\gamma$ and $\alpha$, respectively.
Since $(i^{(a)}\circ f^{(a)})(Y)\subset \cz^{(a)}$, we may consider the pull-back $(i^{(a)}\circ f^{(a)})^{\ast}\cd^{(a)}$ of the Cartier divisor $\cd^{(a)}$.
For each $\beta\in\Upsilon_{\cd^{(a)}}$ and $\delta\in\Upsilon_{D}$, let $m_{\beta,\delta}^{(a)}$ be the multiplicity of $(i^{(a)}\circ f^{(a)})^{\ast}\cd_\beta^{(a)}$ at $D_\delta$, namely we have
	\[(i^{(a)}\circ f^{(a)})^{\ast}\cd^{(a)}_\beta = \sum_{\delta\in\Upsilon_D}m_{\beta,\delta}^{(a)}D_\delta.\]
Now we note that there are natural inclusions
	\begin{align*}
	&\Upsilon_{\ovy^{(a)}}\subset\Upsilon_{\ovcy^{(a)}},&
	&\Upsilon_{\ovy}^{a,\alpha}\subset\Upsilon_{\ovy} = \Upsilon_{\ovy^{(a_0)}}\subset\Upsilon_{\ovcy^{(a_0)}}, \\
	&\Upsilon_{D^{(a)}}\subset\Upsilon_{\cd^{(a)}},&
	&\Upsilon_D = \Upsilon_{D^{(a_0)}}\subset\Upsilon_{\cd^{(a_0)}}.
	\end{align*}
Let $\ovcz''_A$ be the blow-up of $\prod_{\so_F}(\ovcz^{(a)})_{a\in\underline{A}}$ along the ideal
	\begin{equation*}
	\prod_{\substack{a\in\underline{A}\\ \alpha\in\Upsilon_{\ovcy^{(a)}}}}(\mathcal{I}^{(a)}_{\alpha} + \prod_{\gamma\in\Upsilon_{\ovy}^{a,\alpha}}\mathcal{I}^{(a_0)}_{\gamma})
	\times\prod_{\substack{a\in\underline{A}\\ \beta\in\Upsilon_{\cd^{(a)}}}}(\mathcal{J}^{(a)}_{\beta}  +  \prod_{\delta\in\Upsilon_D}\mathcal{J}_\delta^{(a_0),m^{(a)}_{\beta,\delta}}),
	\end{equation*}
where $\mathcal{I}^{(a)}_{\alpha}$ and $\mathcal{J}^{(a)}_{\beta}$ are the ideals of
	\begin{align*}
	&\ovcy^{(a)}_{\alpha}\times\prod_{\so_F}(\ovcz^{(b)})_{b\in\underline{A}\setminus\{a\}}\subset\prod_{\so_F}(\ovcz^{(a)})_{a\in\underline{A}}\qquad\text{and}\\
	&\cd^{(a)}_{\beta}\times\prod_{\so_F}(\ovcz^{(b)})_{b\in\underline{A}\setminus\{a\}}\subset\prod_{\so_F}(\ovcz^{(a)})_{a\in\underline{A}}&&
	\end{align*}
respectively.
Let $\ovcz'_A$ be the complement of the strict transforms in $\ovcz''_A$ of $\ovcy^{(a)}_\alpha\times\prod_{\so_F}(\ovcz^{(b)})_{b\in\underline{A}\setminus\{a\}}$ and $\cd^{(a)}_\beta\times\prod_{\so_F}(\ovcz^{(b)})_{b\in\underline{A}\setminus\{a\}}$ for all $a\in\underline{A}$, $\alpha\in\Upsilon_{\ovcy^{(a)}}$, and $\beta\in\Upsilon_{\cd^{(a)}}$.
Let $\ct_A$ be the blow-up of the $\underline{A}$-indexed self product $\prod_{\so_F}(\ct)_{a\in\underline{A}}$ of $\ct$ along the closed weak formal subscheme $\prod_{\so_F}(\Spwf\so_F)_{a\in\underline{A}}$ defined by $t = 0$.
Then the diagonal embedding $\ct\hookrightarrow\prod_{\so_F}(\ct)_{a\in\underline{A}}$ lifts to an embedding $\ct\hookrightarrow\ct_A$, and there exists a natural morphism $\ovcz'_A \rightarrow \ct_A$.
Let $\ovcz_A: = \ovcz'_A\times_{\ct_A}\ct$.
Let $\mathcal{E}_A$ be the divisor on $\ovcz_A$ defined by $t = 0$ with respect to the natural morphism $\ovcz_A \rightarrow \ct$.
We denote by $\widetilde{\cd}_A$ the closed weak formal subscheme of $\ovcz_A$ defined by the inverse image of
	\begin{align*}
	\prod_{\substack{a\in\underline{A}\\ \beta\in\Upsilon_{\cd^{(a)}}}}(\mathcal{J}^{(a)}_{\beta} + \prod_{\delta\in\Upsilon_D} \mathcal{J}_\delta^{(a_0),m^{(a)}_{\beta,\delta}}).
	\end{align*}
 Let $\cd_A$ be the admissible reduced closed weak formal subscheme of $\ovcz_A$ whose underlying spaces is $\widetilde{\cd}_A$. 
 Let $\cz_A: = \ovcz_A\setminus\cd_A$.

\begin{proposition}\label{prop: datum for A}
	The closed weak formal subschemes $\mathcal{E}_A$ and $\cd_A$ of $\ovcz_A$ are simple normal crossing divisors.
	The diagonal morphism $\prod i^{(a)}\circ f^{(a)}\colon\ovy \rightarrow \prod_{\so_F}(\ovcz^{(a)})_{a\in\underline{A}}$ factors into a strict immersion $i_A\colon\ovy \rightarrow \ovcz_A$ and the natural morphism $\ovcz_A \rightarrow \prod_{\so_F}(\ovcz^{(a)})_{a\in\underline{A}}$.
	Furthermore $i_A$ provides a boundary strict immersion $(Y,\ovy)\hookrightarrow(\cz_A,\ovcz_A)$.
	Moreover $(\cz_A,\ovcz_A)$ is independent of the choice of $a_0$ up to canonical isomorphism and functorial on $A$.
\end{proposition}

\begin{proof}
	Since the construction is local, we may assume that there are commutative diagrams 
		\begin{equation*}\xymatrix{
		\ovy^{(a)} \ar[r] \ar[d]^{i^{(a)}}&\Spec\frac{k[t_\alpha^{(a)},t_{\alpha'}^{(a),-1} s_\beta^{(a)}, s_{\beta'}^{(a),-1}]_{\alpha\in\Upsilon_{\ovcy^{(a)}},\alpha'\in\Upsilon_{\ovcy^{(a)}}\setminus\Upsilon_{\ovy^{(a)}},\beta\in\Upsilon_{\cd^{(a)}},\beta'\in\Upsilon_{\cd^{(a)}} \setminus\Upsilon_{D^{(a)}}}} {(\prod_{\alpha\in\Upsilon_{\ovcy^{(a)}}}t_\alpha^{(a)})} \ar[d]\\
		\ovcz^{(a)} \ar[r]&\Spwf\so_F[t_\alpha^{(a)}, s_\beta^{(a)}]^\dagger_{\alpha\in\Upsilon_{\ovcy^{(a)}},\beta\in\Upsilon_{\cd^{(a)}}}
		}\end{equation*}
	whose horizontal arrows are smooth as in Definition~\ref{def: ss log scheme}.
	We denote the image in $\mathcal{O}_{\ovcz^{(a)}}$ of the coordinates $t_\alpha^{(a)}$ and $s_\beta^{(a)}$ again by the same letters.
	Since $i^{(a)}\circ f^{(a)}\colon \ovy \rightarrow \ovcz^{(a)}$ is a morphism of weak formal log schemes over $\ct$, the homomorphism $(i^{(a)}\circ f^{(a)})^{\sharp}\colon (i^{(a)}\circ f^{(a)})^{-1}\co_{\ovcz^{(a)}} \rightarrow \co_{\ovy}$ sends $\prod_{\alpha\in\Upsilon_{\ovy^{(a)}}}t_\alpha^{(a)}$ to $\prod_{\gamma\in\Upsilon_{\ovy}}t_{\gamma}$, and $t_\alpha^{(a)}$ to $\prod_{\gamma\in\Upsilon_{\ovy}^{a,\alpha}}t_{\gamma}$ up to a unit in $\co_{\ovy}$.
	Moreover since $(i^{(a)}\circ f^{(a)})(Y)\subset \cz^{(a)}$, $(i^{(a)}\circ f^{(a)})^{\sharp}$ sends $s^{(a)}_\beta$ to $\prod_{\delta\in\Upsilon_{D}}s_{\delta}^{m^{(a)}_{\beta,\delta}}$ up to a unit in $\co_{\ovy}$.
	Hence, if we set
		$$
		u^{(a)}_\alpha: = \frac{t_\alpha^{(a)}}{\prod_{\gamma\in\Upsilon_{\ovy}^{a,\alpha}}t_{\gamma}^{(a_0)}} \qquad
		\text{ and }\qquad
		v^{(a)}_\beta: = \frac{s_\beta^{(a)}}{\prod_{\delta\in\Upsilon_{D}}s_{\delta}^{(a_0),m^{(a)}_{\beta,\delta}}}
		$$
	for each $a\in\underline{A}$, $\alpha\in\Upsilon_{\ovcy^{(a)}}$ and $\beta\in\Upsilon_{\cd^{(a)}}$, we obtain a smooth morphism
		\begin{eqnarray*}
		\ovcz'_A& \rightarrow &\Spwf\frac{\so_F[t_\alpha^{(a)},s_\beta^{(a)},u_\alpha^{(a),\pm 1},v_\beta^{(a),\pm 1}]^\dagger_{a\in\underline{A},\alpha\in\Upsilon_{\ovcy^{(a)}},\beta\in\Upsilon_{\cd^{(a)}}}}{\sum_{a\in\underline{A},\alpha\in\Upsilon_{\ovcy^{(a)}}}(t^{(a)}_\alpha-u^{(a)}_\alpha\prod_{\gamma\in\Upsilon^{a,\alpha}_{\ovy}}t^{(a_0)}_\gamma)}\\
		&& = \Spwf\frac{\so_F[t_{\gamma}^{(a_0)},s_{\delta}^{(a_0)},u_\alpha^{(a),\pm 1},v_\beta^{(a),\pm 1}]^\dagger_{\gamma\in\Upsilon_{\ovy},\delta\in\Upsilon_{D},a\in\underline{A},\alpha\in\Upsilon_{\ovcy^{(a)}},\beta\in\Upsilon_{\cd^{(a)}}}}{\sum_{\gamma\in\Upsilon_{\ovy}}(u_\gamma^{(a_0)}-1)},
		\end{eqnarray*}
	where $a_0$ is a fixed element in $\underline{A}$ with $f^{(a_0)} = \mathrm{id}_{(Y,\ovy)}$.
	The embedding $\ovcz_A \rightarrow \ovcz_A'$ is given by the ideal
		\begin{equation*}
		\sum_{a,a'\in\underline{A}} \left(\frac{\prod_{\alpha\in\Upsilon_{\ovcy^{(a)}}}t^{(a)}_\alpha}{\prod_{\alpha'\in\Upsilon_{\ovcy^{(a')}}}t^{(a')}_{\alpha'}}-1 \right) = 
		\sum_{a\in\underline{A}} \left(\prod_{\alpha\in\Upsilon_{\ovcy^{(a)}}}u_\alpha^{(a)}-1 \right).
		\end{equation*}
	Consequently we have a smooth morphism
		\begin{equation*}
		\ovcz_A \rightarrow \Spwf\frac{\so_F[t_{\gamma}^{(a_0)},s_{\delta}^{(a_0)}, u_\alpha^{(a),\pm 1},v_\beta^{(a),\pm 1}]^\dagger_{\gamma\in\Upsilon_{\ovy},\delta\in\Upsilon_{D},a\in\underline{A},\alpha\in\Upsilon_{\ovcy^{(a)}},\beta\in \Upsilon_{\cd^{(a)}}}}{\sum_{\gamma\in\Upsilon_{\ovy}}(u^{(a_0)}_\gamma-1)+\sum_{a\in\underline{A}}(\prod_{\alpha\in\Upsilon_{\ovcy^{(a)}}}u_\alpha^{(a)}-1)}.
		\end{equation*}
	The ideals of $\mathcal{E}_A\subset\ovcz_A$ and $\widetilde{\cd}_A\subset\ovcz_A$ are generated by 
		$$
		\prod_{\alpha\in\Upsilon_{\ovcy^{(a)}}}t^{(a)}_\alpha\ \ (\text{ for any fixed }a\in\underline{A})
		\quad \text{ and }\quad
		\prod_{\substack{a\in\underline{A}\\ \beta\in\Upsilon_{\cd^{(a)}}}}s_\beta^{(a)},
		$$
	and hence by 
		$$
		\prod_{\gamma\in\Upsilon_{\ovy}}t_{\gamma}^{(a_0)}
		\quad \text{ and } \quad
		\prod_{\delta\in\Upsilon_{D}}s_{\delta}^{(a_0),m_{\delta}}
		$$
	respectively, where
		\[m_{\delta}: = \sum_{\substack{a\in\underline{A}\\ \beta\in\Upsilon_{D^{(a)}}}}m^{(a)}_{\beta,\delta}\geqslant 1.\]	
	Thus $\cd_A$ is a simple normal crossing divisor defined by $\prod_{\delta\in\Upsilon_D}s_{\delta}^{(a_0)}$.
	Therefore there is a natural morphism $i_A\colon\ovy \rightarrow \ovcz_A$ sending $u_\alpha^{(a)}$ and $v_\beta^{(\alpha)}$ to units in $\co_{\ovy}$, which is a strict immersion.
	
	If there is another $a_1\in\underline{A}$ with $f^{(a_1)} = \mathrm{id}_{(Y,\ovy)}$, we have
		\begin{align}\label{eq: indep-blow-up}
		w_\alpha^{(a)}: =& \frac{t_\alpha^{(a)}}{\prod_{\gamma\in\Upsilon_{\ovy}^{a,\alpha}}t_{\gamma}^{(a_1)}} = \frac{u_\alpha^{(a)}}{\prod_{\gamma\in\Upsilon_{\ovy}^{a,\alpha}}u_{\gamma}^{(a_1)}}
		\quad\text{and}\quad\\
		z_\beta^{(a)}: =& \frac{s_\beta^{(a)}}{\prod_{\delta\in\Upsilon_D^{a,\beta}}s_{\delta}^{(a_1)}} = \frac{v_\beta^{(a)}}{\prod_{\delta\in\Upsilon_{D}^{a,\beta}}v_{\delta}^{(a_1)}}\nonumber
		\end{align}
	for any $a\in A$, $\alpha\in\Upsilon_{\ovy^{(a)}}$, and $\beta\in\Upsilon_{D^{(a)}}$.
	Note that $w_\alpha^{(a)}$ and $z_\beta^{(a)}$ are counterparts to $u_\alpha^{(a)}$ and $v^{(a)}_\beta$ for the choice $a_1$. The equalities \eqref{eq: indep-blow-up} say that $w_\alpha^{(a)}$ and $z_\beta^{(a)}$ can be written by $u_{\alpha}^{(a),\pm 1}$ and $v_{\beta}^{(a),\pm 1}$. Conversely, by the same reason, $u_\alpha^{(a)}$ and $v_\beta^{(a)}$ are written by $w_{\alpha}^{(a),\pm 1}$ and $z_{\beta}^{(a),\pm 1}$. Since $\ovcz'_A$ for the choice $a_0$ is given by adding $u_{\alpha}^{(a),\pm 1}$ and $v_{\beta}^{(a),\pm 1}$ to the coordinate ring, it is also given by adding $w_{\alpha}^{(a),\pm 1}$ and $z_{\beta}^{(a),\pm 1}$. Thus this coincides with $\ovcz'_A$ for the choice $a_1$.
\end{proof}

We continue to use  the notation of the proof of Proposition~\ref{prop: datum for A}. 
The Frobenius $\phi^{(a)}$ sends up to units  $t_\alpha^{(a)}$ and $s_\beta^{(a)}$ to $t_\alpha^{(a),p}$ and $s_\beta^{(a),p}$, respectively.
Hence $\prod\phi^{(a)}$ on $\prod_{\so_F}(\ovcz^{(a)})_{a\in\underline{A}}$ lifts uniquely to an endomorphism $\phi_A$ on $\ovcz_A$, which sends $u_\alpha^{(a)}$ and $v_\beta^{(a)}$ to $u_\alpha^{(a),p}$ and $v_\beta^{(a),p}$ up to units.
We endow $\ovcz_A$ with the log structure associated to $\mathcal{E}_A\cup\cd_A$.

We precede the definition of the last two entries of the tuple $\Xi_\hk$ with the following statement.

\begin{enumerate}
	\item[(iii)] For $(Y,\ovy)\in\overline{\mathsf{LS}}_{k^0,\Xi_\hk^\rig}^{\mathrm{ss},\mathrm{em}}$, we define a functor $P_{(Y,\ovy)}\colon\mathsf{SET}^0_{\Xi_\hk^\rig}(Y,\ovy) \rightarrow \mathsf{RD}_{\Xi_\hk^\rig}(Y,\ovy)$ by $P_{(Y,\ovy)}(A): = ((\cz_A,\ovcz_A,\phi_A),i_A)$.
\end{enumerate}

\begin{lemma}\label{lem: g circ HK}
	Let $g\colon(Y,\ovy) \rightarrow (Y',\ovy')$ be a morphism in $\overline{\mathsf{LS}}_{k^0,\Xi_\hk^\rig}^{\mathrm{ss},\mathrm{em}}$, $B\in\mathsf{SET}^0_{\Xi_\hk^\rig}(Y,\ovy)$, $C\in\mathsf{SET}^0_{\Xi_\hk^\rig}(Y',\ovy')$, and $\tau\colon g^\circ(C) \rightarrow  B$ be a morphism in $\mathsf{SET}^0_{\Xi_\hk^\rig}(Y,\ovy)$.
	Then the composition of the natural morphism $\ovcz_B \rightarrow  \prod_{\so_F}(\ovcz^{(b)})_{b\in\underline{B}}$ and the morphism
		\[\prod_{\so_F}(\ovcz^{(b)})_{b\in\underline{B}} \rightarrow \prod_{\so_F}(\ovcz^{(b)})_{b\in \underline{g^\circ(C)}} = \prod_{\so_F}(\ovcz^{(c)})_{c\in\underline{C}}\]
	induced by $\tau$ uniquely lifts to a morphism
		\begin{equation}\label{eq: g circ}
		G = G(g,\tau)\colon \ovcz_B \rightarrow \ovcz_C.
		\end{equation}
\end{lemma}

\begin{proof}
	Working locally, we assume without loss of generality that there is a commutative diagram
		\begin{equation*}\xymatrix{
		\ovy^{(a)} \ar[r] \ar[d]^{i^{(a)}}& \Spec\frac{k[t_\alpha^{(a)},t_{\alpha'}^{(a),-1},s_\beta^{(a)}, s_{\beta'}^{(a),-1}]_{\alpha\in\Upsilon_{\ovcy^{(a)}},\alpha' \in\Upsilon_{\ovcy^{(a)}}\setminus\Upsilon_{\ovy^{(a)}},\beta\in\Upsilon_{\cd^{(a)}},\beta'\in\Upsilon_{\cd^{(a)}}\setminus\Upsilon_{D^{(a)}}}}{(\prod_{\alpha\in\Upsilon_{\ovcy^{(a)}}}t_{\alpha}^{(a)})} \ar[d]\\
		\ovcz^{(a)} \ar[r]&\Spwf\so_F[t_\alpha^{(a)}, s_\beta^{(a)}]^\dagger_{\alpha\in\Upsilon_{\ovcy^{(a)}},\beta\in\Upsilon_{\cd^{(a)}}}
		}\end{equation*}
	for any $a\in\underline{B}$ or $a\in\underline{C}$.
	Let $b_0\in\underline{B}$ and $c_0\in\underline{C}$ with $f^{(b_0)} = \mathrm{id}_{(Y,\ovy)}$ and $f^{(c_0)} = \mathrm{id}_{(Y',\ovy')}$.
	Note that for any $c\in\underline{C}$, $\alpha\in\Upsilon_{\ovcy^{(c)}}$, and $\beta\in\Upsilon_{\cd^{(c)}}$, $\delta\in\Upsilon_D$, we have
		$$
		\Upsilon_{\ovy}^{c,\alpha} = \coprod_{\gamma'\in\Upsilon_{\ovy'}^{c,\alpha}}\Upsilon_{\ovy}^{c_0,\gamma'} \quad \text{ and } \quad
		m^{(c)}_{\beta,\delta} = \sum_{\delta'\in\Upsilon_{D'}}m^{(c)}_{\beta,\delta'}\cdot m^{(c_0)}_{\delta',\delta}.
		$$
	Therefore there are equalities in $\mathcal{O}_{\ovcz_B}$
		\begin{eqnarray*}
		\frac{t_\alpha^{(c)}}{\prod_{\gamma'\in\Upsilon_{\ovy'}^{c,\alpha}}t_{\gamma'}^{(c_0)}}
		& = &\frac{t_\alpha^{(c)}}{\prod_{\gamma\in\Upsilon_{\ovy}^{c,\alpha}}t_{\gamma}^{(b_0)}}
		\cdot\prod_{\gamma'\in\Upsilon^{c,\alpha}_{\ovy'}}\left(\frac{\prod_{\gamma\in\Upsilon_{\ovy}^{c_0,\gamma'}}t_{\gamma}^{(b_0)}}{t_{\gamma'}^{(c_0)}}\right),\\
		\frac{s_\beta^{(c)}}{\prod_{\delta'\in\Upsilon_{D'}}s_{\delta'}^{(c_0),m^{(c)}_{\beta,\delta'}}}
		& = &\frac{s_\beta^{(c)}}{\prod_{\delta\in\Upsilon_D}s_{\delta}^{(b_0),m^{(c)}_{\beta,\delta}}}
		\cdot\prod_{\delta'\in\Upsilon_{D'}} \left(\frac{\prod_{\delta\in\Upsilon_D}s_{\delta}^{(b_0), m^{(c_0)}_{\delta',\delta}}}{s_{\delta'}^{(c_0)}}\right)^{m^{(c)}_{\beta,\delta'}},
		\end{eqnarray*}
	and they belong to $\co_{\ovcz_B}^\times$.
	This implies that $\ovcz_B \rightarrow \prod_{\so_F}(\ovcz^{(c)})_{c\in\underline{C}}$ factors through $\ovcz'_C$.
	This factorization is unique by the universal property of the blow-up.
	Moreover since we have $\prod_{\alpha\in\Upsilon_{\ovcy^{(c)}}}t_\alpha^{(c)} = \prod_{\alpha'\in\Upsilon_{\ovcy^{(c')}}}t_{\alpha'}^{(c')}$ in $\mathcal{O}_{\ovcz_B}$ for any $c,c'\in\underline{C}$, this uniquely factors through $\ovcz_C$.
\end{proof}

\begin{enumerate}
	\item[(iv)] For $g,B,C,\tau$ as in Lemma \ref{lem: g circ HK}, we define a morphism
				\[g^\circ_\tau\colon ((Y,\ovy),(\cz_B,\ovcz_B,\phi_B),i_B) \rightarrow ((Y',\ovy'),(\cz_C,\ovcz_C,\phi_C),i_C)\]
			in $\mathsf{RT}_{\Xi_\hk^\rig}$ by $g^\circ_\tau: = (g,G(g,\tau))$, where $G(g,\tau)$ is as in Lemma \ref{lem: g circ HK}.	
	\item[(v)] Moreover one sees easily that, for $((Y,\ovy),(\cz,\ovcz,\phi),i)\in\mathsf{RT}_{\Xi_\hk^\rig}$ and $A\in\mathsf{SET}^0_{\Xi_\hk^\rig}((Y,\ovy),(\cz,\ovcz,\phi),i)$, the diagonal morphism $\ovcz \rightarrow \prod_{\so_F}(\ovcz^{(a)})_{a\in\underline{A}}$ uniquely lifts to a morphism $\ovcz \rightarrow \ovcz_A$.
			This defines a morphism $D_A\colon ((\cz,\ovcz,\phi),i) \rightarrow  ((\cz_A,\ovcz_A,\phi_A),i_A)$ in $\mathsf{RD}_{\Xi_\hk^\rig}(Y,\ovy)$.
\end{enumerate}

By Lemma \ref{lem: qis with KimHain} we have $\Gamma(]\ovy[_{\ovcy}^{\log},\Gd_{\an}\widetilde{\omega}^{\sbt}_{\ovcy,\Q}\llbracket u\rrbracket) \cong\Gamma(]\ovy[_{\ovcy}^{\log},\Gd_{\an}\omega^{\sbt}_{\ovcy/\so_F^0,\Q}).$
Thus Assumption \ref{ass: weq} and Assumption \ref{ass: descent} follow from the corresponding facts for log rigid cohomology of $\ovy$. 
Moreover, the next lemma implies that $\overline{\mathsf{LS}}_{k^0,\Xi_\hk^\rig}^{\mathrm{ss},\mathrm{loc}} = \ssk$, that is strictly semistable $k^0$-log schemes with boundary are locally $\Xi_\hk$-embeddable.

\begin{lemma}\label{lem: HK embedding}
	Let $(Y,\ovy)$ be an object in $\ssk$ admitting a (global) chart as in Definition~\ref{def: ss log scheme} (i), and assume that $\ovy$ is affine.
	Then $(Y,\ovy)$ is $\Xi_\hk^\rig$-embeddable.
\end{lemma}

\begin{proof}
	By \cite[Prop.~11.3]{Kato1996} there exists a closed immersion $\overline{Y}\hookrightarrow\overline{Z}$ into an affine smooth scheme $\overline{Z}$, which fits to a Cartesian diagram
		\[\xymatrix{
		\overline{Y} \ar[r] \ar[d]&\overline{Z} \ar[d]\\
		\Spec k[t_1,\ldots,t_m,s_1,\ldots,s_n]/(t_1\cdots t_m) \ar[r]&\Spec k[t_1,\ldots,t_m,s_1,\ldots,s_n],
		}\]
	where the vertical arrows are smooth. 
	By \cite{Elkik1973} we can lift $\overline{Z}$ to an affine smooth $\so_F$-scheme.
	 Let $\ovcz$ be its weak completion.
	If we take lifts to $\co_{\ovcz}$ of the images of $t_1,\ldots,t_m,s_1,\ldots,s_n$ in $\co_{\overline{Z}}$, we obtain a smooth morphism $\ovcz \rightarrow \Spwf\so_F[t_1,\ldots,t_m,s_1,\ldots,s_n]^\dagger$.
	Moreover the infinitesimal lifting property of $\ovcz \rightarrow \Spwf\so_F[t_1,\ldots,t_m,s_1,\ldots,s_n]^\dagger$ implies that there exists a lift of the $p$-th power Frobenius on $\overline{Z}$  to the completion $\hat{\ovcz}$ of $\ovcz$  which is compatible with $\sigma$ on $\ct$ and which sends $t_i$ and $s_j$ to their $p$-th powers.
	Hence by \cite[Cor.~2.4.3]{vanderPut1986} we obtain a lift $\phi$ of the Frobenius to $\ovcz$ of the same form.
	We endow $\overline{\mathcal{Z}}$ with the log structure defined by the chart $\N^{m + n} \rightarrow \mathcal{O}_{\overline{\mathcal{Z}}},(\ell_1,\ldots,\ell_{m + n}) \mapsto  t_1^{\ell_1}\cdots t_m^{\ell_m}s_1^{\ell_{m + 1}}\cdots s_n^{\ell_{m + n}}$.
	Let $\cd$ be the closed weak formal subscheme of $\ovcz$ defined by $s_1\cdots s_n = 0$, and let $\cz: = \ovcz\setminus\cd$.
	Then $(\cz,\ovcz)$ and $\phi$ give a $\Xi_\hk^\rig$-rigid datum for $(Y,\ovy)$.
\end{proof}

Finally, we may set
	\[\mathbb{R}\Gamma^\rig_\hk(Y,\ovy): = \mathbb{R}\Gamma_{\Xi_\hk^\rig}(Y,\ovy)\]
for any $(Y,\ovy)\in\overline{\mathsf{LS}}_{k^0}^{\mathrm{ss}}$, which we call the {\it rigid Hyodo--Kato cohomology} of $(Y,\ovy)$.

\subsubsection{Base change and Frobenius}\label{subsubsec: Base change and Frobenius}

In this section we use the axiomatization of morphisms of rigid cohomological tuples of \S\ref{subsubsec: Morphisms} to study base change properties of the canonical rigid complexes defined above and the existence of a Frobenius morphism in applicable cases.

First we take care of base change from $\ct$ to $\so_F^0$ by defining a morphism of rigid cohomological tuples $L_{\so_F^0/\ct}\colon\Xi(\ct) \rightarrow \Xi(\so_F^0)$ consisting of functors
	\begin{align*}
	&L_{\so_F^0/\ct,\mathrm{base}}\colon\mathsf{LS}_T \rightarrow \mathsf{LS}_{k^0}, \qquad Y \mapsto  Y\times_Tk^0,\\
	&L_{\so_F^0/\ct,\mathrm{lift}}\colon\mathsf{LS}_\ct^{\mathrm{sm}} \rightarrow \mathsf{LS}_{\so_F^0}^{\mathrm{ss}},\qquad\cz \mapsto \cz\times_\ct\so_F^0,
	\end{align*}
a forgetful functor 
	$$L_{\so_F^0/\ct,\mathrm{mod}}\colon \mathsf{Mod}_F\rightarrow\mathsf{Mod}_{F[t]^\dagger}$$ 
	induced by the map $F[t]^\dagger\rightarrow F, t\mapsto 0$, and a natural transformation $L_{\so_F^0/\ct,\mathrm{coh}} $ 
	defined by the natural morphisms
	\[L_{\so_F^0/\ct,\mathrm{coh}}(Y,\cz,i)\colon \Gamma(]Y[_{\cz}^{\log},\Gd_\an\omega^{\sbt}_{\cz/\ct,\Q})\rightarrow \Gamma(]Y\times_Tk^0[_{\cz\times_{\ct}\so_F^0},\omega^{\sbt}_{\cz\times_{\ct}\so_F^0/\so_F^0,\Q})\]
induced by the inclusion $Y\times_Tk^0\hookrightarrow Y$.
Hence $L_{\so_F^0/\ct}$ induces a morphism
	\begin{equation}\label{eq: LFT}
	\mathbb{L}_{\so_F^0/\ct}(Y)\colon\mathbb{R}\Gamma_\rig(Y/\ct) \rightarrow \mathbb{R}\Gamma_\rig(Y\times_Tk^0/\so_F^0)
	\end{equation}
for any $Y\in\mathsf{LS}_T$.

Similarly for base change from $\ct$ to $\so_K^\pi$ we can define a morphism of rigid cohomological tuples $L_{\so_K^\pi/\ct}\colon\Xi(\ct) \rightarrow \Xi(\so_K^\pi)$, which induces a morphism
	\begin{equation}\label{eq: LKT}
	\mathbb{L}_{\so_K^\pi/\ct}(Y)\colon\mathbb{R}\Gamma_\rig(Y/\ct)\rightarrow \mathbb{R}\Gamma_\rig(Y\times_Tk^0/\so_K^\pi)
	\end{equation}
	for any $Y\in\mathsf{LS}_T$.

We also want to compare rigid cohomology of log schemes with boundary over $\ct$ and over $(\ct,\ovct)$. Thus we define a morphism of rigid cohomological tuples $L^\sharp\colon\Xi_b(\ct,\ovct) \rightarrow \Xi_b(\ct)$ consisting of forgetful functors
	\begin{align*}
		&L_{\mathrm{base}}^\sharp\colon\overline{\mathsf{LS}}_{(T,\ovt)} \rightarrow \overline{\mathsf{LS}}_T,\\
		&L_{\mathrm{lift}}^\sharp\colon\overline{\mathsf{LS}}_{(\ct,\ovct)}^{\mathrm{sm}} \rightarrow \overline{\mathsf{LS}}_{\ct}^{\mathrm{sm}},
	\end{align*}
a functor 
	$$L_{\mathrm{mod}}^\sharp: = \mathrm{id\colon\mathsf{Mod}_F \rightarrow \mathsf{Mod}_F},$$ 
and a natural transformation $L_{\mathrm{coh}}^\sharp(Y,\ovy): = \mathrm{id}$ for $(Y,\ovy)\in\overline{\mathsf{LS}}_{(T,\ovt)}$. 
This induces an isomorphism
	\begin{equation}\label{eq: Lsharp}
	\mathbb{L}^\sharp\colon\mathbb{R}\Gamma_\rig((Y,\ovy)/(\ct,\ovct)) \xrightarrow{\sim} \mathbb{R}\Gamma_\rig((Y,\ovy)/\ct)
	\end{equation}
for any $(Y,\ovy)\in\overline{\mathsf{LS}}_{(T,\ovt),\Xi(\ct,\ovct)}^{\mathrm{loc}}$.

On the other hand we define a morphism of rigid cohomological tuples $L_{(\ct,\ovct)/\ct}\colon\Xi_b(\ct) \rightarrow \Xi_b(\ct,\ovct)$ by functors
	\begin{align*}
		&L_{(\ct,\ovct)/\ct,\mathrm{base}}\colon\overline{\mathsf{LS}}_T \rightarrow \overline{\mathsf{LS}}_{(T,\ovt)},\qquad(Y,\ovy) \mapsto (Y,\ovy)\times_T(T,\ovt),\\
		&L_{(\ct,\ovct)/\ct,\mathrm{lift}}\colon\overline{\mathsf{LS}}_{\ct}^{\mathrm{sm}} \rightarrow \overline{\mathsf{LS}}_{(\ct,\ovct)}^{\mathrm{sm}},\qquad(\cz,\ovcz) \mapsto (\cz,\ovcz)\times_{\ct}(\ct,\ovct),\\
		&L_{(\ct,\ovct)/\ct,\mathrm{mod}}: = \mathrm{id}\colon\mathsf{Mod}_F \rightarrow \mathsf{Mod}_F,
	\end{align*}
and a natural transformation $L_{(\ct,\ovct)/\ct,\mathrm{coh}}$ given by morphisms
	\[L_{(\ct,\ovct)/\ct,\mathrm{coh}}(Y,\ovy)\colon\Gamma(]\ovy[_{\ovcz}^{\log},\Gd_\an\omega^{\sbt}_{(\cz,\ovcz)/\ct,\Q}) \rightarrow  \Gamma(]\ovy\overline{\times}_T \ovt[_{\ovcz\overline{\times}_{\ct}\ovct}^{\log},\Gd_\an\omega^{\sbt}_{(\cz,\ovcz\overline{\times}_{\ct}\ovct)/\ct,\Q})\]
induced by the projections $(Y,\ovy\overline{\times}_T\ovt): = (Y,\ovy)\times(T,\ovt) \rightarrow (Y,\ovy)$ and $(\cz,\ovcz\overline{\times}_{\ct}\ovct): = (\cz,\ovcz)\times(\ct,\ovct) \rightarrow (\cz,\ovcz)$.
Then $L_{(\ct,\ovct)/\ct}$ induces a morphism
	\begin{equation}\label{eq: Lovct}
	\mathbb{L}_{(\ct,\ovct)/\ct}\colon\mathbb{R}\Gamma_\rig((Y,\ovy)/\ct) \rightarrow \mathbb{R}\Gamma_\rig((Y,\ovy)\times (T,\ovt)/(\ct,\ovct))
	\end{equation}
for any $(Y,\ovy)\in\overline{\mathsf{LS}}_{T,\Xi(\ct)}^{\mathrm{loc}}$.

We need to relate cohomology of log schemes with and without boundary over the same log scheme. Specifically we define a morphism of rigid cohomological tuples $L^b\colon\Xi_b(\ct) \rightarrow \Xi(\ct)$ by functors
	\begin{align*}
	&L^b_{\mathrm{base}}\colon\overline{\mathsf{LS}}_T \rightarrow \mathsf{LS}_T,\qquad(Y,\ovy) \mapsto  Y,\\
	&L^b_{\mathrm{lift}}\colon \overline{\mathsf{LS}}_{\ct}^{\mathrm{sm}} \rightarrow \mathsf{LS}_{\ct}^{\mathrm{sm}},\qquad(\cz,\ovcz) \mapsto \cz,
	\end{align*}
a forgetful functor
	\[L^b_{\mathrm{mod}}\colon \mathsf{Mod}_{F[t]^\dagger} \rightarrow \mathsf{Mod}_F,\]
and a natural transformation $L^b_{\mathrm{coh}}$ defined by the natural morphisms
	\[L^b_{\mathrm{coh}}(Y,\cz,i)\colon\Gamma(]\ovy[_{\ovcz}^{\log},\Gd_\an\omega^{\sbt}_{(\cz,\ovcz)/\ct,\Q}) \rightarrow \Gamma(]Y[_{\cz}^{\log},\Gd_\an\omega^{\sbt}_{\cz/\ct,\Q}).\]
Then $L^b$ induces morphisms
	\begin{equation}\label{eq: Lb}
	\mathbb{L}^b\colon \mathbb{R}\Gamma_\rig((Y,\ovy)/\ct) \rightarrow \mathbb{R}\Gamma_\rig(Y/\ct)
	\end{equation}
for any $(Y,\ovy)\in\overline{\mathsf{LS}}_T$.

Let $\Xi'_\hk$ be a rigid cohomological tuple obtained from $\Xi_\hk^\rig$ by forgetting Frobenius and monodromy, and replacing $\mathsf{Mod}_F(\varphi,N)$ by $\mathsf{Mod}_F$. Clearly, the natural morphism $\Xi'_\hk\rightarrow\Xi_\hk^\rig$ induces isomorphisms
	$$\mathbb{R}\Gamma_{\Xi'_\hk}(Y,\ovy)\cong\mathbb{R}\Gamma_{\Xi_\hk^\rig}(Y,\ovy)=\mathbb{R}\Gamma_\hk(Y,\ovy).$$
Keeping this in mind, we define a morphism of rigid cohomological tuples $L^\hk\colon\Xi'_\hk \rightarrow \Xi(\so_F^0)$ consisting of functors
	\begin{align*}
	&L^\hk_{\mathrm{base}}\colon \overline{\mathsf{LS}}_{k^0}^{\mathrm{ss}} \rightarrow \mathsf{LS}_{k^0},\qquad(Y,\ovy) \mapsto \ovy,\\
	&L^\hk_{\mathrm{lift}}\colon \overline{\mathsf{LS}}_{\ct}^{\mathrm{ss}}(\varphi) \rightarrow \mathsf{LS}_{\so_F^0},\qquad(\cz,\ovcz,\phi) \mapsto \ovcz\times_{\ct}\so_F^0,\\
	&L^\hk_{\mathrm{mod}}:=\mathrm{id}\colon \mathsf{Mod}_F\rightarrow\mathsf{Mod}_F,
	\end{align*}
and a natural transformation $L^\hk_{\mathrm{coh}}$ defined by the natural quasi-isomorphisms
	\[L^\hk_{\mathrm{coh}}((Y,\ovy),(\cz,\ovcz,\phi),i)\colon\Gamma(]\ovy[_{\ovcy}^{\log},\Gd_\an\widetilde{\omega}^{\sbt}_{\ovcy,Q}\llbracket u\rrbracket) \rightarrow \Gamma(]\ovy[_{\ovcy}^{\log},\Gd_\an\omega^{\sbt}_{\ovcy/\so_F^0,\Q}),\]
where $\ovcy: = \ovcz\times_\ct\so_F^0$.
Then $L^\hk$ induces isomorphisms
	\begin{equation}\label{eq: LHK}
	\mathbb{L}^\hk\colon \mathbb{R}\Gamma^\rig_\hk(Y,\ovy) \xrightarrow{\sim} \mathbb{R}\Gamma_\rig(\ovy/\so_F^0)
	\end{equation}
for any $(Y,\ovy)\in\overline{\mathsf{LS}}_{k^0}^{\mathrm{ss}}$.

We describe now the twisting by Frobenius in terms of a morphism of rigid cohomological tuples. Let $(R,S\hookrightarrow\mathcal{S})$ be one of $(F,k^0\hookrightarrow\so_F^0)$ or $(F[t]^\dagger,T\hookrightarrow\ct)$.
We define a morphism 
 $L^\sigma\colon\Xi(\mathcal{S}) \rightarrow \Xi(\mathcal{S})$ by functors
	\begin{align*}
	&L^\sigma_{\mathrm{base}}\colon \mathsf{LS}_S \rightarrow \mathsf{LS}_S,\qquad Y \mapsto  Y\times_{S,\sigma}S,\\
	&L^\sigma_{\mathrm{lift}}\colon \mathsf{LS}_{\mathcal{S}}^{\mathrm{sm}} \rightarrow \mathsf{LS}_{\mathcal{S}}^{\mathrm{sm}}, \qquad\cz \mapsto \cz\times_{\ct,\sigma}\ct,
	\end{align*}
a functor 
	$$L^\sigma_{\mathrm{mod}}\colon\mathsf{Mod}_R \rightarrow \mathsf{Mod}_R$$ given by twisting the $R$-action by $\sigma\colon R\rightarrow R$, and a natural transformation $L^\sigma_{\mathrm{coh}}$ defined by a morphism
	\[L^\sigma_{\mathrm{coh}}(Y,\cz,i)\colon\Gamma(]Y[_{\cz},\Gd_\an\omega^{\sbt}_{\cz/\mathcal{S},\Q}) \rightarrow \Gamma(]Y\times_{S,\sigma}S[_{\cz\times_{\mathcal{S},\sigma}\mathcal{S}},\Gd_\an\omega^{\sbt}_{\cz\times_{\mathcal{S},\sigma}\mathcal{S}/\mathcal{S},\Q})\]
induced by the canonical projections $Y\times_{S,\sigma}S \rightarrow  Y$ and $\cz\times_{\ct,\sigma}\ct \rightarrow \cz$.
Then $L^\sigma$ induces
	\begin{equation}\label{eq: Lsigma}
	\mathbb{L}^\sigma\colon \mathbb{R}\Gamma_\rig(Y/\mathcal{S}) \rightarrow \mathbb{R}\Gamma_\rig(Y\times_{S,\sigma}S/\mathcal{S})
	\end{equation}
for any $Y\in\mathsf{LS}_S$.

We need the same construction for log schemes with boundary, which turns out to be more involved. More precisely, we define a morphism $L^{b,\sigma}\colon\Xi_b(\ct,\ovct) \rightarrow \Xi_b(\ct,\ovct)$.
For $(Y,\ovy)\in\overline{\mathsf{LS}}_{(T,\ovt)}$, let $(Y,\ovy)\times_{\sigma}(T,\ovt)$ be the product in $\overline{\mathsf{LS}}_{(T,\ovt)}$ where the structure map of $(T,\ovt)$ is given by $\sigma\colon T \rightarrow  T$.
Let $(Y^\sigma,\ovy^\sigma)$ be an object in $\overline{\mathsf{LS}}_{(T,\ovt)}$ whose underlying log scheme with boundary is $(Y,\ovy)\times_{\sigma}(T,\ovt)$ and whose structure morphism is the canonical projection $Y\times_{T,\sigma}T \rightarrow  T$.
In a similar way we define $(\cz^\sigma,\ovcz^\sigma)$ for any $(\cz,\ovcz)\in\overline{\mathsf{LS}}_{(\ct,\ovct)}^{\mathrm{sm}}$.

\begin{lemma}
	Let $(\cz,\ovcz)$ be an object in $\overline{\mathsf{LS}}_{(\ct,\ovct)}^{\mathrm{sm}}$.
	Then $(\cz^\sigma,\ovcz^\sigma)$ is strongly smooth.
\end{lemma}

\begin{proof}
	This follows from local computations.
	Without loss of generality, we assume that there is a chart $(\alpha\colon P_{\ovcz} \rightarrow \cn_{\ovcz},\ \beta\colon \N \rightarrow  P^\gp)$ for $(\cz,\ovcz)$ and $a,b\in P$ as in Definition~\ref{def: strongly smooth}.
	Let $\mathcal{U}=\Spwf\so_F[s^{\pm 1}]^\dagger$ and $\overline{\mathcal{U}}=\Spwf\so_F[s]^\dagger$ be the strict open weak formal log subschemes of $\ovct$s defined by the equality $s = \frac{x_1}{x_2}$.
	Then $(\ct,\ovct)$ is covered by $(\ct,\ct)$ and $(\mathcal{U},\overline{\mathcal{U}})$, and hence $(\cz^\sigma,\ovcz^\sigma)$ is covered by $(\mathcal{U}_ + ,\overline{\mathcal{U}}_ + ): = (\cz,\ovcz)\times_{\ct,\sigma}(\ct,\ct)$ and $(\mathcal{U}_-,\overline{\mathcal{U}}_-): = (\cz,\ovcz)\times_{\ct,\sigma}(\mathcal{U},\overline{\mathcal{U}})$.
	For $\varepsilon\in\{ + ,-\}$, let $Q_{\varepsilon}$ be the image of
		\[P\oplus\N \rightarrow  P^\gp\oplus_{\Z,p}\Z: = \coker((\beta^\gp,\varepsilon p)\colon\Z \rightarrow  P^\gp\oplus\Z).\]
	Let $\gamma_{\varepsilon}\colon Q_{\varepsilon,\overline{\mathcal{U}}_{\varepsilon}} \rightarrow \cn_{\overline{\mathcal{U}}_{\varepsilon}}$ be the morphism induced by $\alpha$ and by
	\begin{eqnarray*}
	\N_\ct \rightarrow \cn_\ct,\ 1 \mapsto  t & \text{if} & \varepsilon =  + ,\\
	\N_{\overline{\mathcal{U}}} \rightarrow \cn_{\overline{\mathcal{U}}},\ 1 \mapsto  s & \text{if}  & \varepsilon = -.
	\end{eqnarray*}
	Let $\delta_{\varepsilon}\colon\N \rightarrow  Q_{\varepsilon}^\gp$ be the composition of $\N \rightarrow  P^\gp\oplus\Z,\ 1 \mapsto (0,\varepsilon 1)$ and the natural surjection $P^\gp\oplus\Z \rightarrow  P^\gp\oplus_{\Z,p}\Z = Q^\gp$.
	Set
	\begin{align*}
	&a_ + : = (a,0), b_ + : = (b,0)\in Q_ + ,\\
	&a_-: = (a,1) , b_-: = (b,1)\in Q_-.
	 \end{align*}
	Since we have $\kker\delta_{\varepsilon}^\gp \cong\kker\alpha_{\varepsilon}^\gp$ and $\coker\delta_{\varepsilon}^\gp \cong\coker\alpha_{\varepsilon}^\gp$, we see that the chart $(\gamma_\varepsilon,\delta_\varepsilon)$ for $(\mathcal{U}_\varepsilon,\overline{\mathcal{U}}_\varepsilon)$ with $a_\varepsilon,b_\varepsilon\in Q_\varepsilon$ satisfies the conditions in Definition~\ref{def: strongly smooth}.
\end{proof}

Now we define $L^{b,\sigma}$ by functors
	\begin{align*}
	&L^{b,\sigma}_{\mathrm{base}}\colon\overline{\mathsf{LS}}_{(T,\ovt)} \rightarrow \overline{\mathsf{LS}}_{(T,\ovt)}, \qquad(Y,\ovy) \mapsto (Y^\sigma,\ovy^\sigma),\\
	&L^{b,\sigma}_{\mathrm{lift}}\colon\overline{\mathsf{LS}}_{(\ct,\ovct)}^{\mathrm{sm}} \rightarrow \overline{\mathsf{LS}}_{(\ct,\ovct)}^{\mathrm{sm}}, \qquad (\cz,\ovcz) \mapsto (\cz^\sigma,\ovcz^\sigma),
	\end{align*}
a functor 
$$L^{b,\sigma}_{\mathrm{mod}}\colon\mathsf{Mod}_F \rightarrow \mathsf{Mod}_F$$ given by twisting $F$-action by $\sigma\colon F\rightarrow F$, and a natural transformation $L^{b,\sigma}_{\mathrm{coh}}$ defined by morphisms
\[L^{b,\sigma}_{\mathrm{coh}}((Y,\ovy),(\cz,\ovcz),i)\colon\Gamma(]\ovy[_{\ovcz}^{\log},\Gd_\an\omega^{\sbt}_{(\cz,\ovcz)/\ct,\Q}) \rightarrow \Gamma(]\ovy^\sigma[_{\ovcz^\sigma}^{\log},\Gd_\an\omega^{\sbt}_{(\cz^\sigma,\ovcz^\sigma)/\ct,\Q})\]
induced by the canonical projections $(Y^\sigma,\ovy^\sigma) \rightarrow (Y,\ovy)$ and $(\cz^\sigma,\ovcz^\sigma) \rightarrow (\cz,\ovcz)$.
Then $L^{b,\sigma}$ induces
	\begin{equation}\label{eq: Lbsigma}
	\mathbb{L}^{b,\sigma}\colon\mathbb{R}\Gamma_\rig((Y,\ovy)/(\ct,\ovct)) \rightarrow \mathbb{R}\Gamma_\rig((Y^\sigma,\ovy^\sigma)/(\ct,\ovct))
	\end{equation}
for any $(Y,\ovy)\in\overline{\mathsf{LS}}_{(T,\ovt),\Xi_b(\ct,\ovct)}^{\mathrm{loc}}$.

Lastly we can explain Frobenius endomorphisms on rigid complexes.
Let $(R,S\hookrightarrow\mathcal{S})$ be one of $(F,k^0\hookrightarrow\so_F^0)$ or $(F[t]^\dagger,T\hookrightarrow\ct)$.
For $Y\in\mathsf{LS}_S$, let $\chi\colon Y \rightarrow  Y^\sigma$ be the relative Frobenius of $Y$ over $S$.
By functoriality we have a morphism in $\mathscr{D}^ + (\mathsf{Mod}_R)$
	\[\chi^{\ast}\colon\mathbb{R}\Gamma_\rig(Y^\sigma/\mathcal{S}) \rightarrow \mathbb{R}\Gamma_\rig(Y/\mathcal{S}).\]
We define the Frobenius endomorphism
	\begin{equation}\label{eq: base change Frobenius}
	\varphi\colon\mathbb{R}\Gamma_\rig(Y/\mathcal{S}) \rightarrow \mathbb{R}\Gamma_\rig(Y/\mathcal{S})
	\end{equation}
as the composition $\varphi: = \chi^{\ast}\circ(\mathbb{L}^\sigma(Y))$.

Similarly, for $(Y,\ovy)\in\overline{\mathsf{LS}}_{(T,\ovt)}$, let $\chi\colon(Y,\ovy) \rightarrow (Y^\sigma,\ovy^\sigma)$ the morphism in $\overline{\mathsf{LS}}_{(T,\ovt)}$ induced by the absolute Frobenius $(Y,\ovy) \rightarrow (Y,\ovy)$ and the natural morphism $(Z,\ovz) \rightarrow (T,\ovt)$.
When $(Y,\ovy)$ is locally $\Xi_b(\ct,\ovct)$-embeddable, by functoriality we have a morphism in $\mathscr{D}^ + (\mathsf{Mod}_F)$
	\[\chi^{\ast}\colon\mathbb{R}\Gamma_\rig((Y^\sigma,\ovy^\sigma)/(\ct,\ovct)) \rightarrow \mathbb{R}\Gamma_\rig((Y,\ovy)/(\ct,\ovct)).\]
We define the Frobenius endomorphism
	\begin{equation}\label{eq: Frobenius boudary}
	\varphi\colon\mathbb{R}\Gamma_\rig((Y,\ovy)/(\ct,\ovct)) \rightarrow \mathbb{R}\Gamma_\rig((Y,\ovy)/(\ct,\ovct))
	\end{equation}
to be the composition $\varphi: = \chi^{\ast}\circ(\mathbb{L}^{b,\sigma})$.

With this, it is possible to define a Frobenius endomorphism on $\mathbb{R}\Gamma_\rig((Y,\ovy)/\ct)$ in the following case.
Let $(Y,\ovy)\in\overline{\mathsf{LS}}_{T,\Xi_b(\ct)}^{\mathrm{loc}}$ and assume that $\mathbb{L}_{(\ct,\ovct)/\ct}(Y,\ovy)$ is an isomorphism.
Then we define $\varphi$ on $\mathbb{R}\Gamma_\rig((Y,\ovy)/\ct)$ as the composition
	\begin{eqnarray*}
	\mathbb{R}\Gamma_\rig((Y,\ovy)/\ct)& \xrightarrow{\mathbb{L}_{(\ct,\ovct)/\ct}(Y,\ovy)}&
	\mathbb{R}\Gamma_\rig((Y,\ovy\overline{\times}_{T}\ovt)/(\ct,\ovct))\\
	& \xrightarrow{\ \ \ \ \ \ \ \varphi\ \ \ \ \ \ \ }&\mathbb{R}\Gamma_\rig((Y,\ovy\overline{\times}_T\ovt)/(\ct,\ovct))\\
	& \xrightarrow{\mathbb{L}_{(\ct,\ovct)/\ct}(Y,\ovy)^{-1}}&\mathbb{R}\Gamma_\rig((Y,\ovy)/\ct).
	\end{eqnarray*}

\section{Frobenius and Hyodo--Kato map}\label{sec: rigid structures}

In the previous section, we have constructed several canonical rigid complexes.
	\begin{itemize}
		\item For a strictly semistable $k^0$-log scheme with boundary, i.e. $(Y,\ovy) \in \ssk$, the canonical rigid Hyodo--Kato complex $\mathbb{R}\Gamma_{\hk}^{\rig}(Y,\ovy)$ with Frobenius and monodromy operator.
		\item For a fine log scheme over $k^0$, i.e. $Y  \in \lsk$, the canonical log rigid complexes $\mathbb{R}\Gamma_{\rig}(Y/\so_F^0)$ and $\mathbb{R}\Gamma_{\rig}(Y/\so_K^\pi)$.
		\item For a fine log scheme over $T$, i.e. $Y \in \lst$, the canonical log rigid complex $\mathbb{R}\Gamma_{\rig}(Y/\mathcal{T})$.
		\item For a locally $\Xi_b(\mathcal{T})$-embeddable fine $T$-log scheme with boundary, i.e. $(Y,\ovy)\in \lelsb$, the canonical log rigid complex $\mathbb{R}\Gamma_{\rig}((Y,\ovy)/\mathcal{T})$.
	\end{itemize}

In the situation of a strictly semistable $k^0$-log scheme with boundary $(Y,\ovy)$, we will see in \S\ref{subsec: HK-map} how to construct an analogue of the Hyodo--Kato map as a zigzag between $\mathbb{R}\Gamma_{\rig}(\ovy/\so_F^0) $ and $ \mathbb{R}\Gamma_{\rig}(\ovy/\so_K^\pi)$ via a complex $\mathbb{R}\Gamma_{\rig}((\overline{V}_{\sbt},\overline{P}_{\sbt})/\mathcal{T})$ for a certain simplicial $T$-log scheme with boundary. 
Because of the canonical quasi-isomorphism $\mathbb{R}\Gamma^\rig_\hk(Y,\ovy)\cong \mathbb{R}\Gamma_\rig(\ovy/\so_F^0)$, we  have on the complex $\mathbb{R}\Gamma^\rig_\hk(Y,\ovy)$ all the structures needed for the construction of log rigid syntomic cohomology theory.

In \S\ref{Subsec: Frobenius monodromy}, we compare the Frobenius operators on $\mathbb{R}\Gamma^\rig_\hk(Y,\ovy)$ and $\mathbb{R}\Gamma_\rig(\ovy/\so_F^0)$.
The former is defined on the level of complexes from local lifts of Frobenius.
The latter is defined via base change in \S\ref{subsubsec: Base change and Frobenius}, and suitable for the comparison with the crystalline Frobenius.

\subsection{The rigid Hyodo--Kato map}\label{subsec: HK-map}

We modify the construction  of the rigid Hyodo--Kato map  in \cite{GrosseKloenne2005} so that it is functorial, and generalize it to strictly semistable log schemes with boundary.

Let $(Y,\ovy)$ be an object in $\ssk$.
We will construct a Hyodo--Kato map $\mathbb{R}\Gamma_\rig(\ovy/\so_F^0) \rightarrow \mathbb{R}\Gamma_\rig(\ovy/\so_K^\pi)$ as a morphism in $\mathscr{D}^+(\mathsf{Mod}_F)$.
Recall that $\Upsilon_{\ovy}$ and $\Upsilon_D$ are the sets of the irreducible components and the horizontal components of $\ovy$ (Definition~\ref{def: upsilon}).
For $j\in\Upsilon_{\ovy}$, let $\mathcal{L}_j$ be the line bundle on $\ovy$ which corresponds to $j$ as in \cite[\S 2.1]{GrosseKloenne2005}. 
Precisely, let $\ovy_j$ be the irreducible component of $\ovy$ which corresponds to $j$, and let $\mathscr{N}_{\ovy,j}$ be the preimage of $\kker(\co_{\ovy} \rightarrow \co_{\ovy_j})$ in $\mathscr{N}_{\ovy}$.
Then $\mathscr{N}_{\ovy,j}$ is a principal homogeneous space over $\co_{\ovy}^\times$, and its associated line bundle is the dual $\mathcal{L}_j^{-1}$ of $\mathcal{L}_j$.
Note that the natural map $\mathscr{N}_{\ovy,j} \rightarrow \co_{\ovy}$ defines a global section $s_j$ of $\mathcal{L}_j$.
For any non-empty subset $J\subset\Upsilon_{\ovy}$, set $\ovm_J: = \bigcap_{j\in J}\ovy_j$. We regard it as an exact closed log subscheme of $\ovy$.
By abuse of notation, we denote the restriction of $\mathcal{L}_j$ to $\ovm_J$ again by $\mathcal{L}_j$. 
For $j\in J$, set
	$$
	\ovv^j_J : =  \Spec\Sym_{\co_{\ovm_J}}\mathcal{L}_j^{-1} \qquad \text{and} \qquad \ovp_J^j : =  \Proj\Sym_{\co_{\ovm_J}}(\co_{\ovm_J}\oplus\mathcal{L}_j^{-1}).
	$$
We regard $\ovv^j_J$ as an open subscheme of $\ovp_J^j$ via the isomorphism
	\begin{eqnarray*}
	&\Sym_{\co_{\ovm_J}}\mathcal{L}_j^{-1}  \xrightarrow{\sim} \Sym_{\co_{\ovm_J}}(\co_{\ovm_J}\oplus\mathcal{L}_j^{-1})[1^{-1}_{\co_{\ovm_J}}]_0\\ 
	& s  \mapsto   1^{-1}_{\co_{\ovm_J}}\otimes s \quad \text{for }s\in\mathcal{L}_j^{-1},
	\end{eqnarray*}
where $1_{\co_{\ovm_J}}$ is considered as a degree-one element of $\Sym_{\co_{\ovm_J}}(\co_{\ovm_J}\oplus\mathcal{L}_j^{-1})$, and $\Sym_{\co_{\ovm_J}}(\co_{\ovm_J}\oplus\mathcal{L}_j^{-1})[1^{-1}_{\co_{\ovm_J}}]_0$ denotes the degree-zero part of $\Sym_{\co_{\ovm_J}}(\co_{\ovm_J}\oplus\mathcal{L}_j^{-1})[1^{-1}_{\co_{\ovm_J}}]$.
For a subset $J'\subset J$, set
	$$
	\ovv^{J'}_J: = \prod_{\ovm_J}(\ovv^j_J)_{j\in J'}\qquad\text{ and }\qquad\ovp^{J'}_J: = \prod_{\ovm_J}(\ovp_J^j)_{j\in J'}.
	$$
We set $\ovv_J: = \ovv_J^J$ and $\ovp_J: = \ovp_J^J$.
For $j\in J$, the pull-back of the divisor $\ovp_J^j\setminus\ovv_J^j$ on $\ovp_J^j$ is a divisor on $\ovp_J$ denoted by $N_{j,\infty}$.
Let $N_{j,0}$ be the divisor on $\ovp_J$ which is the pull-back of the zero section divisor $\ovm_J\hookrightarrow\ovv_J^j\hookrightarrow\ovp_J^j$ on $\ovp_J^j$.
Set $N_\infty: = \bigcup_{j\in J}N_{j,\infty}$ and $N_0: = \bigcup_{j\in J}N_{j,0}$.
Let $N'$ and $D'$ be the divisors on $\ovp_J$ given by the pull-back of the divisors $\ovm_J\cap\bigcup_{i\in\Upsilon_{\ovy}\setminus J}\ovy_i$ and $\ovm_J\cap (\ovy\setminus Y)$ on $\ovm_J$, respectively, via the structure morphism $\ovp_J \rightarrow \ovm_J$.
We endow $\ovp_J$ with the log structure associated to the normal crossing divisor $N_\infty\cup N_0\cup N'\cup D'$.
We consider $\ovp_J^{J'}$, and thus in particular $\ovm_J = \ovp_J^\varnothing$, as an exact closed log subscheme of $\ovp_J$ by identifying it with the intersection in $\ovp_J$ of all $N_{j,0}$ for $j\in J\setminus J'$.
The global sections $s_i$ of $\mathcal{L}_i$ for $i\in\Upsilon_{\ovy}\setminus J$ define a map $\eta\colon\bigotimes_{\co_{\ovm_J}}(\mathcal{L}_i^{-1})_{i\in\Upsilon_{\ovy}\setminus J} \rightarrow \co_{\ovm_J}$, and hence a map
	\[\co_{\ovm_J} \cong\bigotimes_{\co_{\ovm_J}}(\mathcal{L}_j^{-1})_{j\in J}\otimes\bigotimes_{\co_{\ovm_J}}(\mathcal{L}_i^{-1})_{i\in\Upsilon_{\ovy}\setminus J} \xrightarrow{1\otimes\eta}\bigotimes_{\co_{\ovm_J}}(\mathcal{L}_j^{-1})_{j\in J} \rightarrow \Sym_{\co_{\ovm_J}}(\bigoplus_{j\in J}\mathcal{L}_j^{-1}).\]
We consider  $\ovv_J$ as a log scheme over $T$ by sending $t$ to the image of $1_{\co_{\ovm_J}}$ under this map.
Then $(\ovv_J^{J'},\ovp_J^{J'})$ is a $T$-log scheme with boundary.
As in \cite[\S 2.4]{GrosseKloenne2005} (or Proposition~\ref{prop: embeddable boundary}), one can see that  $(\ovv_J^{J'},\ovp_J^{J'})$ is $\Xi_b(\ct)$-embeddable.

For a set $\Upsilon$ and an integer $m\geqslant 0$ we set
	\[\widetilde{\Lambda}_m(\Upsilon): = \{\lambda = (J_0(\lambda),\ldots,J_m(\lambda))\mid\varnothing\neq J_0(\lambda)\subset J_1(\lambda)\subset\cdots\subset J_m(\lambda)\subset\Upsilon\}.\]
Now we define a simplicial log scheme $\ovm_{\sbt}$ and a simplicial $T$-log scheme with boundary $(\ovv_{\sbt},\ovp_{\sbt})$ by
	$$
	\ovm_m : =  \coprod_{\lambda\in \widetilde{\Lambda}_m(\Upsilon_{\ovy})}\ovm_{J_m(\lambda)} \qquad \text{and} \qquad (\ovv_m,\ovp_m): = \coprod_{\lambda\in\widetilde{\Lambda}_m(\Upsilon_{\ovy})}(\ovv_{J_m(\lambda)}^{J_0(\lambda)},\ovp_{J_m(\lambda)}^{J_0(\lambda)}).
	$$

\begin{remark}
	Here we use the set $\widetilde{\Lambda}_m(\Upsilon)$ instead of the set $\Lambda_m(\Upsilon)$ in \cite[\S 3.2]{GrosseKloenne2005}, in order to ensure the functoriality of the simplicial construction.
	More precisely,  a morphism $f\colon\ovy \rightarrow \ovy'$ in $\ssk$ naturally induces a map $\rho_f\colon\widetilde{\Lambda}_m(\Upsilon_{\ovy}) \rightarrow \widetilde{\Lambda}_m(\Upsilon_{\ovy'})$, but not necessarily a map $\Lambda_m(\Upsilon_{\ovy}) \rightarrow \Lambda_m(\Upsilon_{\ovy'})$.
	If we let $\ovm'_{\sbt}$ and $(\ovv'_{\sbt},\ovp'_{\sbt})$ be the simplicial objects constructed from $\ovy'$ as above,  $\rho_f$ induces natural morphisms $\ovm_{\sbt} \rightarrow \ovm'_{\sbt}$ and $\ovv_{\sbt} \rightarrow \ovv'_{\sbt}$.
	Note that this does not extend to a morphism $\ovp_{\sbt} \rightarrow \ovp'_{\sbt}$ in general.
	The following lemma implies that the log rigid complexes of  $(\ovv_{\sbt},\ovp_{\sbt})$ given by $\widetilde{\Lambda}_m(\Upsilon_{\ovy})$ and $\Lambda_m(\Upsilon)$ are quasi-isomorphic to each other.
\end{remark}

\begin{lemma}
Let $X$ be a Grothendieck topological space, $\{U_j\}_{j\in\Upsilon}$ an admissible open covering of $X$, and $\mathcal{F}^{\sbt}$ a complex of sheaves of abelian groups on $X$.
For $J\subset\Upsilon$ we set $U_J: = \bigcap_{j\in\Upsilon}U_j$.
Then for the simplicial space $\widetilde{U}_{\sbt}$ given by $\widetilde{U}_m: = \coprod_{\lambda\in \widetilde{\Lambda}_m(\Upsilon)}U_{J_m(\lambda)}$ we have an isomorphism
	\[ \R\Gamma(X,\mathcal{F}^{\sbt}) \cong  \R\Gamma(\widetilde{U}_{\sbt},\mathcal{F}^{\sbt} |_{\widetilde{U}_{\sbt}}).\]
\end{lemma}

\begin{proof}
We consider the set
	\[\Lambda_m(\Upsilon): = \{\lambda = (J_0(\lambda),\ldots,J_m(\lambda))\mid\varnothing\neq J_0(\lambda)\subsetneq J_1(\lambda)\subsetneq\cdots\subsetneq J_m(\lambda)\subset\Upsilon\}.\]
Then for a sheaf $\mathcal{F}$ of abelian groups on $X$ we obtain two types of Cech complexes $C^{\sbt}(X,\mathcal{F})$ and $\widetilde{C}^{\sbt}(X,\mathcal{F})$ given by
	$$
	C^m(X,\mathcal{F}): = \bigoplus_{\lambda\in \Lambda_m(\Upsilon)}\Gamma(U_{J_m(\lambda)},\mathcal{F}),\qquad
	\widetilde{C}^m(X,\mathcal{F}): = \bigoplus_{\lambda\in \widetilde{\Lambda}_m(\Upsilon)}\Gamma(U_{J_m(\lambda)},\mathcal{F}).
	$$
The natural direct decomposition 
	$$
	\widetilde{C}^m(X,\mathcal{F}) = C^m(X,\mathcal{F})\oplus \bigoplus_{\lambda\in\widetilde{\Lambda}_m(\Upsilon) \setminus\Lambda_m(\Upsilon)}\Gamma(U_{J_m(\lambda)},\mathcal{F})
	$$ 
induces an inclusion $i\colon C^{\sbt}(X,\mathcal{F}) \rightarrow \widetilde{C}^{\sbt}(X,\mathcal{F})$ and a projection $p\colon \widetilde{C}^{\sbt}(X,\mathcal{F}) \rightarrow  C^{\sbt}(X,\mathcal{F})$ which are morphisms of complexes and satisfy $p\circ i = \mathrm{id}$.
Moreover, by  a straightforward calculation one can see that $i\circ p$ and $\mathrm{id}$ are homotopic with a homotopy $h$ defined as follows.
Fix $m\geqslant 0$ and an element $\lambda\in\widetilde{\Lambda}_m(\Upsilon)$, and take positive integers $k_1,\ldots,k_n$ satisfying $J_{k_1 + \ldots + k_{j-1}}(\lambda) = \cdots = J_{k_1 + \cdots  + k_j-1}(\lambda)$ and $J_{k_1 + \cdots  + k_j-1}(\lambda)\neq J_{k_1 + \cdots + k_j}(\lambda)$ for any $j$.
Let $r$ be the number of $j$'s satisfying $k_j\geqslant 2$ if they exist, and otherwise we set $r = 1$.
For $1\leqslant j\leqslant n$, we set
	\[s_j(\lambda): = (J_0(\lambda),\ldots,J_{k_1 + \cdots + k_j-2}(\lambda),J_{k_1 + \cdots + k_j-1}(\lambda),J_{k_1 + \cdots + k_j-1}(\lambda),J_{k_1 + \cdots + k_j}(\lambda),\ldots,J_m(\lambda)).\]
Then for an element $f = (f_\mu)_{\mu\in\widetilde{\Lambda}_{m + 1}(\Upsilon)}\in \widetilde{C}^m(X,\mathcal{F})$, the $\lambda$-component of $h(f)$ is defined by
	\[h(f)_\lambda: = \sum_{j = 1}^n\frac{1}{r}(-1)^{k_1 + \cdots k_j}f_{s_j(\lambda)}.\]
Hence we see that $C^{\sbt}(X,\mathcal{F})$ and $\widetilde{C}^{\sbt}(X,\mathcal{F})$ are naturally quasi-isomorphic.
This result and \cite[Lem.~3.5]{GrosseKloenne2007} imply the statement.
\end{proof}

Let $(Y,\ovy)$ be an object in $\ssk$, and  $\ovm_{\sbt}$ and $(\ovv_{\sbt},\ovp_{\sbt})$ the associated simplicial objects defined above.
We denote the compositions
	\begin{align*}
	&\mathbb{R}\Gamma_\rig((\ovv_{\sbt},\ovp_{\sbt})/\ct) \xrightarrow{\mathbb{L}^b} \mathbb{R}\Gamma_\rig(\ovv_{\sbt}/\ct) \xrightarrow{\mathbb{L}_{\so_F^0/\ct}} \mathbb{R}\Gamma_\rig(\ovv_{\sbt}\times_Tk^0/\so_F^0)  \rightarrow  \mathbb{R}\Gamma_\rig(\ovm_{\sbt}/\so_F^0), \\
	&\mathbb{R}\Gamma_\rig((\ovv_{\sbt},\ovp_{\sbt})/\ct) \xrightarrow{\mathbb{L}^b} \mathbb{R}\Gamma_\rig(\ovv_{\sbt}/\ct)  \xrightarrow{\mathbb{L}_{\so_K^\pi/\ct}} \mathbb{R}\Gamma_\rig(\ovv_{\sbt}\times_Tk^0/\so_K^\pi)  \rightarrow \mathbb{R}\Gamma_\rig(\ovm_{\sbt}/\so_K^\pi)
	\end{align*}
by $\overline{\xi}_0$ and $\overline{\xi}_\pi$, respectively.
If we repeat the constructions for $(Y,Y)$, we obtain simplicial objects $M_{\sbt}$, $(V_{\sbt},P_{\sbt})$, and morphisms
	\begin{align*}
	& \xi_0\colon \mathbb{R}\Gamma_\rig((V_{\sbt},P_{\sbt})/\ct)  \rightarrow  \mathbb{R}\Gamma_\rig(M_{\sbt}/\so_F^0),\\
	& \xi_\pi\colon \mathbb{R}\Gamma_\rig((V_{\sbt},P_{\sbt})/\ct ) \rightarrow  \mathbb{R}\Gamma_\rig(M_{\sbt}/\so_K^\pi).
	\end{align*}
They fit together in a commutative diagram 
	{\small
	\begin{equation*}\label{eq: hk map diag}\mathclap{
	\xymatrix{
	 \mathbb{R}\Gamma_\rig(Y/\so_F^0) \ar[r]^{(b)} &  \mathbb{R}\Gamma_\rig(M_{\sbt}/\so_F^0) &  \mathbb{R}\Gamma_\rig((V_{\sbt},P_{\sbt})/\ct) \ar[l]_{\xi_0} \ar[r]^{\xi_\pi} &  \mathbb{R}\Gamma_\rig(M_{\sbt}/\so_K^\pi) &  \mathbb{R}\Gamma_\rig(Y/\so_K^\pi) \ar[l]_{(b)}\\
	 \mathbb{R}\Gamma_\rig(\ovy/\so_F^0) \ar[r]^{(b)} \ar[u]^{(a)} &  \mathbb{R}\Gamma_\rig(\ovm_{\sbt}/\so_F^0) \ar[u]& \mathbb{R}\Gamma_\rig((\ovv_{\sbt},\ovp_{\sbt})/\ct) \ar[l]_{\overline{\xi_0}} \ar[r]^{\overline{\xi_\pi}} \ar[u] &  \mathbb{R}\Gamma_\rig(\ovm_{\sbt}/\so_K^\pi) \ar[u]& \mathbb{R}\Gamma_\rig(\ovy/\so_K^\pi) \ar[l]_{(b)} \ar[u]
	}
	}\end{equation*}}
Here, the morphisms $(b)$ are quasi-isomorphisms by  cohomological descent of admissible coverings of dagger spaces.
Moreover $\xi_0$ and $\xi_\pi\otimes 1\colon \mathbb{R}\Gamma_\rig((V_{\sbt},P_{\sbt})/\ct)\otimes_FK \rightarrow  \mathbb{R}\Gamma_\rig(M_{\sbt}/\so_K^\pi)$ are quasi-isomorphisms by \cite[Thm.~3.1]{GrosseKloenne2005}.

\begin{lemma}\label{lem: qis of M V P}
	Let $(Y,\ovy)$ be an object in $\ssk$, and let $\ovm_{\sbt}$, $M_{\sbt}$, $(\ovv_{\sbt},\ovp_{\sbt})$, and $(V_{\sbt},P_{\sbt})$ be as above.
	The morphisms
		\begin{align}
		& \mathbb{R}\Gamma_\rig(\ovm_{\sbt}/\so_F^0)  \rightarrow   \mathbb{R}\Gamma_\rig(M_{\sbt}/\so_F^0),\label{eq: qis of M F}\\
		&  \mathbb{R}\Gamma_\rig(\ovm_{\sbt}/\so_K^\pi)  \rightarrow   \mathbb{R}\Gamma_\rig(M_{\sbt}/\so_K^\pi),\label{eq: qis of M K}\\
		&  \mathbb{R}\Gamma_\rig((\ovv_{\sbt},\ovp_{\sbt})/\ct)  \rightarrow   \mathbb{R}\Gamma_\rig((V_{\sbt},P_{\sbt})/\ct)\label{eq: qis of V P}
		\end{align}
	are quasi-isomorphisms.
\end{lemma}

\begin{proof}
	By taking an affine open covering of $\ovy$, we may assume that $\ovy$ is affine and admits a global chart as in Definitin~\ref{def: ss log scheme} (i).
	As in the proof of Lemma~\ref{lem: HK embedding}, one can take an exact closed immersion $(Y,\ovy)\hookrightarrow(\cz,\ovcz)$ into an object in $\sst$ with $\ovy \cong\ovcz\times_\ct k^0$ and $Y \cong\cz\times_\ct k^0$.
	Set $\ovcy: = \ovcz\times_\ct\so_F^0$ and $\cd: = \ovcz\setminus\cz$.
	Let $Y'$ be the exact closed log subscheme of $\ovcz$ whose underlying scheme is the special fibre of $\ovcy\cup\cd$.
	Then $(Y',Y')$ is an object in $\ssk$, and there is a natural bijection $\Upsilon_{Y'} \cong\Upsilon_{\ovy}\amalg\Upsilon_D$.
	Through this bijection we regard $\Upsilon_{\ovy}$ and $\Upsilon_D$ as subsets of $\Upsilon_{Y'}$.
	Let $J\subset \Upsilon_{\ovy}$ be a non-empty subset, and let $M'_J: = \bigcap_{j\in J}Y'_j$ be the log scheme associated to $(Y',Y')$.
	Then we have isomorphisms
		$$
		\ovm_J \cong M'_J\qquad\text{and} \qquad  M_J \cong M'_J\setminus(M'_J\cap\bigcup_{i\in\Upsilon_D}Y'_i).
		$$
	 If we set $M'_{J,\heartsuit}: = M'_J\setminus(M'_J\cap \bigcup_{i\in\Upsilon_{Y'}\setminus J}Y'_i)$, we can apply \cite[Lem.~4.4]{GrosseKloenne2005}  to obtain that the morphisms
		\begin{align*}
		& \mathbb{R}\Gamma_\rig(M'_J/\so_F^0) \rightarrow  \mathbb{R}\Gamma_\rig(M'_{J,\heartsuit}/\so_F^0)\quad \text{ and } \quad\\ &\mathbb{R}\Gamma_\rig((M'_J\setminus(M'_J\cap \bigcup_{i\in\Upsilon_D}Y'_i))/\so_F^0) \rightarrow  \mathbb{R}\Gamma_\rig(M'_{J,\heartsuit}/\so_F^0)
		\end{align*}
	are quasi-isomorphisms.
	Hence we see that \eqref{eq: qis of M F} is a quasi-isomorphism.
	The other quasi-isomorphisms can be proved similarly.
\end{proof}

The commutativity of the diagram before Lemma \ref{lem: qis of M V P}  implies the following statement.

\begin{corollary}\label{cor: qis of Y}
	In the setting of Lemma~\ref{lem: qis of M V P}, the morphisms
	\begin{align*}
		& \mathbb{R}\Gamma_\rig((\ovv_{\sbt},\ovp_{\sbt})/\ct) \xrightarrow{\overline{\xi}_0} \mathbb{R}\Gamma_\rig(\ovm_{\sbt}/\so_F^0),\\
		& \mathbb{R}\Gamma_\rig((\ovv_{\sbt},\ovp_{\sbt})/\ct)\otimes_FK \xrightarrow{\overline{\xi}_\pi\otimes 1} \mathbb{R}\Gamma_\rig(\ovm_{\sbt}/\so_K^\pi),\\
		& \mathbb{R}\Gamma_\rig(\ovy/\so_K^\pi) \rightarrow  \mathbb{R}\Gamma_\rig(Y/\so_K^\pi),\\
		& \mathbb{R}\Gamma_\rig(\ovy/\so_F^0) \rightarrow  \mathbb{R}\Gamma_\rig(Y/\so_F^0)\\
	\end{align*}
	are quasi-isomorphisms.
\end{corollary}

We did not have to assume that $\ovy$ is proper to obtain the results in this corollary and the previous lemma.
As a consequence, the lower horizontal line in the diagram before Lemma \ref{lem: qis of M V P} gives a morphism
	\begin{align*}
	\iota^\rig_\pi\colon \mathbb{R}\Gamma_\rig(\ovy/\so_F^0) \rightarrow  \mathbb{R}\Gamma_\rig(\ovy/\so_K^\pi)
	\end{align*}
in $\mathscr{D}^+(\mathsf{Mod}_F)$, which induces a quasi-isomorphism after tensoring with $K$.

\begin{proposition}\label{prop: functoriality of HK}
	The morphism $\iota^\rig_\pi$ is functorial in $(Y,\ovy)$.
\end{proposition}

To prove this proposition, let $f\colon (Y,\ovy) \rightarrow  (Y',\ovy')$ be a morphism in $\ssk$, and $\rho = \rho_f\colon\Upsilon_{\ovy} \rightarrow \Upsilon_{\ovy'}$ the map induced by $f$.
For $I\subset J\subset\Upsilon_{\ovy}$, we set $J': = \rho(J)\subset\Upsilon_{\ovy'}$ and $I': = \rho(I)\subset I$.
Denote by $\ovm_J$, $(\ovv_J^I,\ovp_J^I)$ and $\ovm'_{J'}$, $({\ovv'}^{I'}_{J'},{\ovp'}_{J'}^{I'})$ the log schemes (with boundary) associated to $(Y,\ovy)$ and $(Y',\ovy')$.
Then $f$ induces morphisms $\ovm_J \rightarrow  \ovm'_{J'}$ and $\ovv_J^I \rightarrow  {\ovv'}_{J'}^{I'}$, but not $\ovp_J^I \rightarrow {\ovp'}_{J'}^{I'}$ in general.
{As a remedy, we consider} the log schematic image $\widetilde{P}_J^I$ of the diagonal map $\ovv_J^I \rightarrow  \ovp_J^I\times_k {\ovp'}_{J'}^{I'}$.

\begin{lemma}\label{lem: compute tilde P}
	Let $f\colon(Y,\ovy) \rightarrow (Y',\ovy')$ be a morphism in $\ssk$ and let $I\subset J\subset\Upsilon_{\ovy}$.
	With the notation as above, $(\ovv_J^I,\widetilde{P}_J^I)$ is a locally $\ct$-embeddable $T$-log scheme with boundary, that is $(\ovv_J^I,\widetilde{P}_J^I)\in \lelsb$.
\end{lemma}

\begin{proof}
	Without loss of generality we assume that there are $\Xi_\hk^\rig$-rigid data $((\cz,\ovcz,\phi),i)$ for $(Y,\ovy)$ and $((\cz',\ovcz',\phi'),i')$ for $(Y',\ovy')$.
		Furthermore, let $A\in\mathsf{SET}^0_{\Xi_\hk^\rig}(Y,\ovy)$ be a multiset whose multiplicity is one at $(\mathrm{id}_{(Y,\ovy)},(Y,\ovy),(\cz,\ovcz,\phi),i)$ and $(f,(Y',\ovy'),(\cz',\ovcz',\phi'),i')$, and zero otherwise.
	Let $((\cz_A,\ovcz_A,\phi_A),i_A)$ be the $\Xi_\hk^\rig$-rigid datum for $(Y,\ovy)$ associated to $A$ constructed as in \S\ref{subsubsec: Rigid Hyodo-Kato cohomology}.
	As in the proof of Proposition~\ref{prop: datum for A} there is a commutative diagram
		\[\xymatrix{
		\ovy \ar[r] \ar[d]&\ovcz_A \ar[d] \ar[r]&\Spwf\frac{\so_F[t_\gamma,s_\delta,t_{\gamma'}^{-1},u_\alpha^{\pm 1},v_\beta^{\pm 1}]^\dagger_{\alpha\in\Upsilon_{\ovcy'},\beta\in\Upsilon_{\cd'},\gamma\in\Upsilon_{\ovcy},\gamma'\in\Upsilon_{\ovcy}\setminus\Upsilon_{\ovy},\delta\in\Upsilon_{\cd}}}{(\prod_{\alpha\in\Upsilon_{\ovcy'}}u_\alpha-1)} \ar[d]\\
		\ovy' \ar[r]&\ovcz' \ar[r]&\Spwf\so_F[t'_\alpha, s'_\beta]^\dagger_{\alpha\in\Upsilon_{\ovcy'},\beta\in\Upsilon_{\cd'}}.
		}\]
	Here the right horizontal arrows are smooth. The vertical arrow on the right is defined by
		$$
		t'_\alpha \mapsto  u_\alpha\prod_{\gamma\in \Upsilon_{\ovy}^\alpha}t_\gamma\qquad\text{ and }\qquad s'_\beta \mapsto v_\beta \prod_{\delta\in\Upsilon_D}s_\delta^{m_{\beta,\delta}},
		$$
where $\Upsilon_{\ovy}^\alpha = \{\gamma\in\Upsilon_{\ovy}\mid(i'\circ f)(\ovy_\gamma)\subset\ovcy'_\alpha\}$ and $m_{\beta,\delta}$ is the multiplicity of $(i\circ f)^\ast \cd'_\beta$ at $D_\delta$.
	Consider the diagram
		{\small\[\mathclap{
		\xymatrix{
		\co_{\ovcz_A} &\frac{\so_F[t_\gamma,s_\delta,t_{\gamma'}^{-1},u_\alpha^{\pm 1},v_\beta^{\pm 1}]^\dagger_{\alpha\in\Upsilon_{\ovcy'},\beta\in\Upsilon_{\cd'},\gamma\in\Upsilon_{\ovcy},\gamma'\in\Upsilon_{\ovcy}\setminus\Upsilon_{\ovy},\delta\in\Upsilon_{\cd}}}{(\prod_{\alpha\in\Upsilon_{\ovcy'}}u_\alpha-1)} \ar[l] & \mathbb{N}^{\Upsilon_{\ovcy}} \oplus\mathbb{N}^{\Upsilon_\cd} \oplus\mathbb{Z}^{\Upsilon_{\ovcy'}} \oplus\mathbb{Z}^{\Upsilon_{\cd'}} = :Q_A \ar[l]_>>>>>>\kappa\\
		\co_{\ovcz'} \ar[u] & \so_F[t'_\alpha,s'_\gamma]^\dagger_{\alpha\in\Upsilon_{\ovcy'},\beta\in\Upsilon_{\cd'}} \ar[u] \ar[l] &\mathbb{N}^{\Upsilon_{\ovcy'}}\oplus\mathbb{N}^{\Upsilon_{\cd'}} = :Q', \ar[l]_{\kappa'} \ar[u]_c
		}
		}\]}
where $\kappa$ is defined by
		\begin{align*}
		&1_\gamma \mapsto  t_\gamma\ \ (\gamma\in\Upsilon_{\ovcy}),&1_\delta  \mapsto  s_\delta\ \ (\delta\in\Upsilon_{\cd}),\\
		&1_\alpha \mapsto u_\alpha\ \ (\alpha\in\Upsilon_{\ovcy'}),&1_\beta \mapsto  v_\beta\ \ (\beta\in\Upsilon_{\cd'}),
		\end{align*}
	$\kappa'$ is defined by
		\begin{align*}
		&1_\alpha \mapsto  t'_\alpha\ \ (\alpha\in\Upsilon_{\ovcy'}),&&1_\beta \mapsto  s'_\beta\ \ (\beta\in\Upsilon_{\cd'}),
		\end{align*}
	and $c$ is defined by
		\begin{align*}
		&1_\alpha \mapsto  1_\alpha + \sum_{\gamma\in\Upsilon_{\ovcy}^\alpha}1_\gamma\ \ (\alpha\in\Upsilon_{\ovcy'}),&&1_\beta \mapsto  1_\beta + \sum_{\delta\in\Upsilon_{\cd}}m_{\beta,\delta}1_\delta\ \ (\beta\in\Upsilon_{\cd'}).
		\end{align*}
	Here $1_\alpha$ denotes $1$ in the $\alpha$-component of $\mathbb{N}^{\Upsilon_{\ovcy'}}$.
	We used a similar notation for $1_\beta$, $1_\gamma$, and $1_\delta$.
	The above diagram gives a chart of $\ovcz_A \rightarrow \ovcz'$.
	Set
		\begin{align*}
		&\overline{\mathcal{M}}_J: = \bigcap_{j\in J}\ovcy_{A,j}, &&\overline{\mathcal{V}}_J: = \overline{\mathcal{M}}_J\times\ct^J, &&\overline{\mathcal{P}}_J: = \overline{\mathcal{M}}_J\times\ovct^J,\\
		&\overline{\mathcal{M}}'_{J'}: = \bigcap_{j\in J'}\ovcy'_j, &&\overline{\mathcal{V}}'_{J'}: = \overline{\mathcal{M}}'_{J'}\times\ct^{J'},&& \overline{\mathcal{P}}_{J'}: = \overline{\mathcal{M}}'_{J'}\times\ovct^{J'}.
		\end{align*}
	For $\gamma\in J$ (resp. $\alpha\in J'$), we denote the coordinate of the $\gamma$-component (resp. $\alpha$-component) of  $\ct^J$ (resp. $\ct^{J'}$) by $x_\gamma$ (resp. $y_\alpha$).
	We regard $\overline{\mathcal{V}}_J$ and $\overline{\mathcal{V}}'_{J'}$ as log schemes over $\ct$ via the maps
		$$
		t \mapsto \prod_{\gamma\in\Upsilon_{\ovcy}\setminus J}t_\gamma\cdot\prod_{\gamma'\in J}x_{\gamma'}\qquad \text{ and }\qquad
		t \mapsto  \prod_{\alpha\in\Upsilon_{\ovcy'}\setminus J'}t'_\alpha\cdot\prod_{\alpha'\in J'}y_{\alpha'}
		$$
	respectively.
	The morphism $\overline{\mathcal{V}}_J  \rightarrow \overline{\mathcal{V}}'_{J'}$ defined by
	$$
	y_\alpha \mapsto  u_\alpha\cdot\prod_{\gamma\in\Upsilon_{\ovy}^\alpha \setminus J}t_\gamma\cdot\prod_{\gamma'\in\Upsilon_{\ovy}^\alpha \cap J}x_{\gamma'}
	$$
	for $\alpha\in J'$ lifts $\ovv_J \rightarrow \ovv'_{J'}$.
	Let $\widetilde{\mathcal{P}}_J$ be the (weak formal) log schematic image of the diagonal embedding $\overline{\mathcal{V}}_J\hookrightarrow \overline{\mathcal{P}}_J \times\overline{\mathcal{P}}'_{J'}$.
	Note that $\overline{\mathcal{P}}_J\times\overline{\mathcal{P}}'_{J'}$ is covered by
		\[\overline{\mathcal{U}}_\varepsilon : =  \overline{\cm}_J\times\overline{\cm}'_{J'} \times \Spwf\so_F[x_\gamma^{\varepsilon_\gamma}, y_\alpha^{\varepsilon_\alpha}]^\dagger_{\gamma\in J, \alpha\in J'} \subset \overline{\mathcal{P}}_J \times \overline{\mathcal{P}}'_{J'}\]
	for all $\varepsilon = ((\varepsilon_\gamma)_{\gamma\in J},(\varepsilon_\alpha)_{\alpha\in J'})\in \{\pm 1\}^J\times\{\pm 1\}^{J'}$.
	Let
		$$
		\widetilde{\mathcal{U}}_\varepsilon : =  \overline{\mathcal{U}}_\varepsilon\cap \widetilde{\mathcal{P}}_J \qquad \text{ and } \qquad \mathcal{U}_\varepsilon : =  \widetilde{\mathcal{U}}_\varepsilon\cap \overline{\mathcal{V}}_J,
		$$
	and endow them with the log structure which is the pull-back of the log structure on $\widetilde{\mathcal{P}}_J$.
	We define a map $\psi_J\colon Q_A \rightarrow  Q_A^\gp$ by
		\begin{align*}
		&1_\gamma \mapsto 
		\begin{cases}
		1_\gamma&(\gamma\in\Upsilon_{\ovcy}\setminus J,\text{ or }\gamma\in J \text{ and }\varepsilon_\gamma = 1),\\
		-1_\gamma&(\gamma\in J\text{ and }\varepsilon_\gamma = -1),
		\end{cases}
		&&1_\delta \mapsto  1_\delta\ \ (\delta\in\Upsilon_{\cd}),\\
		&1_\alpha \mapsto  1_\alpha\ \ (\alpha_\in\Upsilon_{\ovcy'}),
		&&1_\beta \mapsto  1_\beta\ \ (\beta_\in\Upsilon_{\cd'}),
		\end{align*} 
	and a second map  $\psi'_{J'}\colon Q' \rightarrow  {Q'}^\gp$ by 
		\begin{align*}
		&1_\alpha \mapsto 
		\begin{cases}
		1_\alpha&(\alpha\in\Upsilon_{\ovcy'}\setminus J',\text{ or }\alpha\in J' \text{ and }\varepsilon_\alpha = 1),\\
		-1_\alpha&(\alpha\in J'\text{ and }\varepsilon_\alpha = -1),
		\end{cases}
		&&1_\beta \mapsto  1_\beta\ \ (\beta\in\Upsilon_{\cd'}).
		\end{align*}
	Let $Q$ be the image of
		\[(\psi_J,c^\gp\circ\psi'_{J'})\colon Q_A\oplus Q' \rightarrow  Q_A^\gp.\]
	Then $(\mathcal{U}_\varepsilon,\widetilde{\mathcal{U}}_\varepsilon)$ has a chart $(\xi\colon Q_{\widetilde{\mathcal{U}}_\varepsilon} \rightarrow \cm_{\widetilde{\mathcal{U}}_\varepsilon},\zeta\colon\mathbb{N} \rightarrow  Q^\gp)$, where $\xi$ is induced by the map $Q_{A,\overline{\mathcal{V}}_J}^\gp \rightarrow \cm_{\overline{\mathcal{V}}_J}^\gp$ which sends
		\begin{align*}
		&1_\gamma \mapsto  t_\gamma\ (\gamma\in\Upsilon_{\ovcy}\setminus J),&&1_{\gamma'} \mapsto  x_{\gamma'}\ (\gamma'\in J),&&1_\delta \mapsto  s_\delta\ (\delta\in\Upsilon_{\cd}),\\
		&1_\alpha \mapsto  u_\alpha\ (\alpha\in\Upsilon_{\ovcy'}),&&1_\beta \mapsto  v_\beta\ (\beta\in\Upsilon_{\cd'}),&&
		\end{align*}
		and $\zeta$ sends $1 \mapsto \sum_{\gamma\in\Upsilon_{\ovcy}\setminus J}1_\gamma + \sum_{\gamma'\in J}1_{\gamma'}$.
		Set $a: = -\sum_{\gamma\in J,\ \varepsilon_\gamma = -1}1_\gamma\in Q$ and $b: = \sum_{\gamma\in\Upsilon_{\ovcy}\setminus J}1_\gamma + \sum_{\gamma'\in J,\ \varepsilon_{\gamma'} = 1}1_{\gamma'}\in Q$.
		Then the chart $(\xi,\zeta)$ and $a,b\in Q$ satisfy the conditions in Definition~\ref{def: strongly smooth}.
		Since we have boundary immersions $(\ovv_J^I,\widetilde{P}_J^I)\hookrightarrow(\ovv_J^J,\widetilde{P}_J^J)\hookrightarrow(\overline{\mathcal{V}}_J,\widetilde{\mathcal{P}}_J)$, the lemma holds.
\end{proof}

\begin{lemma}\label{lem: qis with tilde boundary}
	The morphism $ \mathbb{R}\Gamma_\rig((\ovv_J^I,\ovp_J^I)/\ct) \rightarrow  \mathbb{R}\Gamma_\rig((\ovv_J^I,\widetilde{P}_J^I)/\ct)$ induced by the natural projection $(\ovv_J^I,\widetilde{P}_J^I) \rightarrow (\ovv_J^I,\ovp_J^I)$ is a quasi-isomorphism.
\end{lemma}

\begin{proof}
	Again we can work locally on $\ovy$ and $\ovy'$.
	Thus assume that $\ovy$ and $\ovy'$ are affine.
	Let $(\overline{\mathcal{V}}_J,\widetilde{\mathcal{P}}_J)$ be as in the proof of Lemma~\ref{lem: compute tilde P}, and $\pi\colon\widetilde{\mathcal{P}}_{J,\Q} \rightarrow \overline{\mathcal{P}}_{J,\Q}$ be the morphism of dagger spaces induced by the natural projection.
	By the computation in the proof of Lemma~\ref{lem: compute tilde P}, we see that for any $i\geqslant 0$ we have $\omega^i_{(\overline{\mathcal{V}}_J,\widetilde{\mathcal{P}}_J)/\ct,\Q} \cong\pi^\ast\omega^i_{(\overline{\mathcal{V}}_J,\overline{\mathcal{P}}_J)/\ct,\Q}$ and it is free over $\co_{\widetilde{\mathcal{P}}_{J,\Q}}$.
	Now we claim that
	\begin{equation}\label{eq: claim}
	\R\pi_\ast\mathcal{O}_{\widetilde{\mathcal{P}}_{J,\Q}} \cong\mathcal{O}_{\overline{\mathcal{P}}_{J,\Q}}.
	\end{equation}
	If this claim holds, we have $\R\pi_\ast\omega^i_{(\overline{\mathcal{V}}_J,\widetilde{\mathcal{P}}_J)/\ct,\Q} \cong\omega^i_{(\overline{\mathcal{V}}_J,\overline{\mathcal{P}}_J)/\ct,\Q}$.
	Since $]\widetilde{P}_J^I[_{\widetilde{\mathcal{P}}_J} = \pi^{-1}(]\ovp_J^I[_{\overline{\mathcal{P}}_J})$, we obtain
		\[H^j(]\widetilde{P}_J^I[_{\widetilde{\mathcal{P}}_J},\omega^i_{(\overline{\mathcal{V}}_J,\widetilde{\mathcal{P}}_J)/\ct,\Q}) \cong H^j(]\ovp_J^I[_{\overline{\mathcal{P}}_J},\omega^i_{(\overline{\mathcal{V}}_J,\overline{\mathcal{P}}_J)/\ct,\Q})\]
	for any $i,j\geqslant 0$.
	Thus we have an isomorphism between the spectral sequences
		\begin{align*}
		&E_1^{i,j} = H^j(]\widetilde{P}_J^I[_{\widetilde{\mathcal{P}}_J},\omega^i_{(\overline{\mathcal{V}}_J,\widetilde{\mathcal{P}}_J)/\ct,\Q})\Rightarrow H^{i + j}(]\widetilde{P}_J^I[_{\widetilde{\mathcal{P}}_J},\omega^{\sbt}_{(\overline{\mathcal{V}}_J,\widetilde{\mathcal{P}}_J)/\ct,\Q}) = H^{i + j}_\rig((\ovv_J^I,\widetilde{P}_J^I)/\ct)
		&&\\
		&E_1^{i,j} = H^j(]\ovp_J^I[_{\overline{\mathcal{P}}_J},\omega^i_{(\overline{\mathcal{V}}_J,\overline{\mathcal{P}}_J)/\ct,\Q}) \Rightarrow H^{i + j}(]\ovp_J^I[_{\overline{\mathcal{P}}_J},\omega^{\sbt}_{(\overline{\mathcal{V}}_J,\overline{\mathcal{P}}_J)/\ct,\Q}) = H^{i + j}_\rig((\ovv_J^I,\ovp_J^I)/\ct),&&
		\end{align*}
	and this implies the lemma.
	It remains to prove the claim \eqref{eq: claim}.
	For any subset $I\subset J$ and $L\subset J'$ let
		\begin{align*}
		& \overline{\mathcal{V}}_J^I : =  \overline{\cm}_J\times\ct^I, && \overline{\mathcal{P}}_J^I : =  \overline{\cm}_J\times\ovct^I,\\
		& {\overline{\mathcal{V}}'}_{J'}^{L} : =  \overline{\cm}'_{J'}\times\ct^{L}, && {\overline{\mathcal{P}}'}_{J'}^{L} : =  \overline{\cm}'_{J'}\times\ovct^{L}.
		\end{align*}
	Consider the composition of natural morphisms $\overline{\mathcal{V}}_J \rightarrow {\overline{\mathcal{V}}'}_{J'} \rightarrow {\overline{\mathcal{V}}'}_{J'}^{L} \rightarrow {\overline{\mathcal{P}}'}_{J'}^{L}$, where the second map is induced by the natural projection $\ct^{J'} \rightarrow \ct^{L}$.
	Let $\widetilde{\mathcal{P}}_{J,L}$ be the (weak formal) log schematic image of the diagonal immersion $\overline{\mathcal{V}}_J \rightarrow \overline{\mathcal{P}}_J\times{\overline{\mathcal{P}}'}_{J'}^{L}$.
	Then for all subsets $L_1,L_2\subset J'$ with $L_1\cap L_2 = \emptyset$, we have $\widetilde{\mathcal{P}}_{J,L_1\cup L_2} = \widetilde{\mathcal{P}}_{J,L_1}\times_{\overline{\mathcal{P}}_J} \widetilde{\mathcal{P}}_{J,L_2}$.
	Let $\pi^L\colon\widetilde{\mathcal{P}}_{J,L,\Q} \rightarrow \overline{\mathcal{P}}_{J,\Q}$ be the morphism induced by the natural projection.
	Note that   $\widetilde{\mathcal{P}}_{J,J'} = \widetilde{\mathcal{P}}_J$ and $\pi^{J'} = \pi$.
	Recall that locally rigid Hyodo-Kato data can be constructed as   weak completion of a usual scheme (see the proof of Lemma~\ref{lem: HK embedding}), hence one can apply all constructions for schemes, and obtain a morphism of schemes $\widetilde{Z}_{J,L} \rightarrow \overline{Z}_J$ whose weak completion is $\widetilde{\mathcal{P}}_{J,L} \rightarrow \overline{\mathcal{P}}_J$.
	Let $\pi^L_s\colon \widetilde{Z}_{J,L,\Q} \rightarrow \overline{Z}_{J,\Q}$ be the induced morphism.
	By \cite[Thm.~A]{Meredith1972}, it suffices to show that 	
	$$
	\R\pi^{J'}_{s,\ast}\co_{\widetilde{Z}_{J,J',\Q}} \cong\co_{\overline{Z}_{J,\Q}}.
	$$
	We will prove that $\R\pi^L_{s,\ast}\co_{\widetilde{Z}_{J,L,\Q}} \cong\co_{\overline{Z}_{J,\Q}}$ for any $L$ by induction on $\lvert L\rvert$.
	Note that  $\pi^L_{s,\ast}\co_{\widetilde{Z}_{J,L,\Q}} = \co_{\overline{Z}_{J,\Q}}$ since $\overline{Z}_{J,\Q}$ is normal.
	If $\lvert L\rvert = 1$, we have $\R\pi^L_{s,\ast}\co_{\widetilde{Z}_{J,L,\Q}} = 0$ by \cite[Cor.~11.2]{Hartshorne}.
	Let $g\colon X \rightarrow \widetilde{Z}_{J,L,\Q}$ be a resolution of singularities.
	Since $X$ is smooth,  $\R(\pi^L_s\circ g)_\ast\co_X = \co_{\ovz_{J,\Q}}$ holds automatically.
	Hence via the Grothendieck--Leray spectral sequence
		\[E_2^{p,q} = \R^p\pi^L_{s,\ast}\R^qg_\ast\co_X\Rightarrow\R^{p + q}(\pi^L_s\circ g)_\ast\co_X,\]
	we obtain $E_2^{1,0} = \R^1\pi^L_{s,\ast}\co_{\widetilde{Z}_{J,L,\Q}} = 0$.
	Assume $\lvert L\rvert\geqslant 2$, fix an element $\ell\in L$, and set $L': = L\setminus\{\ell\}$.
	Since $\pi^{\{\ell\}}_s$ is flat, we have
		\begin{equation}\label{eq: Higher direct}
		\R\pi^{L,\{\ell\}}_{s,\ast}\co_{\widetilde{Z}_{J,L,\Q}} = \R\pi^{L,\{\ell\}}_{s,\ast}\pi^{L,L',\ast}_s\co_{\widetilde{Z}_{J,L',\Q}} = \pi^{\{\ell\},\ast}_s\R\pi^{L'}_{s,\ast}\co_{\widetilde{Z}_{J,L',\Q}} = \co_{\widetilde{Z}_{J,\{\ell\},\Q}},
		\end{equation}
	where $\pi^{L,L'}_s\colon\widetilde{Z}_{J,L,\Q} \rightarrow \widetilde{Z}_{J,L',\Q}$ and $\pi^{L,\{\ell\}}_s\colon\widetilde{Z}_{J,L,\Q} \rightarrow \widetilde{Z}_{J,\{\ell\},\Q}$ are natural projections, and the third equality is given by the induction hypothesis.
	Now we also have $\R\pi^{\ell}_{s,\ast}\co_{\widetilde{Z}_{J,\{\ell\},\Q}} = \co_{\ovz_{J,\Q}}$.
	Hence by combining this with \eqref{eq: Higher direct} we obtain $\R\pi^L_{s,\ast}\co_{\widetilde{Z}_{J,L,\Q}} = \co_{\ovz_{J,\Q}}$ as desired.
\end{proof}

Consequently, for a morphism $(Y,\ovy) \rightarrow (Y',\ovy')$ in $\ssk$, we obtain a commutative diagram
	{\small\[\mathclap{
	\xymatrix{
	 \mathbb{R}\Gamma_\rig(\ovy/\so_F^0) \ar[r]& \mathbb{R}\Gamma_\rig(\ovm_{\sbt}/\so_F^0)& \mathbb{R}\Gamma_\rig((\ovv_{\sbt},\ovp_{\sbt})/\ct) \ar[l]_{\overline{\xi_0}} \ar[r]^{\overline{\xi_\pi}} \ar[d]^{ \cong}& \mathbb{R}\Gamma_\rig(\ovm_{\sbt}/\so_K^\pi)& \mathbb{R}\Gamma_\rig(\ovy/\so_K^\pi) \ar[l]\\
	&& \mathbb{R}\Gamma_\rig((\ovv_{\sbt},\widetilde{P}_{\sbt})/\ct)&&\\
	 \mathbb{R}\Gamma_\rig(\ovy'/\so_F^0) \ar[r] \ar[uu]& \mathbb{R}\Gamma_\rig(\ovm'_{\sbt}/\so_F^0) \ar[uu]& \mathbb{R}\Gamma_\rig((\ovv'_{\sbt},\ovp'_{\sbt})/\ct) \ar[l]_{\overline{\xi_0}} \ar[r]^{\overline{\xi_\pi}} \ar[u]& \mathbb{R}\Gamma_\rig(\ovm'_{\sbt}/\so_K^\pi) \ar[uu]& \mathbb{R}\Gamma_\rig(\ovy'/\so_K^\pi). \ar[l] \ar[uu]
	}}\]}
This gives the functoriality of the rigid Hyodo-Kato map and finishes the proof of Proposition~\ref{prop: functoriality of HK}.

\subsection{Frobenius on the rigid Hyodo--Kato complex}\label{Subsec: Frobenius monodromy}

We will now show that the locally explicit Frobenius action on the rigid Hyodo--Kato complex is on the level of cohomology given by base change and by the relative Frobenius via functoriality.

Let   $(Y,\ovy) \in \ssk$. On the one hand, Definition~\ref{def: Kim - Hain} induces a Frobenius endomorphism $\varphi_{\hk}$ on $ \mathbb{R}\Gamma^\rig_{\hk}(Y,\ovy)$ coming from local data. 
On the other hand, in (\ref{eq: base change Frobenius}) we obtained a Frobenius endomorphism  $\varphi$ on $ \mathbb{R}\Gamma_\rig(\ovy/\so_F^0)$ by base change and functoriality. We will now see, that they are compatible with the comparison quasi-isomorphism between  $ \mathbb{R}\Gamma^\rig_{\hk}(Y,\ovy)$ and $ \mathbb{R}\Gamma_\rig(\ovy/\so_F^0)$.

\begin{proposition}\label{prop: rigid Frobenius}
	For  $(Y,\ovy)\in\ssk$, the isomorphism $\Theta_{\hk}\colon  \mathbb{R}\Gamma^\rig_{\hk}(Y,\ovy) \rightarrow  \mathbb{R}\Gamma_\rig(\ovy/\so_F^0) $ is compatible with the Frobenius endomorphism. 
	In other words, there is a commutative diagram
	$$
	\xymatrix{
	 \mathbb{R}\Gamma^\rig_{\hk}(Y,\ovy)  \ar[r]^{\varphi_{\hk}}  \ar[d]^{\sim}_{\Theta_{\hk}} &  \mathbb{R}\Gamma^\rig_{\hk}(Y,\ovy) \ar[d]^{\sim}_{\Theta_{\hk}}\\
	 \mathbb{R}\Gamma_\rig(\ovy/\so_F^0)  \ar[r]^\varphi&  \mathbb{R}\Gamma_\rig(\ovy/\so_F^0).
	}$$
\end{proposition}

\begin{proof}
As explained in Remark~\ref{rem: gluing lifts} we can take a simplicial object $((Y_{\sbt},\overline{Y}_{\sbt}),(\cz_{\sbt},\ovcz_{\sbt},\phi_{\sbt}),i_{\sbt})$ in $\rqhk$ such that $\{(Y_{\sbt},\overline{Y}_{\sbt})\}$ is a Zariski covering of $(Y,\ovy)$.  
Let $(\overline{Y}_{\sbt}\overline{\mathcal{Y}}_{\sbt},i_{\sbt})$ be the simplicial object in $\RTOF$ given by $\overline{\mathcal{Y}}_{\sbt}: = \ovcz_{\sbt}\times_\ct\so_F^0$.
Then we have a commutative diagram of quasi-isomorphisms
		{\small\[\mathclap{
		\xymatrix{	
		 \mathbb{R}\Gamma^\rig_{\hk}(Y,\ovy) \ar[r]^\sim \ar[d]^{\mathbb{L}^{\hk}}
		&\widehat{\mathcal{K}}_\hk(Y_{\sbt},\overline{Y}_{\sbt})  \ar[d]^{\widehat{L}^{\hk}_{\mathrm{coh}}}
		&\widetilde{\mathcal{K}}_\hk((Y_{\sbt}, \overline{Y}_{\sbt}),(\cz_{\sbt},\ovcz_{\sbt},\phi_{\sbt}),i_{\sbt}) \ar[l]_<<<<\sim \ar[r]^\sim  \ar[d]^{\widetilde{L}^{\hk}_{\mathrm{coh}}}
		& \mathcal{K}_\hk((Y_{\sbt},\overline{Y}_{\sbt}),(\cz_{\sbt},\ovcz_{\sbt},\phi_{\sbt}),i_{\sbt}) \ar[d]^{L^{\hk}}\\
		 \mathbb{R}\Gamma_\rig(\ovy/\so_F^0) \ar[r]^\sim
		&\widehat{\mathcal{K}}_\rig(\overline{Y}_{\sbt})
		&\widetilde{\mathcal{K}}_\rig(\overline{Y}_{\sbt},\ovcy_{\sbt},i_{\sbt}) \ar[l]_\sim \ar[r]^\sim
		&\mathcal{K}_\rig(\overline{Y}_{\sbt},\ovcy_{\sbt},i_{\sbt}),
		}
		}\]}
	where we denote by $\mathcal{K}_\hk$ and $\mathcal{K}_\rig$ the realization functors of $\Xi_\hk^\rig$ and $\Xi(\so_F^0)$ respectively, and where $\widehat{\mathcal{K}}_\hk$ and $\widetilde{\mathcal{K}}_\hk$, as well as $\widehat{\mathcal{K}}_\rig$ and $\widetilde{\mathcal{K}}_\rig$ are defined appropriately as  in \ref{Thm: cohomology functors}.
	We  denote by $\rho_\rig$ and $\rho^\rig_{\hk}$ the compositions of the upper, respectively   lower,  horizontal arrows  in the above diagram.
	Since $\rho^\rig_{\hk}$ is a morphism in $\mathscr{D}^+(\mathsf{Mod}_F(\varphi,N))$, the Frobenius endomorphisms $\varphi_{\hk}$ on $ \mathbb{R}\Gamma^\rig_{\hk}(Y,\ovy)$ and $ \mathcal{K}_\hk((Y_{\sbt},\overline{Y}_{\sbt}),(\cz_{\sbt},\ovcz_{\sbt},\phi_{\sbt}),i_{\sbt})$ induced locally by Definition~\ref{def: Kim - Hain}  are compatible with each other via $\rho^\rig_{\hk}$.
	
	On the other hand, since the Frobenius $\phi_{\sbt}$ of the simplicial rigid Hyodo--Kato datum induces  a lift of the absolute Frobenius on $\ovcy_{\sbt}$, we obtain a lift $F_{\overline{\cy}_{\sbt}}\colon  \overline{\cy}_{\sbt}  \rightarrow   \overline{\cy}^\sigma_{\sbt}$ over $\ct$ of the relative Frobenius endomorphism $f_{\overline{Y}_{\sbt}} \colon \overline{Y}_{\sbt}  \rightarrow  \overline{Y}_{\sbt}^\sigma$.
	Now we have  a commutative diagram
		\[\xymatrix{
		 \mathbb{R}\Gamma_\rig(\ovy/\so_F^0) \ar[r]^\sim \ar[d]^{\mathbb{L}^\sigma}
		&\widehat{\mathcal{K}}_\rig(\overline{Y}_{\sbt}) \ar[d]^{\widehat{L}^\sigma}
		&\widetilde{ \mathcal{K}}_\rig(\overline{Y}_{\sbt},\ovcy_{\sbt},i_{\sbt}) \ar[d]^{\widetilde{L}^\sigma} \ar[r]^\sim \ar[l]_\sim
		&\mathcal{K}_\rig(\overline{Y}_{\sbt},\ovcy_{\sbt},i_{\sbt}) \ar[d]^{L^\sigma}\\
		 \mathbb{R}\Gamma_\rig(\ovy^\sigma/\so_F^0) \ar[r]^\sim \ar[d]^{f^\ast_{\ovy}}
		&\widehat{\mathcal{K}}_\rig(\overline{Y}^\sigma_{\sbt}) \ar[d]^{f^\ast_{\overline{Y}_{\sbt}}} 
		&\widetilde{\mathcal{K}}_\rig (\overline{Y}^\sigma_{\sbt},\ovcy^\sigma_{\sbt},i_{\sbt}^\sigma) \ar[d]^{(f_{\overline{Y}_{\sbt}} ,F_{\ovcy_{\sbt}})^\ast}  \ar[r]^\sim \ar[l]_\sim
		& \mathcal{K}_\rig (\overline{Y}^\sigma_{\sbt},\ovcy^\sigma_{\sbt},i_{\sbt}^\sigma) \ar[d]^{(f_{\overline{Y}_{\sbt}},F_{\ovcy_{\sbt}})^\ast}\\
		 \mathbb{R}\Gamma_\rig(\ovy/\so_F^0) \ar[r]^\sim
		&\widehat{ \mathcal{K}}_\rig(\overline{Y}_{\sbt})
		&\widetilde{ \mathcal{K}}_\rig(\overline{Y}_{\sbt},\ovcy_{\sbt},i_{\sbt}) \ar[r]^\sim \ar[l]_\sim
		& \mathcal{K}_\rig(\overline{Y}_{\sbt},\ovcy_{\sbt},i_{\sbt}),
		}\]
	and accordingly the Frobenius $\varphi$ on  $ \mathbb{R}\Gamma_\rig(\ovy/\so_F^0)$ is given by 
		$$
		\varphi: = f^\ast_{\ovy}\circ \mathbb{L}^\sigma  = \rho_\rig^{-1}\circ (f_{\overline{Y}_{\sbt}},F_{\ovcy_{\sbt}})^\ast \circ L^\sigma\circ\rho_\rig.
		$$
	Moreover, the morphism  $(f_{\overline{Y}_{\sbt}},F_{\ovcy_{\sbt}})^\ast\circ L^\sigma$ on $ \mathcal{K}_\rig(\overline{Y}_{\sbt},\ovcy_{\sbt},i_{\sbt})$ is induced by the action of $\phi_{\sbt}$ on $\overline{\cy}_{\sbt}$, and hence compatible with the Frobenius $\varphi_{\hk}$ on $ \mathcal{K}_{\hk}((Y_{\sbt},\overline{Y}_{\sbt}),(\cz_{\sbt},\ovcz_{\sbt},\phi_{\sbt}),i_{\sbt})$, so that  we have
	$$ 
	L_{\hk}\circ\varphi_{\hk} = (f_{\overline{Y}_{\sbt}},F_{\ovcy_{\sbt}})^\ast \circ L^\sigma\circ L_{\hk}  \colon
	 \mathcal{K}_{\hk}((Y_{\sbt},\overline{Y}_{\sbt},(\cz_{\sbt},\ovcz_{\sbt},\phi_{\sbt}),i_{\sbt})) \rightarrow  \mathcal{K}_\rig(\overline{Y}_{\sbt},\ovcy_{\sbt},i_{\sbt}).
	$$
	This shows the desired compatibility.
	\end{proof}
	
\section{Syntomic cohomology}\label{section: syntomic cohomology}

In this section we recall the definition of crystalline syntomic cohomology from \cite{NekovarNiziol2016} and give a new definition of rigid syntomic cohomology   in the case of a strictly semistable log scheme with boundary. 
Both of these cohomology theories can be understood in the context of $p$-adic Hodge complexes\cite{Bannai2002,ChiarellottoCiccioniMazzari2013,DegliseNiziol2018}.

\subsection{Review of crystalline syntomic cohomology}\label{subsec: syntomic for K-varieties}

We start with a brief review of crysalline Hyodo--Kato theory and refer to \cite{HyodoKato1994} and  \cite[\S3A]{NekovarNiziol2016} for more details.

Let $\so_F\langle t_l\rangle$ be the divided power polynomial algebra generated by elements $t_l$ for  $l\in(\mathfrak{m}_K/\mathfrak{m}_K^2)\setminus 0$ with the relation $t_{al} = [\overline{a}]t_l$ for $a\in \so_K^\ast$. Let $\mathcal{T}_{PD}$ be the $p$-adic completion of the subalgebra of $\so_F\langle t_l\rangle$ generated by $t_l$ and $\frac{t_l^{ie}}{i!}$. If we fix $l$ and set $t = t_l$ it can be seen as an $\so_F$-subalgebra of $F\llbracket t\rrbracket$. As before, we extend the Frobenius by setting $\sigma(t_l) = t_l^p$. By abuse of notation we denote by $\mathcal{T}_{PD}$ also  the  scheme $\Spec \mathcal{T}_{PD}$ with the log structure generated by the $t_l$'s. There are exact closed embeddings
	$$
	\so_F^0 \xrightarrow{i_0} \mathcal{T}_{PD}  \xleftarrow{i_\pi} \so_K^\pi
	$$
	via $t_l \mapsto  0$ and $t_l \mapsto  [\overline{\frac{l}{\pi}}]\pi$,
For a fine proper log smooth log scheme of Cartier type $X$ over $\so_K^\pi$ we consider the log crystalline complexes
	\begin{eqnarray*}
	 \mathbb{R}\Gamma_{\cris}(X/\so_K^\pi) &: =  & \holim  \mathbb{R}\Gamma_{\cris}(X_n/\so_{K,n}^\pi),\\
	 \mathbb{R}\Gamma_{\cris}(X/\mathcal{T}_{PD}) &: = & \holim  \mathbb{R}\Gamma_{\cris}(X_n/\mathcal{T}_{PD,n}),\\	
	 \mathbb{R}\Gamma_{\hk}^\cris(X) &: =  &  \mathbb{R}\Gamma_{\cris}(X_0/\so_F^0): =  \holim  \mathbb{R}\Gamma_{\cris}(X_0/\so_{F,n}^0).
	\end{eqnarray*}
The latter one is often called the Hyodo--Kato complex. 

The Frobenius action $\varphi$ on $ \mathbb{R}\Gamma_{\cris}(X/\mathcal{T}_{PD})$ and $ \mathbb{R}\Gamma_{\hk}^\cris(X)$ is  given by the relative Frobenius and base change by $\sigma$, similarly to the rigid case \eqref{eq: base change Frobenius}, that is 
	\begin{eqnarray*}
	&\varphi\colon  \mathbb{R}\Gamma_{\cris}(X/\mathcal{T}_{PD})  \xrightarrow{\Theta_\sigma}  \mathbb{R}\Gamma_{\cris}(X^\sigma/\mathcal{T}_{PD})   \xrightarrow{\chi_X^\ast}   \mathbb{R}\Gamma_{\cris}(X/\mathcal{T}_{PD})\\
	&\varphi\colon  \mathbb{R}\Gamma_{\hk}^\cris(X)   \xrightarrow{\Theta_\sigma}   \mathbb{R}\Gamma_{\hk}^\cris(X^\sigma)   \xrightarrow{\chi_{X_0}^\ast}     \mathbb{R}\Gamma_{\hk}^\cris(X)
	\end{eqnarray*}
where $\Theta_\sigma$ is the base change by $\sigma$ on $\mathcal{T}_{PD}$ or $\so_F$, and $\chi_X$ and $\chi_{X_0}$ are the relative Frobenii of $X$ over $\mathcal{T}_{PD}$ and of $X_0$ over $\so_F$, respectively. The Frobenius action is invertible on $ \mathbb{R}\Gamma_{\hk}^\cris(X)_\Q$. As usual we denote by $\varphi_r$ the Frobenius divided by $p^r$.

The monodromy operator on $ \mathbb{R}\Gamma_{\hk}^\cris(X)$ can be obtained as the boundary morphism of a short exact sequence. 
Let $\so_F\langle t\rangle$ be the  divided power polynomial algebra in one variable $t$ over $\so_F$.
Let $\ct^\cris$ be $\Spf\so_F\langle t\rangle$ endowed with the log structure associated to the map $\mathbb{N} \rightarrow \so_F\langle t\rangle,\ 1 \mapsto  t$, and $i_0\colon \so_F^0 \rightarrow  \Spf\so_F\langle t\rangle$ be the exact closed immersion induced by $t \mapsto  0$.
Choose an embedding system for $X_0$ over $\Spec(\so_{F,n}[t], 1 \mapsto  t)$ as explained in \cite[(3.6)]{HyodoKato1994}.
The crystalline complexes $C_{X_0/\so_{F,n}^\varnothing}$ and $C_{X_0/\so_{F,n}\langle t\rangle}$ of $X_0$ over $\so_{F,n}^\varnothing$ and $\ct^\cris_n$ on the \'etale site $X_{0,\eet}$  associated to the chosen embedding system are related by the short exact sequence
	$$
	0 \rightarrow  C_{X_0/\ct^\cris_n}[-1]  \xrightarrow{\wedge d\log t} C_{X_0/\so_{F,n}^\varnothing}  \xrightarrow{\can} C_{X_0/\ct^\cris_n}  \rightarrow  0.
	$$
By taking the pull-back along $i_0$ we obtain  
	\begin{equation}\label{Equ:SES2}
	0 \rightarrow  C_{X_0/\so_{F,n}^0}[-1]  \xrightarrow{\wedge d\log t} i_0^\ast C_{X_0/\so_{F,n}^\varnothing}  \xrightarrow{\can} C_{X_0/\so_{F,n}^0}  \rightarrow  0
	\end{equation}
on $X_{0,\eet}$.
Here $C_{X_0/\so_{F,n}^0}$ computes the crystalline Hyodo--Kato cohomology. 
The monodromy $N$ is the connecting homomorphism of the induced long exact sequence.
One can now take the homotopy limit  and invert $p$ to obtain the desired map
	$$
	N\colon  \mathbb{R}\Gamma_{\hk}^\cris(X)_{\Q}  \rightarrow   \mathbb{R}\Gamma_{\hk}^\cris(X)_{\Q}.
	$$
The homotopy limit of $\{C_{X_0/\so_{F,n}^0}\}_n$ is commonly denoted by $W\omega^{\sbt}_{X_0}$ and called the logarithmic de~Rham--Witt complex.
Similarly, we use the notation
 	\begin{equation}\label{equ: tilde de Rham-Witt}
 	W\widetilde{\omega}^{\sbt}_{X_0}: =  \holim i_0^\ast C_{X_0/\so_{F,n}^\varnothing}.
 	\end{equation}
 On the rational complex $ \mathbb{R}\Gamma_{\hk}^\cris(X)_\Q$  one replaces the usual monodromy operator by the normalized one, that is $e^{-1}N$ where $e$ is the absolute ramification index of $K$, to make it compatible with base change. 

The morphisms of log schemes $i_0$ and $i_\pi$ induce morphisms on cohomology
	\begin{equation}\label{eq:crisHK}
	 \mathbb{R}\Gamma_{\hk}^\cris(X) \xleftarrow{i_0^\ast}  \mathbb{R}\Gamma_{\cris}(X/{\mathcal{T}_{PD}})  \xrightarrow{i_\pi^\ast}  \mathbb{R}\Gamma_{\cris}(X/\so_K^\pi).
	\end{equation}
The map $ \mathbb{R}\Gamma_{\cris}(X/{\mathcal{T}_{PD}}) \rightarrow   \mathbb{R}\Gamma_{\hk}^\cris(X)$ from the above diagram has in the derived category a unique functorial $F$-linear section $s_\pi\colon \mathbb{R}\Gamma_{\hk}^\cris(X)_{\Q} \rightarrow   \mathbb{R}\Gamma_{\cris}(X/{\mathcal{T}_{PD}})_{\Q}$ which commutes with the Frobenius (c.f. the sentence after \cite[Lem.~4.4.11]{Tsuji1999}). We set
	$$
	\iota_\pi^\cris: = i_\pi^\ast\circ s_\pi\colon  \mathbb{R}\Gamma_{\hk}^\cris(X)_{\Q} \rightarrow   \mathbb{R}\Gamma_{\cris}( X/\so_K^\pi)_{\Q}.
	$$
It induces a $K$-linear functorial quasi-isomorphism $\iota_\pi\colon   \mathbb{R}\Gamma_{\hk}^\cris(X)\otimes_{W(k)}K  \rightarrow   \mathbb{R}\Gamma_{\cris}( X/\so_K^\pi)_{\Q}$.

Moreover, there exists a canonical quasi-isomorphism
	$$
	\gamma\colon  \mathbb{R}\Gamma_{\dr}(X_K)  \xrightarrow{\sim}  \mathbb{R}\Gamma_{\cris}(X/\so_K^\pi)_\Q,
	$$
where we endow $X_K$ with the pull-back log structure of $X$ and the left hand side is the log de~Rham cohomology of $X_K$ with the Hodge filtration given by the stupid filtration of the log de~Rham complex.
The composition in the derived category 
	$$
	\iota_{\dr}^\pi: = \gamma^{-1}\circ i_\pi^\ast\circ s_\pi\colon  \mathbb{R}\Gamma_{\hk}^\cris(X)_{\Q}  \rightarrow   \mathbb{R}\Gamma_{\dr}(X_K)
	$$
is called the Hyodo--Kato morphism.

\begin{definition}\label{def: crystalline syntomic}
Let $X$ be a fine proper log smooth log scheme of Cartier type over $\so_K^\pi$. For $r\geqslant 0$ and a choice of uniformizer $\pi$, we define the crystalline syntomic cohomology as the homotopy limit
	\begin{equation}\label{equ: syn cris}
	 \mathbb{R}\Gamma_{\syn}^{\cris}(X,r,\pi): = 
	\left[\begin{aligned}\xymatrix{ && \fil^r  \mathbb{R}\Gamma_{\dr}(X_K)  \ar[d]_{(0,\gamma)}\\
	 \mathbb{R}\Gamma_{\hk}^\cris(X)_{\Q}  \ar[rr]^{(1-\varphi_r,\iota_\pi^\cris) \qquad \quad}  \ar[d]_N &&  \mathbb{R}\Gamma_{\hk}^\cris(X)_\Q \oplus  \mathbb{R}\Gamma_{\cris}(X/\so_K^\pi)  \ar[d]_{(N,0)}\\
	 \mathbb{R}\Gamma_{\hk}^\cris(X)_\Q  \ar[rr]^{1-\varphi_{r-1}} &&   \mathbb{R}\Gamma_{\hk}^\cris(X)_\Q}
	\end{aligned}\right].
	 \end{equation}
\end{definition}
Note that we used the normalized monodromy operator here.
Since $\gamma$ is invertible in the derived category, there is in fact a quasi-isomorphism
	$$
	 \mathbb{R}\Gamma_{\syn}^{\cris}(X,r,\pi)  \cong
	\left[\begin{aligned}\xymatrix{ 
	 \mathbb{R}\Gamma_{\hk}^\cris(X)_\Q  \ar[rr]^{(1-\varphi_r,\iota_{\dr}^\pi) \qquad  \qquad }  \ar[d]_N &&  \mathbb{R}\Gamma_{\hk}^\cris(X)_\Q \oplus  \mathbb{R}\Gamma_{\dr}(X_K)/\fil^r  \ar[d]_{(N,0)}\\
	 \mathbb{R}\Gamma_{\hk}^\cris(X)_\Q  \ar[rr]^{1-\varphi_{r-1}} &&   \mathbb{R}\Gamma_{\hk}^\cris(X)_\Q}
	\end{aligned}\right].
	$$
	
In  \cite[\S3C]{NekovarNiziol2016} Nekov\'a\v{r} and Nizio\l{} use $\h$-sheafification to extend syntomic cohomology to $K$-varieties. 
For a $K$-variety $Z$ and $r\in\Z$ we denote it by $\mathbb{R}\Gamma_{\syn}^{\NN}(Z,r)$. 
We may relate $\mathbb{R}\Gamma_{\syn}^{\NN}(Z,r)$ and (\ref{equ: syn cris}) in the case that there is an open embedding $Z\hookrightarrow X$ of a $K$-variety into a regular proper flat $\so_K$-scheme such that $X\setminus Z$ is a normal crossing divisor and $X_0$ is reduced. 
We endow $X$ with the log structure associated to the divisor $X\setminus Z$.

\begin{lemma}\label{lem: NN comparison}
For any $r\geqslant 0$ there is a canonical isomorphism in $\mathscr{D}^+(\mathsf{Mod}_{\Q_p})$
	\begin{equation}\label{Equ:ArithmeticPHC}
	  \mathbb{R}\Gamma_{\syn}^{\NN}(Z,r) \cong  \mathbb{R}\Gamma^{\cris}_{\syn}(X,r,\pi)
	\end{equation}
depending on the choice of a uniformizer $\pi$.
\end{lemma}

\begin{proof}
This is the combination of \cite[Prop.~3.18]{NekovarNiziol2016} and \cite[Prop.~3.8]{NekovarNiziol2016}.
\end{proof}

\begin{remark}
The rational log syntomic complexes $\mathbb{R}\Gamma_{\syn}^{\cris}(X,r,\pi)$ are graded commutative dg algebras \cite[\S 3B]{NekovarNiziol2016}. 
 A definition of the syntomic product structure can be found in \cite[\S 2.2]{Tsuji1999}. 
 The above comparison morphism  is compatible with the product structure (c.f. Proof of \cite[Prop. 5.2]{NekovarNiziol2016}). 
 For an explicit description of the cup product see also \cite[\S 2.4]{Besser2016}. 
 \end{remark}

\subsection{Rigid syntomic cohomology for strictly  semistable schemes}\label{subsec: syntomic def}

We give a definition of rigid syntomic cohomology for strictly semistable log schemes with boundary over $\so_K^\pi$ as a homotopy limit analogous  to \cite{NekovarNiziol2016}.  

By abuse of notation we  denote by $\so_K^\pi$ also the scheme $\Spec\so_K$ with the log structure associated to the map $\N \rightarrow \so_K,\ 1 \mapsto \pi$, and let $c_{\so_K^\pi}$ be the chart of $\so_K^\pi$ induced by this map.

\begin{definition}\label{def:ss log schemes w/ boundary over so_K}
	A $\so_K^\pi$-log scheme with boundary $(X,\ovx)$ is called {\it strictly semistable} if Zariski locally on $\ovx$ there exists a chart $(\alpha\colon P_{\ovx} \rightarrow \cn_{\ovx},\ \beta\colon\N \rightarrow  P^\gp)$ extendending $c_{\so_K^\pi}$ of the following form:
		\begin{itemize}
		\item $P = \N^m\oplus\N^n$ for some integers $m\geqslant 1$ and $n\geqslant 0$, and $\beta$ is given by the composition of the diagonal map $\N \rightarrow \N^m$ and the canonical injection $\N^m \rightarrow \Z^m\oplus\Z^n$.
		In particular the morphism of $X$ extends to a morphism $\ovx \rightarrow \so_K^\pi$ with a chart $\beta'\colon\N \rightarrow \N^m\oplus\N^n = P$.
		\item The morphism of schemes
			\begin{align*}
			\ovx  \rightarrow & \Spec\so_K\times_{\Spec\so_K[\N]}\Spec\so_K[\N^m\oplus\N^n]\\
			&\qquad  =  \Spec\so_K[t_1,\ldots,t_m,s_1,\ldots,s_n]/(t_1\cdots t_m-\pi)
			\end{align*}
		induced by $\beta'$ is smooth, and makes the diagram 
		\[\xymatrix{
		\ovx \ar[r]&\Spec\so_K[t_1,\ldots,t_m,s_1,\ldots,s_n]/(t_1\cdots t_m-\pi)\\
		X \ar[u] \ar[r]&\Spec\so_K[t_1,\ldots,t_m,s_1^{\pm1},\ldots,s_n^{\pm 1}]/(t_1\cdots t_m-\pi) \ar[u]
		}\]
		Cartesian.
		\end{itemize}
		A strictly semistable $\so_K^\pi$-log scheme with boundary $(X,\overline{X})$ is called proper if the underlying scheme $\overline{X}$ is  proper over $\so_K$. 
		This definition is independent of the choice of $\pi$.
		
		A strictly semistable $\so_K^\pi$-log scheme is a fine log scheme $X$ over $\so_K$ such that $(X,X)$ is a strictly semistable $\so_K^\pi$-log scheme with boundary.
\end{definition}


Let $(X,\ovx)$ be a strictly semistable $\so_K^\pi$-log scheme with boundary.
Instead of crystalline Hyodo--Kato cohomology we use the rigid Hyodo--Kato cohomology $ \mathbb{R}\Gamma^\rig_{\hk}(X_0,\ovx_0)$ as our first building block. 
The second building block is the same as in \cite{NekovarNiziol2016}, namely Deligne's de~Rham complex $ \mathbb{R}\Gamma_\dr^{\D}(X_K)$ \cite{Deligne1974}.  
These are connected to each other by the rigid Hyodo--Kato map.
In \S\ref{subsec: HK-map} we showed that there exists a functorial morphism in $\mathscr{D}^+(\mathsf{Mod}_F)$
	$$
	\iota^\rig_\pi\colon \mathbb{R}\Gamma_{\rig}(\ovx_0/\so_F^0)  \rightarrow   \mathbb{R}\Gamma_{\rig}(\ovx_0/\so_K^\pi).
	$$
Since we have the canonical quasi-isomorphisms $ \mathbb{R}\Gamma_\rig(\ovx_0/\so_K^\pi) \xrightarrow{\sim} \mathbb{R}\Gamma_\rig(X_0/\so_K^\pi)$ of Corollary \ref{cor: qis of Y} and $\mathbb{L}^{\hk}\colon \mathbb{R}\Gamma_{\hk}^\rig(X_0,\ovx_0) \xrightarrow{\sim} \mathbb{R}\Gamma_{\rig}(\ovx_0/\so_F^0)$ of \eqref{eq: LHK}, we obtain a morphism
	\begin{equation}\label{GKHKmap}
	 \mathbb{R}\Gamma_{\hk}^\rig(X_0,\ovx_0)  \rightarrow   \mathbb{R}\Gamma_{\rig}(X_0/\so_K^\pi),
	\end{equation}
which by abuse of notation we  denote again by $\iota^\rig_\pi$.

We now have to link the log rigid cohomology $ \mathbb{R}\Gamma_{\rig}(X_0/\so_K^\pi)$ and the de~Rham cohomology $ \mathbb{R}\Gamma_{\dr}^{\D}(X_K)$.
Let $w\colon X_{K,\an} \rightarrow  X_{K,\Zar}$ be the canonical morphism from the analytic to the Zariski site of $X_K$.
Furthermore, we denote by $X_K^{\an}$  the dagger space associated to $X_K$.
Finally, let $\mathcal{X}$ be a weak completion of $X$ and $u\colon\mathcal{X}_\Q \rightarrow  X_K^{\an}$ the canonical inclusion of its generic fibre into $X_K^{\an}$. Note that $(X_0,\mathcal{X},i)$ with $i:X_0 \rightarrow  \mathcal{X}$ forms a $\Xi(\so_K^\pi)$-rigid triple as introduced in \S\ref{subsubsec: usual rigid cohomology}.

Similarly to \cite[Prop.~4.9]{ChiarellottoCiccioniMazzari2013} we consider the diagram of sites
	\begin{equation}\label{diag: CCM Prop. 4.9}
	\xymatrix{ Pt_{\an}(X_K)  \ar[d]^{u_1} \ar[r] & Pt_{\an + \Zar}(X_K) \ar[d]^{u_2} & Pt_{\Zar}(X_K) \ar[d]^{u_3} \ar[l]\\ X_{K,\an}  \ar[r]^w & X_{K,\Zar} & X_{K,\Zar},  \ar[l]_{\id}}
	\end{equation}
where $Pt_{\an}(X_K)$ denotes the discrete site of rigid points on $X_{K,\an}$, $Pt_{\Zar}(X_K)$ the discrete site of Zariski points of $X_K$, and $Pt_{\an + \Zar}(X_K)$ their direct sum in the category of sites.
Note that $u_2$ is induced by $u_3$ on the Zariski points of $X_K$ and by $w\circ u_1$ on the rigid points of $X_{K,\an}$.

By \cite[Lem.~3.2]{ChiarellottoCiccioniMazzari2013}, the left square with the map $\Omega^{\sbt}_{X_K} \rightarrow  w_\ast\Omega^{\sbt}_{X_K^{\an}}$ induces $\Gd_{\an + \Zar}\Omega^{\sbt}_{X_K} \rightarrow  w_\ast\Gd_{\an}\Omega^{\sbt}_{X_K^{\an}}$, and the right square induces $\Gd_{\an + \Zar}\Omega^{\sbt}_{X_K} \rightarrow \Gd_{\Zar}\Omega^{\sbt}_{X_K}$.
Thus we obtain canonical morphisms
	\begin{equation}\label{eq: an + Zar}
	\Gamma(X^{\an}_K,\Gd_{\an}\Omega^{\sbt}_{X_K^{\an}}) \leftarrow \Gamma(X_K,\Gd_{\an + \Zar}\Omega^{\sbt}_{X_K})  \rightarrow  \Gamma(X_K,\Gd_{\Zar}\Omega^{\sbt}_{X_K}).
	\end{equation}
Note that $u_1$, $u_3$, and hence $u_2$ are exact and conservative (see \cite[\S4]{vanderPutSchneider1995} for $u_1$).
Therefore $\Gd_{\an + \Zar}\Omega^{\sbt}_{X_K}$ and $\Gd_{\Zar}\Omega^{\sbt}_{X_K}$ both give a resolution of $\Omega^{\sbt}_{X_K}$, and the second arrow in \eqref{eq: an + Zar} is a quasi-isomorphism. 
Moreover by \cite[\S 1.8 (b)]{GrosseKloenne2004-2} the first arrow in \eqref{eq: an + Zar} is also a quasi-isomorphism.

Furthermore,  $u\colon\cx_\Q \rightarrow  X_K^{\an}$ induces a canonical map
	\begin{equation}\label{eq: pull back}
	\Gamma(X^{\an}_K,\Gd_{\an}\Omega^{\sbt}_{X_K^{\an}})  \rightarrow  \Gamma(\mathcal{X}_\Q,\Gd_{\an}\Omega^{\sbt}_{\mathcal{X}_K})  =   \mathcal{K}_\rig((X_0,\mathcal{X},i).
	\end{equation}

Since we have $  \mathcal{K}_\rig((X_0,\mathcal{X},i)\cong \mathbb{R}\Gamma_\rig(X_0/\so_K^\pi)$ by Theorem~\ref{Thm: cohomology functors} and it is well known that $\Gamma(X_K,\Gd_\an\Omega^{\sbt}_{X_K}) \cong \mathbb{R}\Gamma_{\dr}^{\D}(X_K)$, we can combine  \eqref{eq: an + Zar} and \eqref{eq: pull back} to obtain
a morphism
	\[\mathrm{sp}\colon \mathbb{R}\Gamma_{\dr}^{\D}(X_K)  \rightarrow   \mathbb{R}\Gamma_{\rig}(X_0/\so_K^\pi)\]
in $\mathscr{D}^ + _K$.
Note that the morphism \eqref{eq: pull back} and hence $\mathrm{sp}$ is not a quasi-isomorphism in general.

Consequently we obtain a pair of morphisms in the derived category
	\begin{equation}\label{HK - zigzag}
	 \mathbb{R}\Gamma_{\hk}^{\rig}(X_0,\ovx_0) \xrightarrow{\iota^\rig_\pi }  \mathbb{R}\Gamma_{\rig}(X_0/\so_K^\pi) \xleftarrow{\mathrm{sp}}  \mathbb{R}\Gamma_{\dr}^{\D}(X_K),
	\end{equation}
where both $\iota^\rig_\pi$ and $\mathrm{sp}$ are given by zigzags on the level of complexes as described above. 

For the definition of rigid syntomic cohomology as homotopy limit we use the normalized monodromy, that is we divide the monodromy coming from Definition~\ref{def: Kim - Hain} by the absolute ramification index $e$.

\begin{definition}\label{Def:LogrigidSyn}
	Let $(X,\ovx)$ be a strictly semistable $\so_K^\pi$-log scheme with boundary.
	For $r\geqslant 0$ and a choice of uniformizer $\pi$ the rigid syntomic cohomology of $(X,\ovx)$ is given by
	\begin{equation}\label{eq: LogrigidSyn}
	 \mathbb{R}\Gamma^{\rig}_{\syn}((X,\ovx),r,\pi)
	 : =  \left[\begin{aligned}\xymatrix{ && \fil^r \mathbb{R}\Gamma_{\dr}^{\D}(X_K)  \ar[d]_{(0,\mathrm{sp})}\\
	 \mathbb{R}\Gamma_{\hk}^{\rig}(X_0,\ovx_0)  \ar[rr]^{(1-\varphi_r,\iota_\pi^\rig)  \qquad }  \ar[d]_N &&  \mathbb{R}\Gamma_{\hk}^{\rig}(X_0,\ovx_0) \oplus  \mathbb{R}\Gamma_{\rig}(X_0/\so_K^\pi)  \ar[d]_{(N,0)}\\
	 \mathbb{R}\Gamma_{\hk}^{\rig}(X_0,\ovx_0)  \ar[rr]^{1-\varphi_{r-1}} &&   \mathbb{R}\Gamma_{\hk}^{\rig}(X_0,\ovx_0)}
	\end{aligned}\right],
	 \end{equation}
where  $\varphi_r$ denotes as usual the Frobenius divided by $p^r$. It is functorial in $(X,\ovx)$.

For a strictly semistable $\so_K^\pi$-log scheme $X$, we define the rigid syntomic cohomology of $X$ to be $ \mathbb{R}\Gamma^\rig_\syn(X,r,\pi): =  \mathbb{R}\Gamma^\rig_\syn((X,X),r,\pi)$.
\end{definition}

We will now show that rigid syntomic cohomology also admits a cup product (c.f. \cite{Besser2016}[\S 2.4]). 
This is the case because each of its building blocks has a cup product.
More precisely, let $(X,\overline{X})$ be a strictly semistable log scheme with boundary over $\so_K^\pi$, and let
	\[ \mathbb{R}\Gamma^\rig_{\hk}(X_0,\overline{X}_0) \xrightarrow{a} \mathbb{R}\Gamma_\rig(X_0/\so_K^\pi)\xleftarrow[]{b} \mathbb{R}\Gamma_{\dr}^{\D}(X_K)\]
be the associated $p$-adic Hodge complex.
The cup products on $ \mathbb{R}\Gamma^\rig_{\hk}(X_0,\overline{X}_0)$, $ \mathbb{R}\Gamma_\rig(X_0/\so_K^\pi)$, and $ \mathbb{R}\Gamma_{\dr}^{\D}(X_K)$ are induced by the wedge product of differential forms.
We want to define a cup product
	\[H^{\rig,i}_\syn((X,\overline{X}),r,\pi)\times H^{\rig,j}_\syn((X,\overline{X}),s,\pi) \rightarrow  H^{\rig,i + j}_\syn((X,\overline{X}),r + s,\pi)\]
	 on rigid syntomic cohomology.
Thus let $\eta\in H^{\rig,i}_\syn((X,\overline{X}),r,\pi)$ and $\eta'\in  H^{\rig,j}_\syn((X,\overline{X}),s,\pi)$.
The class of $\eta$ is represented by a sextuple $(u,v,w,x,y,z)$ with
	\begin{align*}
	&u\in \mathbb{R}\Gamma^{\rig,i}_{\hk}(X_0,\overline{X}_0),&&v\in\mathrm{Fil}^r \mathbb{R}\Gamma^{D,i}_\dr(X_K),\\
	&w,x\in \mathbb{R}\Gamma^{\rig,i-1}_{\hk}(X_0,\overline{X}_0),&&y\in \mathbb{R}\Gamma^{i-1}_\rig(\overline{X}_0/\so_K^\pi),\\
	&z\in \mathbb{R}\Gamma^{\rig,i-2}_{\hk}(X_0,\overline{X}_0),&&
	\end{align*}
such that
	\begin{align*}
	&du = 0,&&dv = 0,&&\\
	&dw = (1-\frac{\varphi}{p^r})(u),&&dx = N(u),&&dy = a(u)-b(v),\\
	&dz = N(w)-(1-\frac{\varphi}{p^{r-1}})(x).&& &&
	\end{align*}
Similarly, the class of $\eta'$ is represented by a sextuple $(u',v',w',x',y',z')$.
We define another  sextuple $(u'',v'',w'',x'',y'',z'')$ by
	\begin{eqnarray*}
	u''&: = &u\cup u',\\
	v''&: = &v\cup v',\\
	w''&: = &w\cup\frac{\varphi(u')}{p^s} + (-1)^i(u\cup w'),\\
	x''&: = &x\cup u' + (-1)^i(u\cup x'),\\
	y''&: = &y\cup a(u') + (-1)^i(b(v)\cup y'),\\
	z''&: = &z\cup\frac{\varphi(u')}{p^s}-(-1)^i(w\cup\frac{\varphi(x')}{p^{s-1}}) + (-1)^i(x\cup w')-u\cup z'.
	\end{eqnarray*}
The class of $H^{\rig,i + j}_\syn((X,\overline{X}),r + s,\pi)$ it represents, is the cup product of $\eta$ and $\eta'$ and denoted by $\eta\cup\eta'$.

To conclude this section, we note that the rigid syntomic cohomology for $X$ and $(X,\ovx)$ are isomorphic.

\begin{proposition}\label{prop: rig syn comp}
	Let $(X,\ovx)$ be a strictly semistable log scheme with boundary over $\so_K^\pi$.
	Then the canonical morphism in $\mathscr{D}^+(\mathsf{Mod}_{\Q_p})$
	\[ \mathbb{R}\Gamma^\rig_\syn((X,\ovx),r,\pi) \rightarrow  \mathbb{R}\Gamma^\rig_\syn(X,r,\pi)\]
	is an isomorphism.
\end{proposition}

\begin{proof}
	It is enough to show that $ \mathbb{R}\Gamma^\rig_{\hk}(X_0,\ovx_0) \rightarrow  \mathbb{R}\Gamma^\rig_{\hk}(X_0)$ is an isomorphism.
	This follows from Corollary~\ref{cor: qis of Y}, since we have $ \mathbb{R}\Gamma^\rig_{\hk}(X_0,\ovx_0) \cong \mathbb{R}\Gamma_\rig(\ovx_0/\so_F^0)$ and $ \mathbb{R}\Gamma^\rig_{\hk}(X_0,X_0) \cong \mathbb{R}\Gamma_\rig(X_0/\so_F^0)$.
\end{proof}

\section{Comparison in the compactifiable case}\label{section: comparison}

In this section we compare the syntomic cohomology theories introduced in the previous section in the case of a strictly semistable log scheme which has a normal crossings compactification.
Our main result is the following.

\begin{theorem}\label{thm: syntomic}
Let $(X,\overline{X})$ be a proper strictly semistable $\so_K^\pi$-log scheme with boundary.
Then for $r\geqslant 0$ there exists a canonical quasi-isomorphism
	$$
	 \mathbb{R}\Gamma^\rig_\syn((X,\ovx),r,\pi) \cong \mathbb{R}\Gamma_{\syn}^{\cris}(\ovx,r,\pi),
	$$ 
which is compatible with cup products.
\end{theorem}

Combined with Lemma~\ref{lem: NN comparison} and Proposition~\ref{prop: rig syn comp}, this implies the following.

\begin{corollary}
	Let $(X,\overline{X})$ be a proper strictly semistable $\so_K^\pi$-log scheme with boundary.
Then for $r\geqslant 0$ there exists a canonical quasi-isomorphism
	$$
	 \mathbb{R}\Gamma^\rig_\syn(X,r,\pi) \cong \mathbb{R}\Gamma_{\syn}^{\NN}(X_K,r),
	$$
which is compatible with cup products. 
\end{corollary}

Recall that the complexes in the theorem are given by the homotopy limits
	$$
 	 \mathbb{R}\Gamma^{\rig}_{\syn}((X,\ovx),r,\pi)
	  =  \left[\begin{aligned}\xymatrix{ && \fil^r \mathbb{R}\Gamma_{\dr}^{\D}(X_K)  \ar[d]_{(0,\mathrm{sp})}\\
	  \mathbb{R}\Gamma^{\rig}_{\hk}(X_0,\ovx_0) \ar[rr]^{(1-\varphi_r,\iota_\pi^\rig) \qquad }  \ar[d]_N &&  \mathbb{R}\Gamma_{\hk}^{\rig}(X_0,\ovx_0)\oplus  \mathbb{R}\Gamma_{\rig}(X_0 /\so_K^\pi)  \ar[d]_{(N,0)}\\
	 \mathbb{R}\Gamma_{\hk}^{\rig}(X_0,\ovx_0)  \ar[rr]^{1-\varphi_{r-1}} &&   \mathbb{R}\Gamma_{\hk}^{\rig}(X_0,\ovx_0)}
	\end{aligned}\right],
	 $$
and
	$$
	 \mathbb{R}\Gamma_{\syn}^{\cris}(\ovx,r,\pi) = 
	\left[\begin{aligned}\xymatrix{ && \fil^r  \mathbb{R}\Gamma_{\dr}(\ovx_K)  \ar[d]_{(0,\gamma)}\\
	 \mathbb{R}\Gamma_{\hk}^\cris(\ovx)  \ar[rr]^{(1-\varphi_r,\iota_\pi^\cris) \qquad }  \ar[d]_N &&  \mathbb{R}\Gamma_{\hk}^\cris(\ovx) \oplus  \mathbb{R}\Gamma_{\cris}(\ovx/\so_K^\pi)  \ar[d]_{(N,0)}\\
	 \mathbb{R}\Gamma_{\hk}^\cris(\ovx)  \ar[rr]^{1-\varphi_{r-1}} &&   \mathbb{R}\Gamma_{\hk}^\cris(\ovx)}
	\end{aligned}\right].
	$$

The first step to link them is to compare the cohomology theories involved.
Note that in this case Deligne's de~Rham cohomology is computed by the compactification $X_K\hookrightarrow\ovx_K$, namely we have a natural filtered quasi-isomorphism $ \mathbb{R}\Gamma_\dr^{\D}(X_K) \cong \mathbb{R}\Gamma_\dr(\ovx_K)$.
Thus it remains to obtain a canonical quasi-isomorphism between $ \mathbb{R}\Gamma^{\rig}_{\hk}(X_0,\ovx_0)$ and $ \mathbb{R}\Gamma_{\hk}^\cris(\ovx)$ which is  compatible with Frobenius $\varphi$ and the (normalized) monodromy $N$.
The next step is to show that the crystalline and rigid Hyodo--Kato morphisms are compatible.
In particular we need to show the compatibility of $\iota_\pi^\rig$ and $\iota_\pi^\cris$. 
This is not obvious from the construction.
However, it suffices to see the compatibility on Frobenius eigenspaces where it will follow from formal arguments.

\subsection{Logarithmic rigid and crystalline cohomology}\label{Subsec:comparison rig - cris}

In this section we compare for a strictly semistable $k^0$-log scheme with boundary log rigid cohomology with log crystalline cohomology passing through Shiho's log convergent cohomology.

All constructions of rigid complexes in Section \ref{sec: Log rigid complexes} can also be carried out for formal schemes and rigid spaces instead of weak formal schemes and dagger spaces.
We call the resulting complexes convergent Hyodo--Kato complex and log convergent complexes, and denote them by replacing the index $\rig$ by $\conv$.
This is a priori different from the log convergent cohomology in the sense of \cite{Shiho2002}, but by \cite[Cor.~2.3.9]{Shiho2002} we may identify them.

In particular we obtain rigid cohomological tuples $\Xi_\conv$, $\Xi^\conv_\hk$, and functors
\begin{align*}
	&\lsk \rightarrow \mathscr{D}^+(\mathsf{Mod}_{\so_\Q}),&&Y \mapsto  \mathbb{R}\Gamma_{\conv}(Y/\cs),\\
	&\ssk \rightarrow \mathscr{D}^+(\mathsf{Mod}_F(\varphi,N)),&&(Y,\ovy) \mapsto  \mathbb{R}\Gamma^{\conv}_{\hk}(Y,\ovy)
\end{align*}
for $\cs = \so_F^0,\so_K^\pi,\widehat{\ct}$ and $\so_Q = F,K,\widehat{K[t]}$.

For a strictly semistable weak formal $\ct$-log scheme with boundary $(\cz,\ovcz) \in \sst$, we denote by $\ovcz^\flat$ the weak formal log scheme whose underlying weak formal scheme is $\ovcz$ and the log structure is associated to the simple normal crossing divisor $\ovcy: = \ovcz\times_\ct\so_F^0$.
(Note that the log structure of $\ovcz$ is associated to $\ovcy\cup\cd$.)
For any weak formal subscheme $\mathcal{V}$ of $\ovcz$ we denote by $\mathcal{V}^\flat$ the weak formal log scheme whose underlying weak formal scheme is $\mathcal{V}$ and whose log structure is the pull-back of the log structure of $\ovcz^\flat$.
Consider the log de~Rham complex  $\omega^{\sbt}_{\ovcy^\flat}: = \omega^{\sbt}_{\ovcy^\flat/\so_F^0}$ of $\ovcy^\flat$ over $\so_F^0$.
We define the motivic weight filtration $P_{\sbt}$ on $\omega^{\sbt}_{\ovcy}$ by counting the log differentials along the divisor $\ovcy\setminus\cy=\ovcy\cap\cd$,
	\[P_k\omega^i_{\ovcy}: = \im(\omega^k_{\ovcy}\otimes\omega^{i-k}_{\ovcy^\flat} \rightarrow \omega^i_{\ovcy}).\] 
Let $\cd$ be the divisor $\ovcz\setminus \cz$. 
For a subset $I\subset \Upsilon_{\cd}$, we let $\cd_I: = \bigcap_{\beta\in I}\cd_\beta$, where $\cd_\beta$ is the irreducible component of $\cd$ which corresponds to $\beta$.
By taking the residue along the divisor $\ovcy\cap\cd$ in $\ovcy$, one computes the graded pieces  as
\[\Gr^P_k\omega^{\sbt}_{\ovcy}: = \bigoplus_{\substack{I\subset\Upsilon_\cd\\ \lvert I\rvert = k}}\omega^{\sbt}_{(\ovcy\cap\cd_I)^\flat}[-k].\]
Thus for an object $((Y,\ovy),(\cz,\ovcz,\phi),i)\in\mathsf{RT}_{\Xi_\hk^\rig}$ this filtration induces a motivic weight spectral sequence of log rigid cohomology groups over $\so_F^0$
	\begin{equation}
	E_1^{p,q} = \bigoplus_{\substack{I\subset\Upsilon_\cd\\ \lvert I\rvert = -p}}H^{2p + q}\mathcal{K}_\rig((\ovy\cap D_I)^\flat,(\ovcy\cap\cd_I)^\flat,i)\Rightarrow H^{p + q}\mathcal{K}_\rig(\ovy,\ovcy,i).
	\end{equation}
The general case can be reduced to this  by Remark~\ref{rem: gluing lifts} and we obtain
	\begin{equation}\label{eq: motivic weight ss}
	E_1^{p,q} = \bigoplus_{\substack{I\subset\Upsilon_\cd\\ \lvert I\rvert = -p}}H^{2p + q}_\rig((\ovy\cap D_I)^\flat/\so_F^0)\Rightarrow H^{p + q}_\rig(\ovy/\so_F^0)
	\end{equation}
for an object $(Y,\ovy)$ in $\ssk$.
Note that $(\ovy\cap D_I)^\flat$ is a strictly semistable $k^0$-log scheme.
Similar spectral sequences exist for log convergent cohomology.

\begin{lemma}\label{lem:rig - conv}
	Let $(X,\ovx)$ be a proper strictly semistable $\so_K^\pi$-log scheme with boundary.
	Then the completion of weak formal schemes and dagger spaces induces quasi-isomorphisms 
	\begin{eqnarray*}
	& \mathbb{R}\Gamma_\rig(\ovx_0/\so_F^0) \xrightarrow{\sim} \mathbb{R}\Gamma_\conv(\ovx_0/\so_F^0),\\ 
	 & \mathbb{R}\Gamma_\rig(\ovx_0/\so_K^\pi) \xrightarrow{\sim} \mathbb{R}\Gamma_\conv(\ovx_0/\so_K^\pi)
	 \end{eqnarray*}
\end{lemma}

\begin{proof}
	For $\so_F^0$, we  reduce to the case $\cd = \emptyset$ by the spectral sequence \eqref{eq: motivic weight ss}.
	But this case is covered by Gro\ss e-Kl\"{o}nne \cite[\S 3.8]{GrosseKloenne2005}.
	Indeed, one constructs another spectral sequence coming from the singularity of $\ovy$, and then  reduces to the classical, i.e.  non-logarithmic, result  due to Berthelot.
 
	For $\so_K^\pi$, the statement follows from \cite[Thm.~3.2]{GrosseKloenne2000}, since $\ovx_0$ has a global lifting $\overline{\mathcal{X}}$, in form of the weak completion of $\ovx$, which is partially proper.
\end{proof}

Now the comparison between log rigid and log crystalline cohomology follows immediately from the comparison of log convergent and log crystalline cohomology in \cite[Thm.~3.1.1]{Shiho2002} and the invariance of convergent cohomology under nilpotent thickenings (compare \cite{Shiho2008}).

\begin{corollary}
	Let $(X,\ovx)$ be a proper strictly semistable $\so_K^\pi$-log scheme with boundary.
	Then there are canonical quasi-isomorphisms
	\begin{align*}
	& \mathbb{R}\Gamma_\rig(\ovx_0/\so_F^0) \xrightarrow{\sim} \mathbb{R}\Gamma_\cris(\ovx_0/\so_F^0)_\Q,\\ 
	& \mathbb{R}\Gamma_\rig(\ovx_0/\so_K^\pi) \xrightarrow{\sim} \mathbb{R}\Gamma_\cris(\ovx/\so_K^\pi)_\Q,\\
	& \mathbb{R}\Gamma_{\hk}^\rig(X_0,\ovx_0) \xrightarrow{\sim}  \mathbb{R}\Gamma_{\hk}^\cris(\ovx)_\Q.
	\end{align*}
\end{corollary}

\subsection{Compatibility of structures}\label{subsec: comparison structures}

First we turn our attention to the structure given by Frobenius and monodromy.

\begin{proposition}
	Let $(X,\overline{X})$ be a proper strictly semistable $\so_K^\pi$-log scheme with boundary.
	Then the action of Frobenius and the (normalized) monodromy operators on $ \mathbb{R}\Gamma^{\rig}_{\hk}(X_0, \ovx_0)$ and $ \mathbb{R}\Gamma_{\hk}^\cris(\overline{X})_{\Q}$ are compatible with each other.
\end{proposition}	

\begin{proof}
By construction the Frobenius and monodromy operators on $ \mathbb{R}\Gamma^\rig_{\hk}(X_0,\ovx_0)$ and $ \mathbb{R}\Gamma^\conv_{\hk}(X_0,\ovx_0)$ are compatible with each other.
To show the compatibility between the Frobenius and monodromy on convergent Hyodo--Kato and on crystalline Hyodo--Kato cohomology, we use the alternative construction from \cite{KimHain2004},
which we used as a guide for Definition~\ref{def: Kim - Hain}.

Let $W\widetilde{\omega}^{\sbt}_{\overline{X}_0}\llbracket u\rrbracket$ be the completion with respect to the ideal $(u^{[1]})$ of the commutative dg-algebra obtained from $W\widetilde{\omega}^{\sbt}_{\overline{X}_0}$ defined in (\ref{equ: tilde de Rham-Witt}) by adjoining the divided powers $u^{[i]}$ of a variable $u$ in degree zero, i.e. they satisfy the relations $du^{[i]} = d\log t\cdot u^{[i-1]}$ and $u^{[0]} = 1$.
The monodromy operator on $W\widetilde{\omega}^{\sbt}_{\overline{X}_0}\llbracket u\rrbracket$ is defined to be the $W\widetilde{\omega}^{\sbt}_{\overline{X}_0}$-linear morphism that maps $u^{[i]}$ to $u^{[i-1]}$.
By \cite[Lem.~6]{KimHain2004} and \cite[Lem.~7]{KimHain2004}, the natural morphism $W\widetilde{\omega}^{\sbt}_{\overline{X}_0}\llbracket u\rrbracket  \rightarrow  W\omega^{\sbt}_{\overline{X}_0}$ is a quasi-isomorphism, compatible with Frobenius and monodromy (compare Remark~\ref{rem: Kim Hain}). 
Hence $W\widetilde{\omega}^{\sbt}_{\overline{X}_0}\llbracket u\rrbracket$ also computes the crystalline Hyodo--Kato cohomology and its structures.

To compare this with the Frobenius and monodromy on the convergent Hyodo--Kato complex, we note that by Lemma~\ref{lem: HK embedding} and Remark~\ref{rem: gluing lifts} (or rather their convergent analogue) we can choose an embedding system $(\overline{U}_{\sbt},\ovz_{\sbt})$ in the sense of \cite[2.18]{HyodoKato1994} for $\overline{X}_0  \rightarrow  \Spec(\so_F[t],1 \mapsto  t)$ in the Zariski topology, whose  completion gives a simplicial object $((U_{\sbt},\overline{U}_{\sbt}),(\cz_{\sbt},\ovcz_{\sbt},\phi_{\sbt}),i_{\sbt}) \in\mathsf{RT}_{\Xi_\hk^\conv}$.
Hence, without loss of generality we assume that $(X_0,\overline{X}_0)$ is $\Xi_\hk^\conv$-embeddable and choose $((X_0,\overline{X}_0),(\cz,\ovcz,\phi),i) \in\mathsf{RT}_{\Xi_\hk^\conv}$ appropriately.

By construction there is a morphism 
	$$
	 \mathcal{K}_{\Xi^\conv_\hk}((X_0,\ovx_0),(\cz,\ovcz,\phi),i) \rightarrow   R\Gamma(\ovx_0,W\widetilde{\omega}^{\sbt}_{\overline{X}_0}\llbracket u\rrbracket)_{\Q}
	$$
which is compatible with Frobenius and monodromy.
It is in fact a quasi-isomorphism by the usual comparison between log convergent and log crystalline cohomology and this shows the claim.
\end{proof}

Next we study the compatibility of the Hyodo--Kato morphisms.

\begin{proposition}
	Let $(X,\overline{X})$ be a proper strictly semistable $\so_K^\pi$-log scheme with boundary. The Hyodo--Kato morphisms on $ \mathbb{R}\Gamma_{\hk}^\rig(X_0,\ovx_0)$ and $  \mathbb{R}\Gamma_{\hk}^\cris(\ovx)_\Q$ are compatible on the Frobenius eigenspaces relative to $p^r$ for $r\geqslant -1$. 
\end{proposition}	

\begin{proof}
To compare the crystalline and rigid Hyodo--Kato morphisms we once again pass through log convergent cohomology. 
Analogous to \eqref{eq:crisHK} we consider the exact closed immersions
	$$
	\so_F^0 \xrightarrow{i_0} \widehat{\mathcal{T}} \xleftarrow{i_\pi} \so_K^\pi
	$$
given by $t \mapsto  0$ and $t \mapsto  \pi$, which induce the base change morphisms
	\begin{equation}\label{eq:convHK}
	 \mathbb{R}\Gamma_{\conv}(\overline{X}_0/\so_F^0)\xleftarrow{i_0^\ast}  \mathbb{R}\Gamma_{\conv}(\overline{X}_0/\widehat{\mathcal{T}})  \xrightarrow{i_\pi^\ast}  \mathbb{R}\Gamma_{\conv}(\overline{X}_0/\so_K^\pi).
	\end{equation}

The functor between the log convergent and log crystalline site defined in \cite[\S 5.3]{Shiho2000} together with the crystalline and convergent Poincar\'e Lemma induces a commutative diagram
	\begin{equation}\label{eq:diageConvCris}
	\xymatrix{
	 \mathbb{R}\Gamma_{\conv}(\overline{X}_0/\so_F^0)  \ar[d]^\sim &   \mathbb{R}\Gamma_{\conv}(\overline{X}_0/\widehat{\mathcal{T}}) \ar[d]  \ar[r]^{i^\ast_\pi}   \ar[l]_{i_0^\ast} &     \mathbb{R}\Gamma_{\conv}(\overline{X}_0/\so_K^\pi) \ar[d]^{\sim} & \\
	 \mathbb{R}\Gamma_{\hk}^\cris(\ovx)_\Q &  \mathbb{R}\Gamma_{\cris}(\overline{X}/\mathcal{T}_{PD})_{\Q} \ar[r]^{i^\ast_\pi}  \ar[l]_{i_0^\ast}^{\sim_\varphi} & \mathbb{R}\Gamma_{\cris}(\overline{X}/\so_K^\pi)_{\Q}& \mathbb{R}\Gamma_{\dr}(\ovx_K)  \ar[lu]_{\quad\mathrm{sp}}^{\quad\sim}  \ar[l]_{\gamma}^{\sim}
	}
	\end{equation}
where   $ \mathbb{R}\Gamma_{\cris}(\overline{X}/\mathcal{T}_{PD})_{\Q}  \xrightarrow{i_0^\ast}  \mathbb{R}\Gamma_{\hk}^\cris(\overline{X})_{\Q} $ is a quasi-isomorphism on Frobenius eigenspaces, i.e. on the homotopy limits of $1-\varphi_r$, by \cite[Proof of Prop.~3.8]{NekovarNiziol2016}.

In the case of rigid cohomology, the closed immersions of weak formal log schemes,
	$$
	\so_F^0 \xrightarrow{i_0} \mathcal{T} \xleftarrow{i_\pi} \so_K^\pi
	$$
induce canonical morphisms on cohomology
	\begin{equation}\label{eq:rigHK}
	 \mathbb{R}\Gamma_{\rig}(\ovx_0/\so_F^0)\xleftarrow{i_0^\ast}  \mathbb{R}\Gamma_{\rig}(\ovx_0/{\mathcal{T}})  \xrightarrow{i_\pi^\ast}  \mathbb{R}\Gamma_{\rig}(\ovx_0/\so_K^\pi),
	\end{equation}
and similarly for $X_0$. 
By construction they are compatible with the convergent analogue (\ref{eq:convHK}) via the comparison morphisms from the previous section.

However it is unclear how to obtain a section of the left morphism. 
Instead, we use the construction in the diagram before Lemma \ref{lem: qis of M V P} to obtain $\iota_\pi^\rig$.
In what follows we use the same notation as in \S\ref{subsec: HK-map} with $(X_0,\ovx_0)$ instead of $(Y,\ovy)$.
Note that in the diagram before Lemma \ref{lem: qis of M V P} $\overline{\xi}_0$ is induced by $t \mapsto  0$ while $\overline{\xi}_\pi$ is induced by $t \mapsto  \pi$. 

Let $\overline{M}_{\sbt}$ be the simplicial log scheme over $k^0$ associated to $(X_0,\ovx_0)$ as in \S\ref{subsec: HK-map}, and $(\overline{V}_{\sbt},\overline{P}_{\sbt})$ the corresponding $T$-log scheme with boundary.
The augmentation $\overline{M}_{\sbt} \rightarrow  \overline{Y}$ induces quasi-isomorphisms for $\mathcal{S} = \so_F^0, \so_K^\pi,\mathcal{T}$
	$$
	 \mathbb{R}\Gamma_\rig(\ovx_0/\mathcal{S}) \xrightarrow{\sim}  \mathbb{R}\Gamma_\rig(\overline{M}_{\sbt}/\mathcal{S}).
	$$
Together with the canonical morphism $ \mathbb{R}\Gamma_\rig((\overline{V}_{\sbt},\overline{P}_{\sbt})/\mathcal{T})  \rightarrow   \mathbb{R}\Gamma_\rig(\overline{M}_{\sbt}/\mathcal{T})$ they allow us to fit (\ref{eq:rigHK}) and the diagram before Lemma \ref{lem: qis of M V P} into one diagram
	$$
	\xymatrix{ &&  \mathbb{R}\Gamma_\rig((\overline{V}_{\sbt},\overline{P}_{\sbt})/\mathcal{T})  \ar[lld]_{\sim} \ar[rrd]  \ar[d]&& \\
	  \mathbb{R}\Gamma_\rig(\overline{M}_{\sbt}/\so_F^0)  &&  \mathbb{R}\Gamma_\rig(\overline{M}_{\sbt}/\mathcal{T})  \ar[ll] \ar[rr]&&  \mathbb{R}\Gamma_\rig(\overline{M}_{\sbt}/\so_K^\pi) \\
	  \mathbb{R}\Gamma_\rig(\ovx_0/\so_F^0) \ar[u]^{\sim} & &  \mathbb{R}\Gamma_\rig(\ovx_0/\mathcal{T})  \ar[rr] \ar[ll] \ar[u]^{\sim} &&  \mathbb{R}\Gamma_\rig(\ovx_0/\so_K^\pi)  \ar[u]^{\sim}}
	 $$
which obviously commutes since the all maps which point left are induced by $t \mapsto  0$ and all maps which point right are induced by $t \mapsto  \pi$.

To see that $\iota_\pi^\rig$ and $\iota_\pi^\cris$ are compatible, we have to invert the morphism  $i_0^\ast$ in  (\ref{eq:crisHK}).
This is possible on Frobenius eigenspaces. 
To consider the Frobenius on $ \mathbb{R}\Gamma_\rig((\ovv_{\sbt},\ovp_{\sbt})/\ct)$, we need the following lemma.

\begin{lemma}
	The morphism $\mathbb{L}_{(\ct,\ovct)/\ct}\colon \mathbb{R}\Gamma_\rig((\ovv_{\sbt},\ovp_{\sbt})/\ct) \rightarrow  \mathbb{R}\Gamma_\rig((\ovv_{\sbt},\ovp_{\sbt})\times_T(T,\ovt)/(\ct,\ovct))$ in \eqref{eq: Lovct} is a quasi-isomorphism.
\end{lemma}

\begin{proof}
	It suffices to show that $ \mathbb{R}\Gamma_\rig((\ovv_J^I,\ovp_J^I)/\ct) \cong \mathbb{R}\Gamma((\ovv_J^I,\ovp_J^I)\times_T(T,\ovt)/(\ct,\ovct))$ for any $I\subset J\subset\Upsilon_{\ovx_0}$.
	This is just the construction $(\ovv_J^I,\widetilde{P}_J^I)$ in \S\ref{subsec: HK-map} with respect to the structure morphism $\ovx_0 \rightarrow  k^0$.
	More precisely, if we let $(Y',\ovy'): = k^0$ and let $f\colon(Y,\ovy) \rightarrow (Y',\ovy') = k^0$ be the structure morphism, then $(\ovv_J^I,\widetilde{P}_J^I)$ associated to $f$ is just $(\ovv_J^I,\ovp_J^I)\times_T(T,\ovt)$.
	Thus we have
		\[\mathbb{R}\Gamma_\rig((\ovv_J^I,\ovp_J^I)/\ct) \cong \mathbb{R}\Gamma((\ovv_J^I,\ovp_J^I)\times_T(T,\ovt)/\ct)\]
	by Lemma~\ref{lem: qis with tilde boundary}.
	The right hand side is isomorphic to $ \mathbb{R}\Gamma_\rig((\ovv_J^I,\ovp_J^I)\times_T(T,\ovt)/(\ct,\ovct))$ via $\mathbb{L}^\sharp$ in \eqref{eq: Lsharp}.
\end{proof}

Consequently, one can define a Frobenius endomorphism on $ \mathbb{R}\Gamma_\rig((\ovv_{\sbt},\ovp_{\sbt})/\ct)$ as explained at the end of \S\ref{subsubsec: Base change and Frobenius}, which is compatible with the Frobenius on $ \mathbb{R}\Gamma_\rig(\overline{M}_{\sbt}/\ct)$ and hence with the Frobenius on  $ \mathbb{R}\Gamma_{\rig}(\overline{X}_0/\mathcal{T})$ and with the Frobenius on $ \mathbb{R}\Gamma_{\rig}(\overline{X}_0/\so_F^0) $.

Putting everything that we discussed together we obtain a commutative diagram
	{\small
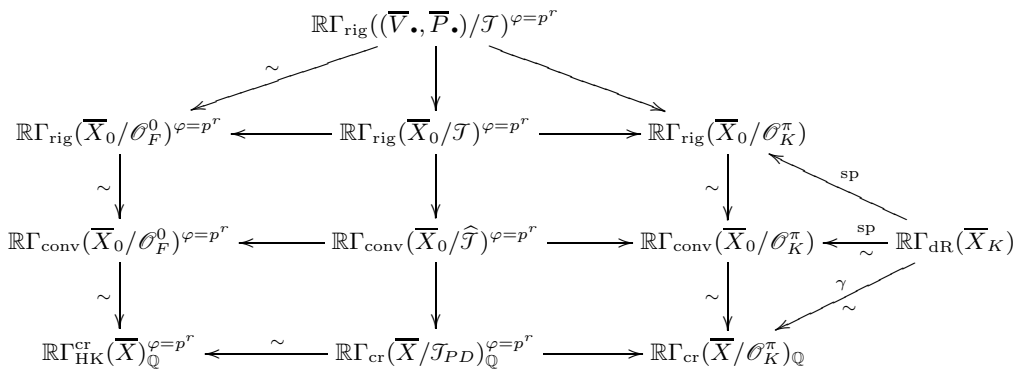
	$$
	\mathclap{
	\xymatrix{	
	&   \mathbb{R}\Gamma_{\rig}((\overline{V}_{\sbt},\overline{P}_{\sbt})/\mathcal{T})^{\varphi = p^r}  \ar[d] \ar[dl]_{\sim}   \ar[dr] & \\
	 \mathbb{R}\Gamma_{\rig}(\overline{X}_0/\so_F^0)^{\varphi = p^r}  \ar[d]_{\sim}   &  \mathbb{R}\Gamma_{\rig}(\overline{X}_0/\mathcal{T})^{\varphi = p^r}  \ar[d]  \ar[l] \ar[r]&  \mathbb{R}\Gamma_{\rig}(\overline{X}_0/\so_K^\pi) \ar[d]_{\sim}& \\
	 \mathbb{R}\Gamma_{\conv}(\overline{X}_0/\so_F^0)^{\varphi = p^r}  \ar[d]_{\sim}   &  \mathbb{R}\Gamma_{\conv}(\overline{X}_0/\widehat{\mathcal{T}})^{\varphi = p^r}  \ar[l] \ar[r] \ar[d]  &  \mathbb{R}\Gamma_{\conv}(\overline{X}_0/\so_K^\pi) \ar[d]_{\sim} &  \mathbb{R}\Gamma_{\dr}(\ovx_K)   \ar[l]^{ \qquad \sim}_{ \qquad \mathrm{sp}}  \ar[ld]^{\sim}_{\quad\gamma}  \ar[lu]_{\mathrm{sp}}  \\
	 \mathbb{R}\Gamma_{\hk}^\cris(\overline{X})_{\Q}^{\varphi = p^r}   & \mathbb{R}\Gamma_{\cris}(\overline{X}/\mathcal{T}_{PD})_{\Q}^{\varphi = p^r}  \ar[l]_{\sim}  \ar[r] & \mathbb{R}\Gamma_{\cris}(\overline{X}/\so_K^\pi)_{\Q}&
	}
	}$$}
where we  restricted the cohomology theories in the left and middle vertical row to Frobenius eigenspaces.
From the commutativity of the triangles on the right hand side of the diagram we conclude that the canonical morphism $\mathrm{sp}\colon  \mathbb{R}\Gamma_{\dr}(\ovx_K) \rightarrow   \mathbb{R}\Gamma_{\rig}(\overline{X}_0/\so_K^\pi)$ and hence $\mathrm{sp}\colon  \mathbb{R}\Gamma_{\dr}^{\D}(X_K) \rightarrow   \mathbb{R}\Gamma_{\rig}(X_0/\so_K^\pi)$  constructed in \S\ref{subsec: syntomic def} are quasi-isomorphisms.

This shows as desired  that the rigid and crystalline Hyodo--Kato morphisms are compatible on the Frobenius eigenspaces for eigenvalues $p^r$ with $r\geqslant -1$.
\end{proof}

As a consequence, we obtain a quasi-isomorphism between the diagrams which define $ \mathbb{R}\Gamma_\syn^\rig((X,\ovx),r,\pi)$ and $ \mathbb{R}\Gamma^\cris_\syn(\ovx,r,\pi)$ for $r\geqslant 0$. 
Since our construction of the cup product on $ \mathbb{R}\Gamma^\rig_\syn((X_0,\ovx_0),r,\pi)$ follows the construction of \cite[\S 2.4]{Besser2016}, this quasi-isomorphism is compatible with cup products.

\hrule
\vspace{.5cm}
\noindent
\verb=veronika.ertl@mathematik.uni-regensburg.de=\hfill \verb=kazuki72@a3.keio.jp=  \\

\end{document}